\documentclass[reqno]{amsart}

\usepackage{amsmath}
\usepackage{amsfonts}
\usepackage{rotating}
\usepackage{bbm}
\usepackage[all,cmtip]{xy}
\usepackage{amscd}
 
\usepackage{tikz}
\usetikzlibrary{arrows}

\usepackage{mathrsfs}
\usepackage{color}
\usepackage{xcolor}
\usepackage{mathtools}
\usepackage{amssymb}
\usepackage[latin1]{inputenc}
\usepackage{graphicx}
\usepackage{dsfont}
\usepackage{color}
\usepackage{hyperref}

\usepackage{marginnote}

\usepackage{comment}

\newtheorem{theorem}{Theorem}[section]

\newtheorem{lemma}[theorem]{Lemma}
\newtheorem{proposition}[theorem]{Proposition}
\newtheorem{corollary}[theorem]{Corollary}

\theoremstyle{definition}
\newtheorem{definition}[theorem]{Definition}

\newtheorem{example}[theorem]{Example}

\theoremstyle{remark}
\newtheorem{remark}[theorem]{Remark}

\numberwithin{equation}{section}

\newcommand{\p}{\partial}

\newcommand{\D}{{\mathbb{D}}}

\newcommand{\C}{{\mathbb{C}}}

\newcommand{\R}{{\mathbb{R}}}

\newcommand{\Z}{{\mathbb{Z}}}
\newcommand{\N}{{\mathbb{N}}}

 \newcommand{\im}{\mathrm{im}}
 \newcommand{\wind}{\mathrm{wind}}
 
 \newcommand{\muRS}{\mu_{\mathrm{RS}}}
 \newcommand{\muCZ}{\mu_{\mathrm{CZ}}}

\renewcommand{\epsilon}{\varepsilon}
\renewcommand{\theta}{\vartheta}

\usepackage{graphicx}

\begin{document}
\title[]{Symmetric periodic orbits and invariant disk-like global surfaces of section on the three-sphere}
\author{Seongchan Kim}
 \address{Institut de Math\'ematiques, Universit\'e de Neuch\^atel, Rue Emile-Argand 11, 2000 Neuch\^atel, Switzerland}
 \email {seongchan.kim@unine.ch}

\date{Recent modification; \today}
\setcounter{tocdepth}{2}
\maketitle

\begin{abstract}
 We study Reeb dynamics on the three-sphere equipped with a tight contact form and an anti-contact involution. We prove the existence of a symmetric periodic orbit and provide necessary and sufficient conditions for it  to bound an invariant disk-like global surface of section. We also study the same questions under the presence of    additional  symmetry and    obtain similar results in this case. The proofs make use of  pseudoholomorphic curves in symplectizations. As applications, we study   Birkhoff's conjecture on disk-like global surfaces of section in the planar circular restricted three-body problem and  the existence of symmetric closed Finsler geodesics on the two-sphere.  We also present applications  to some classical Hamiltonian systems.

\end{abstract}

\tableofcontents

\newpage 

\section{Introduction} \label{sec:intro}

The goal of this article is to study   global dynamical properties of  Reeb flows on the tight three-sphere equipped with a symmetry.  Recall that a \emph{contact three-manifold} is a three-manifold $M$ equipped with a maximally non-integrable hyperplane distribution $\xi$, called a \emph{contact structure}. If $\xi$ is co-oriented, then there exists a global one-form $\lambda$ on $M$, called a \emph{contact form} (defining $\xi$),  such that $\xi = \ker  \lambda $.  Note that if $\lambda$ is a contact form, then $\lambda \wedge d\lambda$ is non-vanishing and that for every smooth non-vanishing function $f$ on $M$, the one-form $f\lambda$ is  a contact form defining the same contact structure as $\lambda$. By abuse of terminology, we also call the pair $(M, \lambda)$ a contact manifold. We are interested in the dynamics of the \emph{Reeb vector field} $R = R_{\lambda}$ of $\lambda$   uniquely characterized by the equations
\[
\lambda(R) = 1 \quad \text{ and } \quad \iota_R d \lambda = 0.
\]
A periodic  Reeb  orbit will be denoted by $P=(x,T)$, where   $ x \colon \R \to M$ solves the differential equation $\dot{x} = R \circ x $ and $T>0$ is a period. It is said to be  \emph{simply covered} if $T$ is the minimal period, namely, if $k \in \N$ is such that $T / k$ is a period of $x$, then $k=1$. It is called \emph{unknotted} if the trace $x(\R)$ bounds an embedded closed disk.  In the sequel, by a periodic orbit we always mean a periodic Reeb orbit.   Since the Reeb flow $\phi_R^t$ satisfies $(\phi_R^t)^* \lambda = \lambda, $  it preserves the splitting $TM = \left< R \right> \oplus \xi$ so that we have  the  transverse linearized Reeb flow  
\[
d\phi_R^t(z) |_{\xi_z} \colon \xi_z \to \xi_{\phi_R^t(z)}, \quad z \in M.
\]
A periodic orbit $P=(x,T)$ is called non-degenerate if $1$ is not an eigenvalue of  $d\phi_R^T( x(0)) |_{\xi_{x(0)}}$. This does not depend on the parametrization. If every periodic orbit is non-degenerate, then we say that $\lambda$ is non-degenerate.

A     smooth involution $\rho$ defined on a contact manifold $(M, \lambda)$ is said to be   \emph{anti-contact} if $\rho^* \lambda =  - \lambda$.  In this case, the triple $(M, \lambda, \rho)$ is called a \emph{real contact manifold}. Since $\dim M = 3$, every anti-contact involution  is orientation-preserving. We assume that the fixed point set $\mathrm{Fix}(\rho)$ is non-empty.   The Reeb vector field $R $ satisfies  $\rho^* R = -R$ from which we find      $\phi_R^t = \rho \circ \phi_R^{-t} \circ \rho$. 
A \emph{chord} $C=(c,T/2)$ is an integral curve $c \colon [0, T/2] \to M$ of $R$ with   boundary condition $c(0), c(T/2) \in \mathrm{Fix}(\rho)$.  Closing it up via  $\rho$  provides us with one of the main characters of this article,   a \emph{symmetric periodic orbit} $P=(x,T),$ defined as
\[
x \colon [0,T] \to M, \quad \quad x(t):= \begin{cases} c(t) & \text{ if }  t \in [0,T/2], \\ \overline{c}(t):= \rho \circ c(T-t) & \text{ if } t \in [T/2, T]. \end{cases}
\]


The other main character is an \emph{invariant global surface of section}. Recall that a global surface of section for $R$   is an embedded compact surface $\Sigma \subset M$ such that each connected component of the boundary $\p \Sigma$, called the \emph{spanning orbits},  is a periodic   orbit,    the interior $\mathring {\Sigma}$ of $\Sigma$ is transverse to $R$, and all   orbits of $R$  in $M \setminus \p \Sigma$ hit $\mathring{\Sigma}$ infinitely often forward and backward in time. When   $M$ is equipped with an anti-contact involution $\rho$, we call a global surface of section $\Sigma$ \emph{invariant} if $\rho(\Sigma) = \Sigma$.  In what follows, we only consider the case $\Sigma=\mathfrak{D}$ is a disk. It is obvious that the boundary is a symmetric periodic orbit. A result of {K}er\'{e}kj\'{a}rt\'{o}, see for instance  \cite[Proposition 3.3]{CK94involution}, shows that $\mathrm{Fix}(\rho) \cap \mathfrak{D}$ is a simple arc such that the two connected components of $\mathfrak{D} \setminus \mathrm{Fix}(\rho)$ are permuted by $\rho$. The presence of an invariant  disk-like  global surface of section $\mathfrak{D}$ enables us to reduce the Reeb dynamics on $M$ to the study of an area-preserving homeomorphism of $\mathfrak{D}$ satisfying a certain property: Abbreviate by $\psi \colon \mathring{\mathfrak{D}} \to \mathring{\mathfrak{D}} $ the associated first return map defined by $\psi(x) = \phi_R^{\tau(x)}(x)$, where
\[
\tau(x) := \min \{ t >0 \mid \phi_R^t(x) \in \mathring{\mathfrak{D}} \}
\]
is the first return time of $ x \in \mathring{\mathfrak{D}}$.  It   is reversible with respect to $\rho$ in the sense that $\psi \circ \rho |_{\mathring{\mathfrak{D}}} \circ \psi = \rho |_{\mathring{\mathfrak{D}}}$. It then follows from \cite[Proposition 5.3]{Kang2} that there exists  $x \in \mathring{\mathfrak{D}}$ such that $\psi(x) = x = \rho(x)$, called a \emph{symmetric fixed point}, corresponding to  a symmetric periodic orbit distinct from $ \p \mathfrak{D}$. See also \cite[Section 9.7]{FvK18book}.  By \cite[Theorem 1.1]{Kang2} there exist  either two or infinitely many symmetric periodic orbits. Even the existence of a non-symmetric periodic orbit  forces the existence of infinitely many symmetric periodic orbits.

\begin{example}\label{exam:hopf}    Recall that the Reeb flow of the standard contact form   on $S^3 =\{ (z_1, z_2) \in \C^2 \mid \lvert z_1 \rvert^2 + \lvert z_2 \rvert^2 =1 \}$ is given by $\phi_R^t(z_1, z_2) = (e^{ 2 it}z_1, e^{2 i t}z_2)$ and hence defines the Hopf fibration.  The associated first return map is the identity map, and the first return time is  the constant function $\pi$.    The fiber $S^1 \times \{ 0 \} \subset S^3$ binds an open book decomposition of $S^3$ whose pages are the disks 
 \[
 \mathfrak{D}_{\theta}=  \left\{ \frac{(z, e^{ i \theta})}{ \sqrt{ \lvert z\rvert^2 +1 }}   \; \;  \bigg|  \; \; z \in \C \right\} \subset S^3 , \quad e^{ i \theta} \in S^1.
   \]
 Each $\overline{\mathfrak{D}_{\theta}}$ is a disk-like global surface of section. 
 Consider the anti-contact involution $\rho (z_1, z_2) = (\overline{z_1}, - \overline{z_2})$.  The fiber $S^1 \times \{ 0 \}$ is a $\rho$-symmetric periodic orbit, and the open book decomposition is $\rho$-symmetric in the sense that $\rho( \mathfrak{D}_{\theta}) = \mathfrak{D}_{ \pi - \theta}$. 
 Among the disks, only $\mathcal{D}_{\pi/2}$ and $\mathcal{D}_{3\pi/2}$ are $\rho $-invariant.    
\end{example}

A  contact structure $\xi  $ is said to be \emph{overtwisted} if there exists an overtwisted disk, i.e.\ an embedded disk $D$   such that the boundary $\p {D}$ is a Legendrian curve and $T_q {D} \neq \xi_q$ for all $ q \in \p {D}$. If there exists no overtwisted disk, then $\xi$ is called \emph{tight}.  Assume that $M=S^3$. Since $\rho$ is orientation-preserving,   a result of  Smith \cite{Smi38} tells us that  $\mathrm{Fix}(\rho) \cong S^1$, provided that $\mathrm{Fix}(\rho) \neq \emptyset$.    In view of   results by Bennequin  and Eliashberg \cite{Ben83, Eli92},  there exists a unique, up to   diffeomorphism, tight contact structure $\xi_0$ on $S^3$ which is determined by a contact form  $\lambda = f \lambda_0 |_{S^3}$, where $f$ is a smooth non-vanishing function on $S^3$ and $\lambda_0 :=(1/2) (q_1 dp_1 - p_1 dq_1 + q_2 dp_2 - p_2 dq_2)$ is the Liouville form on $\R^4$. We  call the pair $(S^3, \xi_0)$ the \emph{tight three-sphere}. A contact form defining $\xi_0$ is called a \emph{tight contact form}. Let $\lambda = f\lambda_0$ be a tight contact form. Then a smooth involution $\rho$ is anti-contact if and only if $f \circ \rho = f$ and $\rho ^* \lambda_0 = -\lambda_0$. The triple $(S^3, \lambda, \rho)$ is called a \emph{real tight three-sphere} (with respect to $\rho$).  For example, a compact hypersurface in $\R^4$ which is star-shaped with respect to the origin and invariant under complex conjugation is a real tight three-sphere.  We say that a real tight three-sphere $(S^3, \lambda, \rho)$ is non-degenerate if $\lambda$ is non-degenerate.

The following statement is proved in Section \ref{se:profthemBF}. It follows easily from automatic transversality of embedded fast finite energy planes due to Hryniewicz \cite{Hry12fast} and the existence result of (not necessarily invariant) disk-like global surfaces of section due to Hryniewicz and Salom\~ao \cite{HS11}. We remark that     not every Reeb flow on the tight three-sphere admits a disk-like global surface of section, see  \cite{vK19}.

\begin{theorem}\label{ThmA1}  Assume that a real tight three-sphere $(S^3, \lambda, \rho)$ is non-degenerate.  A simply covered symmetric periodic orbit $P=(x,T)$ bounds an invariant disk-like global surface of section if and only if it is unknotted,   has self-linking number equal to $-1$ and   Conley-Zehnder index at least $3$,  and     $P$ is linked to every periodic orbit of  Conley-Zehnder index equal to $2$. 
 In this case, $P$ bounds two invariant disk-like global surfaces of section, and they form an embedded two-sphere containing $\mathrm{Fix}(\rho) \cong S^1$. 
 If $\lambda$ is dynamically convex, i.e.\ every periodic orbit has   Conley-Zehnder index at least $3$, the non-degeneracy assumption can be dropped. 
 \end{theorem}

  \medskip
 
\noindent
We say that a periodic orbit $\overline{P}=( \overline{x}, \overline{T})$ is  \emph{geometrically distinct from $P=(x,T)$} if $x(\R) \cap  \overline{x} (\R) = \emptyset$. In this case, $\overline{P}$ determines  a homology class $[\overline{x}] \in H_1(S^3 \setminus x(\R); \Z) \cong \Z$. If $[\overline{x}] \neq 0$, then $\overline{P} $ is said to be \emph{linked to $P$}. 
For definition of the self-linking number, see Subsection \ref{sec:self}.

\begin{remark} \label{rmk:selfadd} The two global surfaces of section $\mathfrak{D}_1$ and $\mathfrak{D}_2$ obtained in the above theorem are related by
\[
\mathring{\mathfrak{D}}_2 = \{ \phi_R^{\tau(x) /2}(x) \mid x \in \mathring{\mathfrak{D}}_1 \}.
\]
Therefore, if $x \in  \mathring{\mathfrak{D}}_1$ is a symmetric fixed point of the associated first return map,  then $\phi_R^{\tau(x)/2}(x)  \in  \mathring{\mathfrak{D}}_2$ is a symmetric fixed point of the first return map associated to $\mathfrak{D}_2$.  Thus, the two points belong to  a single symmetric periodic orbit geometrically distinct from $P = \p \mathfrak{D}_1 = \p \mathfrak{D}_2$. In view of a result of Hryniewicz \cite{Hry14system}, it is also unknotted and has self-linking number equal to $-1$.

\end{remark}

 Consider the family of   anti-contact involutions 
 \begin{equation}\label{antis}
  (z_1, z_2) \mapsto  ( e^{ i \theta_1}\overline{z _1}, e^{i \theta_2}\overline{z_2} )
  \end{equation}
   on $S^3$ for $\theta_1, \theta_2 \in \R$.    Abbreviate by $\mathscr{P}_*$  the set of unknotted  simply covered non-symmetric periodic orbits of self-linking number equal to $-1$ and Conley-Zehnder index equal to $2$.

\begin{theorem}\label{ThmA2} Consider   a non-degenerate real tight three-sphere $(S^3, \lambda, \rho )$ such that $\rho$ has the form \eqref{antis}. Assume that  $\mathscr{P}_* = \emptyset$. Then there exists    an unknotted simply covered symmetric periodic orbit $P$ of self-linking number equal to $-1$ and having Conley-Zehnder index in  $\{ 2,3,4 \}$. In the dynamically convex case,   one can drop the non-degeneracy assumption, and the Conley-Zehnder index of $P$ belongs to $\{3,4\}.$  
\end{theorem}

\noindent
The reason why we only consider the anti-contact involutions $\rho$ of the form \eqref{antis}  is that  in the proof we interpolate between the  Reeb dynamics of $\lambda$ and the ``simplest" one,  where  the latter carries exactly two unknotted simply covered $\rho$-symmetric periodic orbits of self-linking number equal to $-1$ and Conley-Zehnder index $\geq 3$, see Section \ref{sec:thmA}. The hypothesis that $\mathscr{P}_* = \emptyset$ is used in the bubbling off analysis. The existence of a symmetric periodic orbit (not assuming any other properties) is related to a conjecture of Seifert,  see Remark \ref{rmk:Seifert} below.

\begin{remark}\label{rmk:jypothes}  As mentioned above, the proof of  Theorem \ref{ThmA1} makes use of  a result due to Hryniewicz and Salom\~ao saying that a simply covered periodic orbit  $P$  on a non-degenerate tight three-sphere bounds a disk-like global surface of section if and only if it is unknotted,   has self-linking number equal to $-1$ and Conley-Zehnder index $\geq 3$,   and is linked to   every periodic orbit of Conley-Zehnder index equal to $2$. In their proof, in fact they only make use of the assumption that $P$ is linked to every periodic orbit   which is the asymptotic limit of an embedded finite energy plane of  index equal to $2$. See \cite{HS11} for more details. By means of \cite[Theorem 1.7]{HWZ96unknotted} such periodic orbits have self-linking number equal to $-1$. Therefore, the assertion of Theorem \ref{ThmA1} holds as well if we replace the last condition   by saying   that  the set of unknotted  simply covered   periodic orbits of self-linking number equal to $-1$ and Conley-Zehnder index equal to $2$ is empty. 
\end{remark}

As an immediate corollary of Theorems \ref{ThmA1} and \ref{ThmA2},   Remark \ref{rmk:jypothes}, and  \cite[Corollary 1.2]{Kang2}  we obtain the following assertion which is a refinement of the results of Frauenfelder and Kang \cite{FK16}, where a contact form is assumed to be dynamically convex.

\begin{theorem}\label{ThmA3}
 Assume that  a non-degenerate real tight three-sphere $(S^3, \lambda, \rho ),$ where $\rho$ has the form \eqref{antis},  satisfies that there exists no  unknotted  simply covered   periodic orbit  of self-linking number equal to $-1$ and Conley-Zehnder index equal to $2$.   Then   there exists a symmetric periodic orbit which bounds   an invariant disk-like global surface of section  and has   Conley-Zehnder index equal to  3 or 4.  As a result, there exist either two or infinitely many symmetric periodic orbits. The existence of a non-symmetric periodic orbit  forces the existence of infinitely many symmetric periodic orbits. If $\lambda$ is dynamically convex, then the same assertions hold without the non-degeneracy assumption.
 \end{theorem}

 \begin{remark}\label{rmk:Seifert} In \cite{Sei48} Seifert showed that  in a mechanical Hamiltonian system defined on $\R^{2n}$, a regular energy level satisfying certain conditions carries a symmetric periodic orbit with respect to complex conjugation. He conjectured the existence of at least $n$ geometrically distinct  symmetric periodic orbits. The case  $n=2$ is proved recently by     Giamb\`o,   Giannoni,  and Piccione \cite{GGP15}.  If an energy level  in $\R^4$  is star-shaped and the associated contact form  is dynamically convex, then a stronger statement holds: there exist at least two unknotted symmetric periodic orbits of self-linking number equal to $-1$, see \cite[Theorems 2.3 and 2.6]{FK16}. In view of Remark \ref{rmk:selfadd},  Theorem \ref{ThmA3} refines this statement.  
 \end{remark}

 Suppose that a real  tight three-sphere $(S^3, \lambda, \rho)$ is endowed with  a diffeomorphism $\sigma$ such that $\sigma^*\lambda =\lambda$ and $\sigma^p = \mathrm{Id}$ for  some $p \in \N$.    
Assume that 
\begin{equation}\label{eq:action}
\sigma \circ \rho \circ \sigma = \rho. 
\end{equation}
For each $j=0, 1, \ldots, p-1$, the smooth map $\rho_j :=  \sigma^{ j} \circ \rho  $ is an anti-contact involution with $\rho_0 = \rho$.   A periodic orbit $P=(x,T)$ is said to be \emph{$(\rho, \sigma)$-symmetric} if $\rho_j ( x(\R)) = x(\R)$ for all $j$.   A disk-like global surface of section $\mathfrak{D}$ is called \emph{$ (\rho, \sigma)$-invariant} if $\rho_j (\mathfrak{D} )  = \mathfrak{D} $ for all $j$. The boundary $\p \mathfrak{D}$ is a $(\rho, \sigma)$-symmetric periodic orbit. Since $\rho_0
|_{\mathfrak{D}}$ and $   \rho_1 |_{\mathfrak{D}}$ are area-preserving,   $\mathrm{Fix}(\rho_0|_{\mathfrak{D}} ) \cap \mathrm{Fix}(\rho_1|_{\mathfrak{D}})  \neq \emptyset$, i.e.\ there exists $ z_* \in \mathring{\mathfrak{D}}$ such that $\rho_0(z) = z= \rho_1(z)$. By definition of $\rho_j$, we   find that $\sigma(z_*) = z_*$ (and $\rho_j(z_*) = z_*$ for all $j$).  Therefore, the presence of a common fixed point of $\rho$ and $\sigma$  is a necessary  condition for  a  $ (\rho, \sigma) $-invariant global surface of section to exist.

Consider the special case that $p=2$. In this case $\sigma$ is   a contact involution, meaning that it is an   involution such that $\sigma^* \lambda= \lambda.$   Assumption \eqref{eq:action} now reads   $\rho \circ \sigma = \sigma \circ \rho$. The two commuting anti-contact involutions $\rho$ and $\rho_{\sigma} := \sigma \circ \rho$      generate   dihedral symmetry on $S^3$. A $(\rho, \sigma)$-symmetric periodic orbit or a $(\rho, \sigma)$-invariant global surface of section is also referred to as a \emph{doubly-symmetric  periodic orbit} or a \emph{doubly-invariant  global surface of section}, respectively. If $P=(x,T)$ is  a doubly-symmetric periodic orbit, then  $c  = x|_{[0,T/4]}$ is the  chord  with  boundary condition $c(0) \in \mathrm{Fix}(\rho )$ and $c(T/4) \in \mathrm{Fix}(\rho_{\sigma})$  and    recovers $x$ in such a way that

\[
x(t) = \begin{cases}  c(t) & \text{ if } t \in [0,T/4],  \\      \rho_{\sigma} \circ c(T/2-t) & \text{ if }  t \in [T/4, T/2], \\     \rho \circ \rho_{\sigma} \circ c(t-T/2) & \text{ if } t \in [T/2, 3T/4], \\ \rho \circ c(T-t) & \text{ if } t \in [3T/4, T].    \end{cases}
\]
As before, $\mathrm{Fix}(\rho) \cap \mathrm{Fix}(\rho_{\sigma}) \neq \emptyset$ is a necessary condition to the existence of a doubly-invariant disk-like global surface of section.

\begin{example} [{continued}] Fix $p \in \N$ and consider $\sigma (z_1, z_2)=   ( e^{2 \pi i / p} z_1, z_2)$.  Note that  $\rho_j (z_1, z_2) = (e^{ 2 \pi i j /p}\overline{z_1}, - \overline{z_2})$ and that $\mathrm{Fix}(\rho ) \cap \mathrm{Fix}(\sigma) = \{ (0,0,0, \pm 1) \}.$
The fiber $S^1 \times \{ 0 \}$  is a  $(\rho, \sigma)$-symmetric periodic orbit, and the given open book decomposition is $(\rho, \sigma)$-symmetric in the sense that $\rho_j ( \mathfrak{D}_{\theta}) = \mathfrak{D}_{\pi - \theta}$ for $j=0,1,\ldots, p-1$. The two disks $\mathfrak{D}_{\pi/2}$ and $\mathfrak{D}_{3\pi/2}$ are the only  $(\rho, \sigma)$-invariant  disk-like global surfaces of section, and they contain  $(0,0,0,1)$ and $(0,0,0,-1)$, respectively. 
 
  \end{example}

    The following assertion is proved in  Section \ref{se:profthemBF} in the same way as Theorem \ref{ThmA1}.

\begin{theorem}\label{ThmB1}  Assume that a non-degenerate real tight three-sphere $(S^3, \lambda, \rho)$ admits a $p$-periodic  strict contactomorphism $\sigma$ satisfying condition \eqref{eq:action}.   Assume further that  $\mathrm{Fix}(\rho) \cap \mathrm{Fix}(\sigma) \neq \emptyset$. 
A simply covered $(\rho, \sigma)$-symmetric periodic orbit $P=(x,T)$ bounds a  $(\rho, \sigma)$-invariant disk-like global surface of section if and only if it is unknotted,   has self-linking number equal to $-1$ and   Conley-Zehnder index greater than or equal to $3$,  and  $P$ is linked to every periodic orbit of Conley-Zehnder index equal to $2$. 
Moreover, $\# \left( \mathrm{Fix}(\rho ) \cap \mathrm{Fix}(\sigma) \right)= 2$, and  $P$ bounds   two $(\rho, \sigma)$-invariant disk-like global surfaces of section containing either point  in $  \mathrm{Fix}(\rho ) \cap \mathrm{Fix}(\sigma).$ If $\lambda$ is dynamically convex, the non-degeneracy assumption can be dropped.
\end{theorem}

The most interesting $p$-periodic contactomorphisms on $S^3$ might be the following:    Given coprime integers $p \geq q \geq1$, the three-sphere $S^3$ admits the  $p$-periodic diffeomorphism
  \[
  g_{p,q} \colon S^3 \to S^3, \quad (z_1, z_2) \mapsto \left( e^{ 2 \pi i /p}z_1, e^{ 2 \pi i q/ p} z_2 \right).
  \]
  Each   $g_{p,q}$ satisfies condition \eqref{eq:action} with respect to the anti-contact involutions of the form    \eqref{antis}, and $g_{p,q}^* \lambda_0 = \lambda_0$. Moreover,  $g_{p,q}$  generates a free $\Z_p$-action on $S^3$, and the orbit space $L(p,q)=S^3/\Z_p$ is a lens space. Note that $L(1,1)=S^3$ and $L(2,1) \cong \R P^3$. Recall that
   \[
   \pi_1( L(p,q)) \cong \Z_p \quad \text{ and } \quad \pi_2( L(p,q)) = \{ 0 \}.
   \]
We fix $p  \geq q \geq 1$ and assume that a tight contact form $\lambda$ on $S^3$ satisfies   $g_{p,q}^* \lambda = \lambda$. Thus,  it descends to a contact form $\overline{\lambda}$ on $L(p,q)$ whose associated contact structure is denoted by $\overline{\xi_0}$.  It is \emph{universally tight}, meaning that its lift to the universal covering is tight.  In view of condition \eqref{eq:action} the anti-contact involution $\rho$ descends to an anti-contact involution $\overline{\rho}$ on $(L(p,q), \overline{\lambda})$. We call the triple $(L(p,q), \overline{\lambda}, \overline{\rho})$ a \emph{real universally tight lens space}.

Since $g_{p,q}$ and $\rho$ have no common fixed point on $S^3$, there does not exist a $(\rho, g_{p,q})$-invariant disk-like global surface of section. Instead, we study   a $\overline{\rho}$-invariant $p$-rational disk-like global surface of section on $(L(p,q), \overline{\lambda}, \overline{\rho})$. Recall that a periodic orbit $P=(x,T)$ is said to be \emph{$p$-unknotted} if there exists an immersion $u \colon \D \to L(p,q)$ such that $u|_{ \D \setminus \p \D} $ is an embedding, and $u|_{\p \D} \colon \p \D \to x(\R)$ is a $p$-covering map. In this case, $u$ is called a \emph{$p$-disk for $P$}. For a $p$-unknotted periodic orbit, there exists a well-defined (rational) self-linking number, see Subsection \ref{sec:self}. Since $\pi_2( L(p,q)) = \{ 0 \}$, it does not depend on the choice of a $p$-disk. 
 By a \emph{$p$-rational disk-like global surface of section}, we mean   a $p$-disk  $u \colon \D \to L(p,q) $  such that     $u(\D \setminus \p \D)$  satisfies the same properties as the interior of a disk-like  global surface of section.  

\begin{theorem}\label{Thmnew} Let $(L(p,q), \overline{\lambda}, \overline{\rho})$ be as above. Assume that $\overline{\lambda}$ is dynamically convex, i.e.\ every contractible periodic orbit has  Conley-Zehnder index at least $3$. Then a simply covered symmetric periodic orbit $P$ bounds an invariant $p$-rational disk-like global surface of section if and only if it is $p$-unknotted and has self-linking number equal to $-1/p$. 
\end{theorem}

 The existence of a simply covered $p$-unknotted periodic orbit with self-linking number equal to $-1/p$  on $(L(p,1), \overline{\lambda})$ is proved by Schneider \cite{Sch19}. The statement is left open for $q >1$. Following Schneider's argument, we prove

\begin{theorem}\label{ThmB2} Let $(L(p,1), \overline{\lambda}, \overline{\rho})$ be a dynamically convex real universally tight lens space, where $\overline{\rho}$ is the restriction of an anti-contact involution on $S^3$ of the form \eqref{antis} with either $\theta_1=0$ or $\theta_2=0$.  Then   there exists    a simply covered symmetric periodic orbit which is $p$-unknotted and has self-linking number equal to $-1/p$. Its $p$-th iterate has   Conley-Zehnder index equal to  $3$ or $4$. 
\end{theorem}
\noindent 
The condition either $\theta_1 =0$ or $\theta_2=0$ is used in the bubbling off analysis in the proof, see Lemma \ref{lemindxc}.

 Suppose that a real tight three-sphere $(S^3, \lambda, \rho)$ admits $g_{p,1}$ for some $p$. We impose the following equivalence relation on the set of periodic orbits. Two periodic orbits $P_1$ and $P_2$ are equivalent if and only if each of them is a $\rho_j$-symmetric periodic orbit for some $j$ and $P_2 = g_{p,1}^i (P_1)$ for some $i$.  Therefore,   an equivalence class consists of either a single $(\rho, \sigma)$-symmetric periodic orbit or $p$ periodic orbits.        By the preceding two theorems and results due to Kang \cite{Kang2} we obtain the following statement as a corollary.

 \begin{theorem}\label{ThmB3}
 Let $(S^3, \lambda, \rho)$ be a dynamically convex real tight three-sphere such that $\rho$ is of the form \eqref{antis} with either $\theta_1=0$ or $\theta_2=0$. Assume that $\lambda$ satisfies $g_{p,1}^* \lambda = \lambda$ for some $p \geq 1$. Then there exists  a     $(\rho, g_{p,1})$-symmetric periodic orbit  on $S^3$ which  bounds   a  $\rho$-invariant   disk-like global surface  of section.    It is   the lift of a  $\overline{\rho}$-symmetric periodic orbit on $L(p,1)$ that bounds a    $\overline{\rho}$-invariant $p$-rational disk-like global surface of section.  As a result, there exist two or infinitely  many    equivalence classes of periodic orbits on $S^3$.  In particular,   in the case   $p=2$  there exist two or infinitely many doubly-symmetric periodic orbits.
 \end{theorem}

\noindent
{\bf Further discussion.}  A symmetric periodic orbit is said to be \emph{symmetrically unknotted} if there exists a spanning disk which is invariant under the symmetry.  The following question naturally arises.

\begin{quote}
  \emph{Is an unknotted symmetric periodic orbit     symmetrically unknotted?}   
 \end{quote}

\noindent
  Theorem \ref{ThmA1} shows that if an unknotted symmetric periodic orbit $P$ has $sl(P) =-1$ and $\muCZ(P) \geq 3$ and satisfies the described linking condition, then the answer is affirmative, provided that the contact form is non-degenerate. In particular, if the  contact form is dynamically convex, then every unknotted symmetric periodic orbit with self-linking number equal to $-1$ is symmetrically unknotted.

 \bigskip

\noindent
\textbf{Acknowledgements:}  This work has its origin in inspiring discussion with Umberto Hryniewicz.  I  cordially thank  Naiara de Paulo, Urs Frauenfelder, Umberto Hryniewicz, Jungsoo Kang,  Pedro Salom\~ao, Felix Schlenk, and Otto van Koert for many helpful discussions. I also thank the anonymous referee for a number of valuable comments. This research was supported by   the grant 200021-181980/1 of the Swiss National Foundation.

\section{Applications}

\subsection{The   restricted three-body problem and Birkhoff's conjecture}\label{sec:3bp}

The planar circular restricted three-body problem (PCR3BP) studies the dynamics of a massless body, called the moon, attracted by two primaries, called the earth and the sun, according to Newton's law of gravitation. We scale the total mass of the system to one so that the mass of the sun is given by $\mu \in (0,1)$ and the mass of the earth equals $1-\mu$.  We assume that they move in circular orbits.  The moon is restricted to the plane spanned by the earth and the sun. Due to Jacobi, passing to a rotating coordinate system the Hamiltonian of the system becomes
\begin{align*}
&H : T^*\left( \C \setminus \left \{ E,S \right \} \right) \rightarrow \R \\
&H(q,p) = \frac{1}{2} |p|^2 - \frac{1-\mu}{|q -E|} - \frac{\mu}{|q-S |} + q_1 p_2 - q_2 p_1 -\mu p_2,
\end{align*}
where $E=(0,0)$ and $S=(1,0)$ are the positions of the earth and the sun, respectively.  This Hamiltonian is time-independent, and hence it is preserved along the Hamiltonian flow.

It is well-known that there are precisely three (if $\mu=1/2)$ or four (if $\mu \neq 1/2$) critical values of $H$. We are interested in   energies  $c$ smaller than the first (or smallest) critical level  $c_1$, where the corresponding energy hypersurface $H^{-1}(c)$ consists of three connected components.
When they are projected into the $q$-plane, two of them are bounded and   surround either primary.     We concentrate on the component $\Sigma_c \subset H^{-1}(c)$   whose projection into the $q$-plane surrounds the earth.  It is non-compact due to collisions with the earth. However, the two-body collision  can be always regularized. Using the Moser regularization \cite{Moser70} it was shown in \cite{AFvKG12} that the regularization  $\overline{\Sigma_c}$ of $\Sigma_c$ is contactomorphic to the universally tight real projective space $\R P^3 = L(2,1)$. It is invariant under the anti-contact involution   $\rho(q_1, q_2, p_1, p_2) = (q_1, -q_2, -p_1, p_2).$ Abbreviate $\mathcal{L}  :=   \text{Fix}(\rho)$ and   define $\mathcal{L}_1, \mathcal{L}_2 \subset \mathcal{L}$  such that  $\mathcal{L} = \mathcal{L}_1 \cup \mathcal{L}_2$ by 
$$
\pi(\mathcal{L}_1 )= \left \{  (q_1, 0) \in \pi( \overline{\Sigma_c}) :  q_1 \leq 0    \right \}, \quad  \pi(\mathcal{L}_2 )= \left \{  (q_1, 0) \in \pi( \overline{\Sigma_c}) :  q_1  \geq0    \right \},
$$
where    $\pi : T^* \C \rightarrow \C$ is the footpoint projection. Kang \cite{Kang14} classifies symmetric periodic orbits on $\overline{\Sigma_c}$ into the following two types. A symmetric periodic orbit $P= (x,T)$ is of type I if the associated half-chord $C=(c, T/2)$ satisfies   $c(0) \in \mathcal{L}_1$ and  $c(T/2) \in \mathcal{L}_2$, or the other way around.  Otherwise, it is of type II.

On the other hand, via  the Levi-Civita regularization \cite{Levi},   the regularized double cover $\widetilde{\Sigma_c}$ of $\Sigma_c$ is contactomorphic to the tight three-sphere. The anti-contact involution $\rho$ on $\overline{\Sigma_c}$ lifts to   the two commuting anti-contact involutions   on $\widetilde{\Sigma_c}$
\[
\rho_1 ( v_1, v_2, u_1, u_2)= (v_1, -v_2, -u_1, u_2)
\]
and
\[
  \rho_2 ( v_1, v_2, u_1, u_2)= (-v_1, v_2, u_1, -u_2)
  \]
    for $  (v_1, v_2, u_1, u_2) \in S^3 \subset \R^4 \equiv \C^2.$  The composition $\iota := \rho_1 \circ \rho_2$ is the antipodal map $(v_1, v_2, u_1, u_2) \mapsto (-v_1, -v_2, -u_1, -u_2)$.     
    The sets $\mathcal{L}_1$ and $\mathcal{L}_2$ lift to $\overline{\mathcal{L}_1} := \text{Fix}(\rho_1) $ and $\overline{\mathcal{L}_2} := \text{Fix}(\rho_2)$, respectively.  In this way, Kang observes that 
  symmetric periodic orbits of type I on $\overline{\Sigma_c}$ correspond to   doubly-symmetric periodic orbits   on $\widetilde{\Sigma_c}$ and those  of type II on $\overline{\Sigma_c}$ correspond to  symmetric, but not doubly-symmetric periodic orbits   on $\widetilde{\Sigma_c}$.

  In 1915,  Birkhoff proved the existence of  a  so-called   \emph{retrograde periodic orbit}  on each bounded component of the energy hypersurface for $c <c_1$,  see \cite{Birk15}. This is a symmetric periodic orbit of type I and   projects into the $q$-plane as a simple closed curve  encircling the corresponding primary in the   direction opposite to the    coordinate system.  A periodic orbit  projecting into the $q$-plane as a simple closed curve, but   encircling the corresponding primary in the  same direction as   the    coordinate system is called a \emph{direct periodic orbit}.  In  a real-world situation, a direct periodic orbit is more important since most orbits of moons in the solar system are direct. Note that the orbit of Triton, the largest moon of Neptune, is retrograde. Numerical experiments support the existence of a direct periodic orbit.  While Birkhoff did not give an analytic proof of its existence,  he pointed out that a direct periodic orbit is sometimes  degenerate. This led him to   conjecture that a (doubly covered)    retrograde periodic orbit always bounds a disk-like global surface of section.
 Birkhoff believed that a fixed point of  the associated first  return map, whose existence is assured by Brouwer's translation theorem,  corresponds to a (doubly covered) direct periodic orbit.

Albers, Fish, Frauenfelder, Hofer, and van Koert proved  in  \cite{AFFHvK12} that for every $c<-3/2$, there exists $\mu_0 = \mu_0(c) $ close enough to $1$ such that for each $\mu \in (\mu_0, 1)$,  $\widetilde{\Sigma_c}$ is strictly convex. In view of a result due to Hofer, Wysocki, and Zehnder \cite{HWZ98convex} this implies that Birkhoff's conjecture holds in this case.  In \cite{HS16elliptic} Hryniewicz and Salom\~ao studied if one can find a {rational global surface of section} in $\overline{\Sigma_c}$ whose spanning orbit is the simply covered retrograde periodic orbit.    Using  the result of   \cite{AFFHvK12}, they provide an affirmative answer if $\mu = 1- \epsilon$ for $\epsilon>0$ small enough. As the doubly covered retrograde periodic orbit is unknotted and has self-linking number equal to $-1$,   we have  the following refinement of the result of Hryniewicz and Salom\~ao.

 \begin{theorem}\label{thm:L3cp}
 Given $c<-3/2$, there exists $\mu_0 = \mu_0(c) $ close enough to $1$ such that for all $\mu \in (\mu_0, 1)$, the  doubly covered retrograde periodic orbit bounds  $\rho_1$-invariant disk-like global surfaces of section $\mathfrak{D}$ and $\iota(\mathfrak{D})$  in $\widetilde{\Sigma_c}$. They descend  to a  $\rho$-invariant rational disk-like global surface of section   in $\overline{\Sigma_c}$ spanned by the simply covered retrograde periodic orbit.    Moreover, there exist either two or infinitely many symmetric periodic orbits of type $I$ on $\overline{\Sigma_c}$. 
\end{theorem}

\subsection{Existence of    symmetric closed Finsler geodesics}

 
Let $N$ be a closed connected $n$-dimensional smooth manifold. Choose local coordinates on   $TN$,   $(q,v)=(q_1, \dots, q_n, v_1, \dots, v_n)$.   A \textit{Finsler metric} $F$ on  $N$ is a non-negative function   on $TN$ having the following properties:
\begin{itemize}
\item  $G:=F^2/2$ is smooth on $TN \setminus N$.
\item $F(q,v) = 0$ if and only if $v=0$.
\item  $F(q, \lambda v) = \lambda F(q,v)$ for all  $\lambda >0$.
\item  For each $ q \in N$ and $v \neq 0 \in T_q N$, the symmetric bilinear form 
\[
\left< \cdot, \cdot \right>_v : T_q N \times T_qN \rightarrow \R, \quad (u,w) \mapsto \frac{\p ^2 G}{\p s \p r}\bigg|_{s=r=0} (q, v + su +rw),
\]
is positive-definite.
\end{itemize}
The critical points of the Finsler energy functional
$$
\mathscr{E}(q) := \int_a^{b} G(q(t), \dot{q}(t) )dt, \quad q \colon [a,b]\to N 
$$ 
are called \emph{geodesics of the Finsler metric $F$} or simply \emph{$F$-geodesics}.  
Suppose that $N$ is equipped with an involution $f$. Abbreviate by $Q=\text{Fix}(f)$ the fixed point set of $f$, which is assumed to be non-empty.   We impose the following additional condition.
\begin{itemize}
\item ($f$-invariance) $F(f(q), -df(q) v) = F(q,v)$ for all $(q,v) \in TN\setminus N$.
\end{itemize}
We call such $F$ a \emph{real Finsler metric}  (with respect to $f$) and the triple $(N, F, f)$ a   \textit{real Finsler manifold}.

A closed  $F$-geodesic $q \colon [0, T] \to N$ is said to be \emph{symmetric} if $f( q(t)) = q(-t)$. It is straightforward to check that $q(t) \in Q$ and $-df(q(t))\dot{q}(t) =\dot{q}(t)$ for $t = 0, T/2$.   Define  the involution $\mathcal{F}$ on $TN \setminus N  $ as   $\mathcal{F}   (q,v) = (f(q), -df(q)v) $.  
The fixed point set $\mathrm{Fix}(\mathcal{F})$ consists of  those $(q,v) \in TN \setminus N$ such that $q \in Q$ and $v \in T_qN$ satisfy $-df(q)v = v$.  It follows that if $q \colon [0,T] \to N$ is a symmetric closed $F$-geodesic of speed $c>0$, then its tangent lift $\gamma =(q, \dot{q}) \colon [0,T] \to F^{-1}(c) $ is a closed curve such that  $\gamma(t) \in \mathrm{Fix}(\mathcal{F})$ for $ t = 0, T/2$.


There exists a globally defined one-form $\tau$ on $TN \setminus N$,  which is locally given by
$$
\tau   =  \sum_{j=1}^n\frac{ \partial G}{\partial v_j} dq_j .
$$
It defines  a contact form on $F^{-1}(c)$ for all $c>0$. An interesting feature of this form is that $d\tau$ is a symplectic form on $TN \setminus N$,  see  \cite[Proposition 2.1]{DGZ17}. Abbreviate  $\Sigma = F^{-1}(1) \subset TN$.   By abuse of notation we denote the restriction of $\tau$ to $\Sigma$  again by the same letter.  It was shown in \cite[Proposition 2.4]{DGZ17} that  unit speed $F$-geodesics on $N$ are precisely the projections of Reeb orbits  on $(\Sigma, \tau)$. Hence, we have the one-to-one correspondence
$$
\begin{Bmatrix} \text{unit speed symmetric closed} \\ \text{ geodesics on $(N,F,f)$} \end{Bmatrix} \longleftrightarrow \begin{Bmatrix} \text{symmetric periodic  orbits} \\ \text{on $( \Sigma , \tau, \mathcal{F})$}   \end{Bmatrix}.
$$

  Choose local coordinates on $T^*N$,   $(q,p) = (q_1, \ldots, q_n, p_1, \ldots, p_n)$. The fiberwise convexity of   $G$ gives rise to the Legendre transformation 
  \[
   \mathscr{L} : TN \setminus N \longrightarrow T^*N \setminus N, \quad (q,p) \mapsto  (q,p=\p G / \p v).
   \]
     It is straightforward to check that 
$$
 \mathscr{L}^* \lambda_{\mathrm{can}} = \tau,
$$
where $\lambda_{\mathrm{can}}$ is the Liouville 1-form on $T^*N$.  Abbreviate $\Sigma^*  = \mathscr{L}(\Sigma)\subset T^*N$ and $\lambda=\lambda_{\mathrm{can}}|_{\Sigma^*}$. As has been shown in \cite[Section 3]{DGZ17},   the  orbits of the Reeb vector field $R_{\tau}$ of $\tau$ on $\Sigma$ correspond  to the  orbits of the Reeb vector field $R_{\lambda} $  of $\lambda$ on $\Sigma^* $. 

Define the smooth involution $f_I : T^*N \setminus N \rightarrow T^*N \setminus N$ as $f_I  = \mathscr{L} \circ  \mathcal{F} \circ \mathscr{L}^{-1}$ satisfying $f_I ^* \lambda_{\mathrm{can}} = - \lambda_{\mathrm{can}}$. Note that
$$
f_I = d_* f \circ I = I \circ d_*f,
$$
where   $d_*f : T^*N \rightarrow \R$ is the cotangent lift of the involution $f$ and  $I(q,p) = (q,-p)$. It descends to an anti-contact involution on $(\Sigma^*, \lambda)$, denoted again by $f_I$.  As before,  there is the following one-to-one correspondence
$$
\begin{Bmatrix} \text{unit speed symmetric closed} \\ \text{   geodesics on $(N,F,f)$} \end{Bmatrix} \longleftrightarrow \begin{Bmatrix} \text{symmetric periodic  orbits} \\ \text{on $( \Sigma^* , \lambda, f_I  )$}   \end{Bmatrix}.
$$


Consider $\C^2$ with coordinates $(z_1, z_2)$, where $z_j = q_j + ip_j$ for $j=1,2$.   In \cite[Section 4]{HP08} Harris and Paternain constructed   a double covering  map $\Phi : S^3  \rightarrow \Sigma^*$ which satisfies 
$$
\Phi^* \lambda = h \lambda_0|_{S^3},
$$
where   the smooth map  $ h =h(F)  :  S^3 \rightarrow (0, \infty) $
is invariant under the antipodal map, i.e.  $h(z_1, z_2) = h(-z_1, -z_2)$. This in particular implies that the double cover of the geodesic flow of $(S^2, F)$ is conjugate to the Reeb flow   on $S^3$ of the   tight contact form $h\lambda_0$.  By pulling back the involution $f_I$ under $\Phi$, we obtain the two anti-contact involutions $\rho_1, \rho_2$ on $(S^3,h\lambda_0)$ such that $\rho_1 \circ \rho_2 = \rho_2 \circ \rho_1 = \iota$, where $\iota$ is the antipodal map on $S^3$. Thus,   we obtain the following one-to-one correspondence
$$
\begin{Bmatrix} \text{unit speed symmetric closed} \\ \text{   geodesics on $(N,F,f)$} \end{Bmatrix} \longleftrightarrow \begin{Bmatrix} \text{doubly-symmetric periodic   orbits} \\ \text{on $( S^3, h\lambda_0, \rho_1, \rho_2 )$}   \end{Bmatrix}.
$$

The \textit{reversibility} of a Finsler metric $F$ is defined as
$$
\lambda := \max \left\{ F(q,-v) : (q,v) \in TN, \; F(q,v)=1\right \} \geq 1.
$$
If $\lambda=1$, then we call $F$ \textit{reversible}, and if $\lambda>1$, it is \textit{non-reversible}. The results of Harris and Paternain \cite[Theorem B]{HP08} and Rademacher \cite[Theorem 4]{Rade04}, \cite[Proposition 1]{Rade08} imply the following. 

\begin{theorem}\label{thm"HP}
If a Finsler metric $F$ on $S^2$   with reversibility $\lambda$ is $\delta$-pinched for some $\delta > ( \lambda / (\lambda+1))^2$, i.e.\ the flag curvature  $K$ satisfies
\[
\left( \frac{\lambda}{\lambda+1} \right)^2 < \delta \leq K \leq 1,
\]
 then the contact form $h\lambda_0$ on $S^3$ is dynamically convex.  
\end{theorem}

\begin{remark}
In \cite{HS13Finsler} Hryniewicz and Salom\~ao showed that the pinching condition in the preceding theorem is sharp in the sense that  if $0< \delta < ( \lambda / (\lambda+1))^2$, there exists a $\delta$-pinched Finsler metric on $S^2$ with reversibility $\lambda$ such that the contact form $h\lambda_0$ is not dynamically convex.  
\end{remark}

Consider a real tight three-sphere $(S^3, \lambda, \rho)$, where $\rho$ has the form \eqref{antis} with either $\theta_1=0$ or $\theta_2=0$.     Choose $\hat{\rho}(z_1, z_2) =  ( e^{ i (\pi- \theta_1)} \overline{z_1}, e^{ i (\pi- \theta_2 )} \overline{z_2})$, so that $\rho \circ \hat{\rho}= \iota = \hat{\rho} \circ \rho$.   Via $\Phi$, $\rho$ and $\hat{\rho}$ descend to an anti-contact involution  $f_I $    on $(\Sigma^*, \lambda)$.  By means of the Legendre transformation and the projection $\Sigma^* \to S^2$, $f_I$ defines an involution  $f $   on $ S^2 $.   By the preceding discussion and Theorem \ref{ThmB3} with $p=2$, we obtain the following.

\begin{theorem} Let   $f$ be the smooth function on $S^2$ described above with  either $\theta_1=0$ or $\theta_2=0$.    Assume that $S^2$ admits a real Finsler metric $F$ with respect to $f$.   If   $F$   has    reversibility $\lambda$  and is  $\delta$-pinched for some $\delta > ( \lambda / (\lambda+1))^2$,  then  there exist two or infinitely many symmetric closed $F$-geodesics. 
\end{theorem}

\subsection{H\'enon$-$Heiles system}

The H\'enon-Heiles system describes chaotic motions of a star in a galaxy with an axis of symmetry \cite{HH64}. The Hamiltonian is   given by
\[
H(q_1, q_2, p_1, p_2) = \frac{1}{2} \lvert p \rvert^2 + \frac{1}{2} \lvert q \rvert^2 + q_1^2 q_2  -\frac{1}{3} q_2^3, \quad (q_1, q_2, p_1, p_2) \in \R^4.
\]
It is invariant under the   exact anti-symplectic involution
\[
\rho (q_1, q_2, p_1, p_2)=(-q_1, q_2, p_1, -p_2)  
\]
and the $3$-periodic exact symplectomorphism 
\[
 \sigma(q,p) := ( e^{ 2 \pi i /3}q,e^{ 2 \pi i /3}p), \quad (q,p)=(q_1 +iq_2, p_1 +ip_2)
  \]
satisfying  $\sigma \circ \rho \circ \sigma = \rho$.

There exist exactly two critical levels $0$ and $1/6$ of $H$, and for every $c \in (0, 1/6)$, the energy level $H^{-1}(c)$ contains a unique bounded component $\Sigma_c$ that    is diffeomorphic to $S^3$. It is proved by Schneider that $\Sigma_c$ bounds a strictly convex domain in $\R^4$ containing the origin, see \cite[Theorem 2.2]{Sch19}. Due to Hofer, Wysocki, and Zehnder  \cite[Theorem 3.7]{HWZ98convex} this shows that the contact form $\lambda =\lambda_0 |_{\Sigma_c}$ is dynamically convex, where $\lambda_0$ is the Liouville $1$-form on $\R^4$. In \cite{CPR79} Churchill, Pecelli, and Rod found eight periodic orbits on $\Sigma_c$, labeled by $\Pi_1, \ldots, \Pi_8$. They are related by 
\[
\sigma^2 ( \Pi_j) = \sigma( \Pi_{j+1}) = \Pi_{j+2}, \quad j=1, 4,
\]
and $\Pi_7$ and $\Pi_8$ are  $(\rho, \sigma)$-symmetric periodic orbits.  Therefore, they descend to four periodic orbits $\widehat{\Pi_1}, \widehat{\Pi_4},  \widehat{\Pi_7}$ and $\widehat{\Pi _8}$ on $(\Sigma_c / \Z_3, \widehat{\lambda}, \widehat{\rho})$, where the last two   are symmetric periodic orbits that are $3$-unknotted and have $sl=-1/3$.

Note that $\sigma|_{\Sigma_c} \neq g_{3,1}$. Therefore, the induced contact structure $\widehat{\xi}$ on $L(p,1) \cong \Sigma_c / \Z_3$ is different from $\overline{\xi_0}$. Consider the diffeomorphism $\psi \colon S^3 \to \Sigma_c$ defined as $\psi(z) = (1/\sqrt{H(z)})z$ and  the contactomorphism $\Phi$  on $(S^3, \lambda_0)$ given by
\[
\Phi ( q_1, q_2, p_1, p_2) = \frac{1}{\sqrt{2}} ( q_1 -p_2, -q_1-p_2, q_2 + p_1, q_2 - p_1).
\]
 It is easy to check that $(\Phi \circ \psi^{-1})^* \rho |_{\Sigma_c} = \rho_0 |_{S^3}$, where  
 \[
 \rho_0(q_1, q_2,p_1,p_2) = (-q_1, -q_2, p_1,p_2) ,
 \]
    and $(\Phi \circ \psi^{-1})^* \sigma |_{\Sigma_c} =  g_{3,1} $. Note that $\rho_0$ has the form \eqref{antis} with $\theta_1 =\theta_2=\pi$.  In this way we find that  $(\Sigma_c / \Z_3,  \widehat{\lambda}, \widehat{\rho})$ is contactomorphic to the dynamically convex real universally tight lens space $( L(3,2), \overline{\lambda_0}, \overline{\rho_0})$, where $\overline{\lambda_0}$ is the contact form induced by $\lambda$ via the maps $\psi$ and $\Phi$. The   periodic orbits $\widehat{\Pi_1}, \widehat{\Pi_4}, \widehat{\Pi_7}$, and $\widehat{\Pi _8}$   induce periodic orbits  $\overline{\Pi_1}, \overline{\Pi_4},  \overline{\Pi_7}$ and $\overline{\Pi _8}$    on $(L(3,2), \overline{\lambda}, \overline{\rho_0})$, where the last two   are symmetric periodic orbits which are $3$-unknotted and have $sl=-1/3$. An application of   Theorem \ref{Thmnew} and the result of Kang \cite{Kang2} implies  the following refinement of the result due to Schneider \cite[Corollary 2.5]{Sch19}.

\begin{theorem} For every $c \in (0, 1/6)$, the quotient manifold $ \Sigma_c / \Z_3$   carries a $\overline{\rho}$-symmetric periodic orbits that bounds a $\overline{\rho}$-invariant $3$-rational disk-like global surface of section. Moreover, $\Sigma_c/\Z_3$ carries   infinitely many $\overline{\rho}$-symmetric periodic orbits. They lift to infinitely many equivalence classes of periodic orbits on $\Sigma_c$, where the equivalence relation is defined as in the introduction.  
\end{theorem}

\subsection{Hill's lunar problem}
In \cite{Hill78} Hill studied the motion of the moon. The Hamiltonian of his lunar  problem,        obtained from  the  planar circular restricted three-body problem by blowing up the coordinates near the lighter primary, which we will call the earth,  is       given by
\[ 
H\colon T^* ( \R^2 \setminus \left \{ 0 \right \}) \rightarrow \R, \quad (q,p) = \frac{1}{2} |p|^2 - \frac{1}{|q|} + q_1 p_2 - q_2 p_1 - q_1^2 + \frac{1}{2}q_2^2 .
\]
 Note that $H$   admits the two commuting exact anti-symplectic involutions 
\begin{equation}\label{eq:twpinv}
\rho_1(q_1, q_2, p_1, p_2) = (q_1, -q_2, -p_1, p_2), \quad \rho_2 (q_1, q_2, p_1, p_2) = (-q_1 , q_2, p_1, -p_2).
\end{equation}
There exists a unique critical value $c_1 = -3^{4/3}/2$, and for every $c <c_1$, the energy level set $H^{-1}(c)$ contains a unique   component $\Sigma_c$ whose projection to the $q$-plane is bounded.
In \cite{Birk15} Birkhoff showed the existence of a retrograde periodic orbit on $\Sigma_c$ which is doubly symmetric, i.e.\ its trace is invariant under both $\rho_1$ and $\rho_2$. Numerical experiments show that a  direct periodic orbit also exists and is doubly-symmetric. See for instance \cite{Ayd19, Hen69}.

As before, $\Sigma_c$ is non-compact, and by regularizing we find that $(\overline{\Sigma_c}, \overline{\lambda}=\overline{\lambda_0}|_{\overline{\Sigma_c}})$ is contactomorphic to a universally tight lens space $ L(2,1)$.   It is easy to see that there exists a very negative energy $c_* <c_1$   such that for every $c<c_*$, the  image of the Levi-Civita embedding of $\Sigma_c$ bounds a strictly convex domain in $\C^2$. Therefore, $(\overline{\Sigma_c}, \overline{\lambda}=\overline{\lambda}_0|_{\overline{\Sigma_c}})$ is dynamically convex for  $c<c_*$. Because a retrograde periodic orbit is $2$-unknotted and has $sl=-1/2$, Theorem \ref{Thmnew} tells us that it bounds a $\rho_1$-invariant $2$-rational disk-like global surface of section  $\mathcal{D}$. Note that $\iota(\mathcal{D})$ is also a $\rho_1$-invariant $2$-rational disk-like global surface of section, where $\iota =-\mathrm{Id}= \rho_1 \circ \rho_2$. Together with \cite[Theorem 1.3]{Kang2} this proves the following statement.

\begin{theorem} Given $c<c_*$, a simply covered retrograde periodic orbit on $\overline{\Sigma_c}$ bounds a $\rho_1$-invariant rational disk-like global surface of section. As a consequence, there exist either two or infinitely many doubly-symmetric periodic orbits on $\overline{\Sigma_c}$. 
\end{theorem}

\noindent
Note that doubly-symmetric periodic orbits are symmetric periodic orbits of type I, see Subsection \ref{sec:3bp}.  Since Hill's lunar problem is a limit case of the planar circular restricted three-body problem that studies the motion near the earth, the statement of the above theorem is compatible with Theorem \ref{thm:L3cp}. For this reason, we expect that the above statement   holds for all $c<c_1$.

\section{Preliminaries}

In this section $(M, \lambda)$ or $(M, \lambda, \rho)$ always denotes a closed contact three-manifold or a closed real contact three-manifold with non-empty $\mathrm{Fix}(\rho)$, respectively.

\subsection{The Conley-Zehnder indices of periodic  orbits}\label{sec:ind1}

Let $P=(x,T)$ be a non-degenerate periodic  orbit on $(M, \lambda)$ and let $\mathfrak{T} \colon x_T^* \xi \to S^1 \times \C$ be a unitary trivialization, where $x_T(t):= x(Tt)$ and $S^1 = \R /\Z$.  We obtain a smooth path of $2 \times 2$ symplectic matrices 
\begin{equation}\label{eq:pathsymp}
\Phi_P^{\mathfrak{T}}(t) := \mathfrak{T} ( x_T(t)) \circ d\phi_R^{Tt}(x_T(0))|_{ \xi_{x_T(0)}} \circ \mathfrak{T} (x_T(0))^{-1}, \quad t \in \R
\end{equation}
satisfying that $\Phi_P^{\mathfrak{T}}(0) = \mathrm{Id}$ and  $\Phi_P^{\mathfrak{T}}(t +1 ) = \Phi_P^{\mathfrak{T}}(t) \Phi_P^{\mathfrak{T}}(1)$. 
The non-degeneracy of $P$ is equivalent to   $\det(\Phi_P^{\mathfrak{T}}(1) - \mathrm{Id}) \neq 0$.   We  define the Conley-Zehnder index of $P$ with respect to $\mathfrak{T}$ as 
\[
\muCZ^{\mathfrak{T}}(P ):= \muCZ (\Phi_P^{\mathfrak{T}} |_{[0,1]}) .
\]
For the original definition of the latter,  we refer the reader to   \cite{CZ83, CZ84}.

We now give an alternative analytic definition of the Conley-Zehnder index due to Hofer, Wysocki, and Zehnder \cite[Section 3]{HWZII} and extend the definition to \emph{any} periodic orbit.   Let $\Phi(t), t \in \R,$  be a smooth path of $2 \times 2$ symplectic matrices satisfying
\begin{equation}\label{eq:conditionsphi}
\Phi(0) = \mathrm{Id} \quad \text{ and } \quad  \Phi(t+1) = \Phi(t) \Phi(1). 
\end{equation}
The associated smooth loop of $2 \times 2$ symmetric matrices $S=S_{\Phi} $ is defined as
\[
S(t) := - J_0 \p_t \Phi(t) \Phi(t)^{-1}, \quad t \in S^1
\]
 where
 \[
 J_0 = \begin{pmatrix} 0 & -1 \\ 1 & 0 \end{pmatrix}
 \]
 is the standard complex structure. It provides us with the  unbounded operator
\[
A_S \colon W^{1,2}(S^1, \C) \to L^2(S^1, \C), \quad v \mapsto -J_0 \p_t v - Sv 
\]
having the following properties:
\begin{itemize}
\item It is self-adjoint.
\item The map $\ker A_S \to \ker( \Phi(1) - \mathrm{Id}), \; v \mapsto v(0)$ is a vector space isomorphism. 
\item The spectrum $\sigma(A_S) = \{ \eta \in \C \mid A_S - \eta \mathrm{Id} \text{ is non-invertible} \}$  consists of eigenvalues of $A_S$, is a discrete subset of $\R$, and accumulates only at $\pm \infty$. 
\end{itemize}
Let $ v \not\equiv 0$ be an eigenfunction of $A_S$ to the  eigenvalue $\eta$. As it is a solution of the ODE $-J_0 \p_t v - Sv = \eta v$, it is non-vanishing, so   we can  define the winding number of $v$ as 
\[
w(v, \eta, S) = \frac{1}{2\pi} \left( \mathrm{arg}( v(1)) - \mathrm{arg}( v( 0)) \right)  \in \Z.
\]
The following crucial properties are proved in \cite[Lemmas 3.4, 3.6, and 3.7]{HWZII}. 

\begin{lemma} [{Properties of the winding number}] \
\begin{enumerate}
\item If  $v_1$ and $v_2$ are eigenfunctions to the same eigenvalue $\eta$, then $w(v_1, \eta, S) = w(v_2, \eta, S)$. As a consequence, we can associate  with each eigenvalue $\eta$ of $A_S$ the winding number $w(\eta, S):=w(v, \eta, S)$, where $v$ is any eigenfunction of $A_S$ to the eigenvalue $\eta$. 

\item   For every $k \in \Z$, there exist exactly two eigenvalues  $\eta_1, \eta_2$ of $A_S$, counted with multiplicity, such that $k=w(\eta_1, S) = w(\eta_2, S)$. 

\item  The map $w \colon \sigma(A_S) \to \Z$, $\eta \mapsto w(\eta, S)$, is monotone, i.e.\ $\eta_1 \leq \eta_2$ implies that $w(\eta_1, S) \leq w(\eta_2, S)$. 
\end{enumerate}

\end{lemma}
\noindent
We define $\alpha(S) := \max \{ w(\eta, S) \mid \eta \in \sigma(A_S) \cap (-\infty, 0) \}$ and 
\[
p(S) := \begin{cases} 0 & \text{ if  there exists $\eta \in \sigma(A_S) \cap [0, \infty)$ such that $\alpha(S) = w(\eta, S)$}, \\ 1 & \text{ otherwise}.  \end{cases}
\]
 \begin{theorem}[{\cite[Theorem 3.10]{HWZII}}]\label{thm:indexdefhwz} If  $\ker(\Phi(1) - \mathrm{Id}) = \{ 0 \}$, then
 \begin{equation}\label{cznew}
 \muCZ(\Phi |_{[0,1]})  = 2\alpha(S) + p(S).
 \end{equation}
 \end{theorem}
\noindent
Following \cite{HWZ98convex}, for \emph{any} smooth path of $2 \times 2$ symplectic matrices satisfying \eqref{eq:conditionsphi},  we define its Conley-Zehnder index by the formula \eqref{cznew}.  
The following statement is well-known. See for instance \cite[Appendix 8.1]{HWZ03foliation}. 

\begin{proposition}  \label{lem:indeitercz}   Let $\Phi(t), t \in \R,$ be a smooth path of $2 \times 2$ symplectic matrices satisfying \eqref{eq:conditionsphi}.
For every $k \in \N$, we have 
 \begin{align*}
&\muCZ (\Phi |_{[0,1]})  \geq 3 \quad  \Rightarrow \quad \muCZ (\Phi |_{[0,k]} ) \geq 2k+1 ,\\
 &\muCZ (\Phi |_{[0,1]}) \geq 1 \quad  \Rightarrow \quad \muCZ (\Phi |_{[0,k]}) \geq1,\\
 &\muCZ (\Phi |_{[0,1]})  <1   \quad  \Rightarrow \quad \muCZ  (\Phi |_{[0,k]}) <1 .
\end{align*}

\end{proposition}

If $P$ is contractible, then we choose a unitary trivialization $\mathfrak{T}$    that extends over a spanning disk. As a disk is contractible,   the Conley-Zehnder index depends only on the choice of a spanning disk. If $c_1(\xi)$ vanishes on $\pi_2(M)$,  then it is also independent of the choice of a spanning disk and hence  provides us with   a well-defined Conley-Zehnder index of a contractible periodic orbit.   If $M=S^3$, then the tight contact structure is  a unique, up to homotopy, trivial symplectic vector bundle, and, therefore,  every   periodic   orbit on   the tight three-sphere  is assigned a  well-defined Conley-Zehnder index.   In the case of lens spaces, the Conley-Zehnder index is well-defined for every contractible periodic orbit because $\pi_2$ vanishes. In either case, we simply write $\muCZ(P)$.

\subsection{Two indices of  symmetric  periodic   orbits}\label{sec:ind2}

 Let $P=(x,T)$ be a non-degenerate symmetric periodic orbit  on $(M, \lambda, \rho)$.  Its Conley-Zehnder index with respect to a unitary trivialization is defined as in the preceding subsection.  Recall that $P$ can   be interpreted as a Legendrian intersection point, i.e.\ $c(t) := x(t)|_{[0,T/2]}$ is a   chord with boundary condition $c(0), c(T/2) \in \mathrm{Fix}(\rho)$.  This implies that one can assign a symmetric periodic orbit a different index explained as below.

Let  $\mathfrak{T} $ be a  symmetric unitary trivialization of $x_T^* \xi$,  namely,   a unitary trivialization having the additional property that 
\[
\mathfrak{T} ( x_T(t)) \circ D\rho (x_T(t)  )|_{ \xi_{x_T(t) }} = I \circ  \mathfrak{T}  (x_T(-t))   , \quad t \in S^1, 
\] 
where $ I (z) = \overline{z}.$
 The existence of a symmetric unitary trivialization was proved in \cite[Lemma 3.10]{FK16}. We define the associated smooth path $\Phi_P^{\mathfrak{T}}$ of $2 \times 2$ symplectic matrices by the same formula   \eqref{eq:pathsymp}. We identify $\R \equiv \R \times \{ 0 \} \subset \C$. The chord $C=(c,T/2)$ is said to be non-degenerate if $ \Phi_P^{\mathfrak{T}}(1/2) \R \cap \R = \{ 0 \}$. If $P$ is non-degenerate (as a periodic orbit), then $C$ is non-degenerate (as a chord), see \cite[Proposition 3.7]{FK16}.   It is easy to check that    $ \Phi_P^{\mathfrak{T}}(-t) = I  \Phi_P^{\mathfrak{T}}(t) I$ for all $ t \in  \R$. We   define the Robbin-Salamon index of  a non-degenerate symmetric periodic orbit $P$ with respect to $\mathfrak{T}$ as the Maslov index for the Lagrangian path $\Phi_P^{\mathfrak{T}}|_{[0,1/2]} \R$ with respect to the reference Lagrangian $\R$:
\[
\muRS^{\mathfrak{T}} (P  ) := \muRS( \Phi_P^{\mathfrak{T}}|_{[0,1/2]} \R, \R).
\]
This definition is due to   Robbin and Salamon    \cite{RS93}. 

As before, we give  an alternative analytic definition of the Robbin-Salamon index due to Frauenfelder and Kang \cite[Section 3]{FK16}.  Let $\Phi(t), t \in \R$  be a smooth path of $2 \times 2$ symplectic matrices such that $\Phi(0) = \mathrm{Id}$,   $\Phi(t+1) = \Phi(t)\Phi(1)$, and $\Phi(-t) = I \Phi(t) I$.  Denote by $S=S_{\Phi}$ the associated smooth loop of $2 \times 2$ symmetric matrices. It is straightforward to check that $S(-t) = I S(t) I $. We in particular find that $S(0), S(1/2)$ are diagonal matrices. 

Set $D:= S|_{[0,1/2]}$. 
Consider the following Hilbert space of paths satisfying a Lagrangian boundary condition
\[
W^{1,2}_{\R} \left( \left[0, \frac{1}{2}\right] ,\C\right) = \left \{ v \in W^{1,2}\left( \left[ 0, \frac{1}{2} \right] ,\C\right) \; \bigg| \; v(0), v\left( \frac{1}{2} \right) \in \R  \right \}
\]
and the unbounded operator $A_D \colon W^{1,2}_{\R} ([0,1/2], \C) \to L^2([0,1/2], \C)$  defined by $A_D(v) = -J_0 \p_t v - Dv$. It has the   following properties:
\begin{itemize}
\item It is self-adjoint.
\item The map $\ker A_D \to  \Phi(1/2)\R \cap \R, \; v \mapsto v(0)$ is a vector space isomorphism. 
\item The spectrum $\sigma(A_D)$   consists of eigenvalues of $A_D$, is a discrete subset of $\R$, and accumulates only at $\pm \infty$. 
\end{itemize}
Let $ v \not\equiv 0$ be an eigenfunction of $A_D$ to the  eigenvalue $\eta$. As before, we   assign it the well-defined relative winding number
\[
w(v, \eta, D) = \frac{1}{2\pi} \left( \mathrm{arg}( v(1/2)) - \mathrm{arg}( v( 0)) \right)  \in \frac{1}{2} \Z.
\]
The following  properties are proved in \cite[Lemmas 3.4 and 3.5]{FK16}.  

\begin{lemma} [{Properties of the relative winding number}] \
\begin{enumerate}
\item The relative winding number depends only on the eigenvalue. Hence,   we   associate with every eigenvalue $\eta$ of $A_D$ the relative winding number $w(\eta, D):=w(v, \eta, D)$, where $v$ is any eigenfunction of $A_D$ to the eigenvalue $\eta$. 

\item   For every $k \in \Z$, there exists a unique eigenvalue $\eta$ of $A_D$   such that
\[
w( \eta, D) = \frac{k}{2}.
\]
\end{enumerate}

\end{lemma}
\noindent
Define $\alpha(D) := \max \{ w(\eta, D) \mid \eta \in \sigma(A_D) \cap (-\infty, 0) \}.$ If $ \Phi(1/2)\R \cap \R = \{ 0 \}$, then
\[
\muRS(\Phi|_{[0,1/2]}) = 2\alpha(D) + \frac{1}{2} \in \Z + \frac{1}{2}.
\]
We have a symmetric counterpart of Proposition \ref{lem:indeitercz}.

\begin{proposition} [{\cite[Proposition 3.9]{FK16}}] \label{lem:indeiterrs} For every $k \in \N$ we have
\begin{align*}
\muRS(\Phi|_{[0,1/2]}) \geq \frac{3}{2} \quad  &\Longrightarrow \quad \muRS(\Phi|_{[0,k/2]}) \geq \frac{2k+1}{2},\\
\muRS(\Phi|_{[0,1/2]})\geq \frac{1}{2} \quad  &\Longrightarrow \quad\muRS(\Phi|_{[0,k/2]})   \geq \frac{1}{2}, \\
\muRS(\Phi|_{[0,1/2]}) \leq  - \frac{1}{2} \quad  &\Longrightarrow \quad \muRS(\Phi|_{[0,k/2]}) \leq  -\frac{1}{2}.
\end{align*}

\end{proposition}

 The following lemma reveals a relationship between the two indices   of a symmetric periodic orbit.

\begin{lemma}[{\cite[Proposition 3.8]{FK16}}] \label{inexd"rel}   If $P=(x,T)$ is a non-degenerate symmetric periodic orbit equipped with a symmetric unitary trivialization $\mathfrak{T}$, then   
\[
\lvert \muCZ^{\mathfrak{T}}(P) - 2\muRS^{\mathfrak{T}}(P) \rvert \leq 1.
\]
In particular, if $\muCZ^{\mathfrak{T}}(P) \geq 3$, then $\muRS^{\mathfrak{T}}(P) \geq 3/2$. 
\end{lemma}

Let $\mathfrak{T}_1, \mathfrak{T}_2 \colon x_T^* \xi \to S^1 \times \C$ be symmetric unitary trivializations of $x_T^* \xi$, where $P=(x,T)$ is a symmetric periodic orbit.   We define the winding number $\wind(\mathfrak{T}_1, \mathfrak{T}_2)$ of $\mathfrak{T}_2$ with respect to $\mathfrak{T}_1$ as follows. For any $\alpha \in \C \setminus \{ 0 \}$, we write $\alpha (t) =(t, \alpha)$, $t \in S^1$.   The winding number  $\wind(\mathfrak{T}_1, \mathfrak{T}_2)$ is then defined as the degree of the loop $\mathfrak{T}_1 \circ \mathfrak{T}_2^{-1} \circ \alpha \colon S^1 \to \C \setminus \{ 0 \}$. This does not depend on the choice of $\alpha$.

\begin{lemma} \label{lem:difftrirel} 
Let $P$,  $\mathfrak{T}_1$, and $\mathfrak{T}_2$ be as above.   We then have
\[
\muRS^{\mathfrak{T}_1}(P) = \muRS^{\mathfrak{T}_2}(P) + \wind(\mathfrak{T}_1, \mathfrak{T}_2).
\]

\end{lemma}
\begin{proof} In view of  \cite[Corollary 3.14]{FK16} we have
\[
\muRS^{\mathfrak{T}_1}(P) = \muRS^{\mathfrak{T}_2}(P) + \frac{1}{2} \muRS( \Lambda \R, \R), 
\]
where $(\Lambda \R)(t)=\Lambda(t)\R := \mathfrak{T}_1(t) \circ \mathfrak{T}_2 ^{-1}(t) \R$, $t \in S^1$, is a loop of Lagrangians, where $\mathfrak{T}_j (t) := \mathfrak{T}_j (x_T(t), \cdot)$ and $\R \equiv \R \times \{ 0 \} \subset \C$. By the Maslov index property of the Robbin-Salamon index, $\muRS( \Lambda \R, \R)$ equals the Maslov index of the loop $\Lambda \R$. By definition of the Maslov index, we then conclude that
\[
\muRS( \Lambda \R, \R) = 2 \wind(\mathfrak{T}_1, \mathfrak{T}_2) 
\]
from which the lemma is proved. 
\end{proof}

If $P=(x,T)$ is a non-degenerate symmetric periodic orbit on $(M, \lambda, \rho)$  that is symmetrically contractible, meaning that it admits a spanning disk invariant under $\rho$, then we choose a symmetric unitary trivialization  that extends over the symmetric spanning disk. In this case, the Robbin-Salamon index of $P$ depends only on the spanning disk. If $c_1(\xi)|_{  \pi_2(M)}=\{0\}$, then it is well-defined, i.e.\ it is independent of the choice of the spanning disk, see \cite[Corollary 3.14]{FK16}.
If $M=S^3$, then $\pi_1(S^3, \mathrm{Fix}(\rho))=\{0\}$, and hence every symmetric periodic orbit is symmetrically contractible. Therefore, they have a well-defined Robbin-Salamon index. On lens spaces, every symmetrically contractible periodic orbit has a well-defined Robbin-Salamon index since $\pi_2$ vanishes.   In either case, we simply write $\muRS(P)$.

\subsection{Self-linking number}\label{sec:self}

Let $\gamma \colon S^1 \to M$ be an unknot transverse to $\xi = \ker \lambda$ and $u \colon \D \to M$ be an embedding such that $u( e^{ 2 \pi i t} )=\gamma(t)$, $ t \in S^1$. Choose any non-vanishing section $X$ of the pullback bundle $u^* \xi  \to \D$. The orientation of $\gamma$ and $u$ are given by the standard orientations, and $M$ is oriented in such a way that  $\lambda \wedge d \lambda >0$. We fix a Riemannian metric $g$ on $M$ and denote by exp the associated exponential map. As $\gamma$ is a transverse knot, if $X$ is sufficiently small, then
\[
\gamma(S^1) \cap \gamma_X (S^1) = \emptyset,
\]
where $\gamma_X \colon S^1 \to M$ is defined as $\gamma_X(t) := \exp_{\gamma(t)}X(t)$. By definition, the \emph{self-linking number} of $\gamma$ with respect to $u$, denoted by $sl(\gamma, u)$, is given by the algebraic intersection number of $\gamma_X$ with $u$.  It follows from the homotopy invariance of the algebraic intersection  number that it is independent of the choices of $X$ and $g$. If $c_1(\xi)|_{\pi_2(M)}=0$, then it is also independent of the choice of $u$. In this case, we simply write it as $sl(\gamma)$.

We now assume that a transverse knot $\gamma \colon S^1 \to M$ is  {$p$-unknotted} for some $p \geq 2$. Let $u$ be a $p$-disk for $\gamma$ and    define $\gamma_X$ as above. Then the \emph{rational self-linking number} of $\gamma$ with respect to $u$ is defined as 
\[
sl(\gamma, u):=\frac{1}{p^2}  \bullet \text{algebraic intersection number of $\gamma_X$ with $u$.}
\]
The orientations are given as above, and   the definition is also independent of all the involved choices, provided that $c_1(\xi)|_{\pi_2(M)}=0$. In this case,  we write it as $sl(\gamma)$.

\subsection{Open book decompositions}

 Let $M$ be a closed oriented three-dimensional manifold. An \emph{open book decomposition} of $M$ is a pair $(B, \pi)$, where $B$ is an oriented link, called the \emph{binding}, and $\pi \colon M \setminus B \to S^1$ is  a fibration. For every $\theta \in S^1$, the fiber $\Sigma_{\theta}:=\pi^{-1}(\theta)$ is the interior of a compact embedded surface  $\Sigma \subset M$ with $\p  \Sigma  = B$. The surface $\Sigma$ is  called a \emph{page} of the open book.  The existence of an open book decomposition of every closed oriented three-manifold was  proved by Alexander  \cite{Ale23}. In particular, every closed contact three-manifold admits an open book decomposition.

  If every $\Sigma_{\theta}$ is a  disk, then we say that $(B, \pi)$ is an open book decomposition \emph{with disk-like pages}. In this case,   $M$ is    diffeomorphic to   $S^3$.     We say that an open book decomposition $(B, \pi)$ with disk-like pages of a closed contact three-manifold $(M, \lambda)$ is \emph{adapted to $\lambda$} if  $B$ is a periodic  orbit, and every page is a disk-like global surface of section for $R = R_{\lambda}$.

 A $\lambda$-adapted \emph{$p$-rational open book decomposition}  with disk-like pages is a pair $(P, \pi)$, where $P$ is a $p$-unknotted periodic orbit and a fibration $\pi$ satisfies that for every $\theta \in S^1$, the closure of the fiber $\pi^{-1}(\theta)$ is the image of a $p$-disk for $P$.

\subsection{Pseudoholomorphic curves in symplectizations}

In this subsection we recall the definitions  and   basic properties of   pseudoholomorphic curves in symplectizations of closed contact three-manifolds introduced by Hofer \cite{Hofer93} and developed by Hofer, Wysocki, and Zehnder \cite{HWZ95I,  HWZII, HWZI, HWZ98convex, HWZ99II, HWZIII, HWZ03foliation}. We also follow a recent book   \cite{FvK18book}. For real pseudoholomorphic curves, we refer to  \cite{FK16}.

\subsubsection{Finite energy spheres}
Let $(M, \xi =\ker \lambda)$ be a closed contact three-manifold.  
It is well-known that the space $\mathcal{J} = \mathcal{J}(\xi, d\lambda)$ of $d\lambda$-compatible almost complex structures on $\xi$ is non-empty and contractible in the $C^{\infty}$-topology. 
Denote by $\pi \colon TM \to \xi$ the projection along $R$.  Associated to every $J \in \mathcal{J}$ is the  \emph{SFT-like almost complex structure} $\widetilde{J}$ on $ \R \times M$  defined by
\[
\widetilde{J}(r,p)(h,k) = ( - \lambda(p)k, J(p) (\pi k) + hR(p)), \quad (r,p) \in \R \times M.
\]
Here, the non-compact four-manifold $\R \times M$ is equipped with the symplectic form $d(e^r \lambda)$. We call the pair $(\R \times M, d(e^r \lambda))$ the \emph{symplectization of $(M, \lambda)$}. It is straightforward to check that   $\widetilde{J}$ is $d(e^r \lambda)$-admissible.

Let $(S^2 = \C \cup \{ \infty \}, i)$ be a closed Riemann sphere and $\Gamma \subset S^2$ be a finite set of points. Fix $J \in \mathcal{J}$. A smooth map $\widetilde{u}=(a,u) \colon S^2 \setminus \Gamma \to \R \times M$ is called a \emph{finite energy $(\widetilde{J}$-holomorphic$)$ sphere} if it satisfies the nonlinear Cauchy-Riemann equation
\[
d\widetilde{u} \circ i = \widetilde{J} \circ d\widetilde{u}
\]
and has   non-zero  finite energy $E(\widetilde{u}) \in (0, \infty)$, where $E(\widetilde{u})$ is defined as
\[
E(\widetilde{u}) := \sup_{\phi \in \Xi} \int_{S^2 \setminus \Gamma} \widetilde{u}^* d (\phi \lambda).
\]
Here $\Xi$ is the set of  all smooth functions  $ \phi \in C^{\infty}(\R, [0,1])$  with $\phi' \geq 0 $. The elements in $\Gamma$ are called \emph{punctures}. If $\Gamma = \{ \infty\}$ so that $S^2 \setminus \Gamma = \C$, then $\widetilde{u}$ is called a \emph{finite energy plane}.  Since the SFT-like almost complex structure $\widetilde{J}$ is $\R$-invariant,  for every $r \in \R$ and for every finite energy $\widetilde{J}$-holomorphic sphere $\widetilde{u} = (a,u)$, the map $\widetilde{u}_r := (a+r, u)$  is a finite energy  $\widetilde{J}$-holomorphic sphere as well.

Fix $ z \in \Gamma$ and choose a small neighborhood of $U_z$ of $z$ in $S^2 \setminus \Gamma$. If $a   $ is bounded on $U_z$,  then $z$ is called a \emph{removable puncture}. If $a$ is unbounded on $U_z$, then it is bounded either from  below or from above. In the first case, we call $z$ a \emph{positive puncture}. It is called a \emph{negative puncture} in the latter case. In the sequel we always assume that $\Gamma$ consists only of non-removable punctures and write 
\[
\Gamma = \Gamma^+ \sqcup \Gamma^-,
\]
where $\Gamma^{\pm}$ are the set of positive and negative punctures, respectively.  The maximum principle says that $\Gamma^+ \neq \emptyset$, see \cite[Lemma 5.2]{HWZII}. If $\widetilde{u}$ is a finite energy plane, then $\Gamma = \Gamma^+ = \{ +\infty \}$.

We now consider a special class of finite energy spheres that respect a given symmetry. Let $(M, \lambda, \rho)$  be a closed real contact three-manifold. The anti-contact involution $\rho$ on $(M, \lambda)$ defines the exact anti-symplectic involution $\widetilde{\rho}:=\mathrm{Id}_{\R} \times \rho$ on $(\R \times M, d(e^r \lambda))$, meaning that $\widetilde{\rho}^* ( e^r \lambda) = - e^r \lambda$. An almost complex structure $J \in \mathcal{J}$ is \emph{$\rho$-anti-invariant} if 
\[
D\rho|_{\xi} \circ J \circ ( D\rho|_{\xi})^{-1} = -J.
\]
Abbreviate by $\mathcal{J}_{\rho} \subset \mathcal{J}$ the set of such almost complex structures. For $J \in \mathcal{J}_{\rho}$, the associated SFT-like almost complex structure $\widetilde{J}$ is $\widetilde{\rho}$-anti-invariant. Assume that $\widetilde{u} = (a,u)$ is a finite energy sphere with respect to $J \in \mathcal{J}_{\rho}$. It is easy to check that $\widetilde{\rho} \circ \widetilde{u} \circ I$ is also such  a finite energy sphere, where $I(z) = \overline{z}$. The following definition is due to Frauenfelder and Kang \cite{FK16}.

\begin{definition} A finite energy sphere $\widetilde{u} = (a,u)$ with respect to $J \in \mathcal{J}_{\rho}$ is said to be \emph{invariant} if
\[
\widetilde{\rho} \circ \widetilde{u} \circ I = \widetilde{u}.
\]
\end{definition}

 Let $\widetilde{u} = (a,u)$ be an invariant finite energy sphere. In this case we have $I(\Gamma) = \Gamma$ and   write it   as the disjoint union $\Gamma = \Gamma_{\p} \sqcup \Gamma_{\mathrm{int}}$,  where  $I(z)= z$ for all $ z \in \Gamma_{\p}$ and $I(z) \neq z $ for all $z  \in \Gamma_{\mathrm{int}}$. Obviously, $+\infty \in \Gamma_{\p}$. Note that   $z \in \Gamma_{\p} $ if and only if $ \mathrm{Im}(z) =0$. Via $I$, we divide $S^2$ into two (closed) hemispheres and denote by $H$ the one containing $+\infty$. Note that $\Gamma_{\p} \subset H$.  Set $\Gamma_I = \Gamma_{\p} \sqcup ( \Gamma_{\mathrm{int}} \cap H)$.  Associated to $\widetilde{u}$ is a holomorphic map $\widetilde{u}_I =(a_I, u_I) \colon  (H \setminus \Gamma_I, i) \to ( \R \times M, \widetilde{J})$ with the boundary condition $\widetilde{u}_I (\mathrm{Fix}(\rho) \setminus \Gamma_{\p}) \subset \mathrm{Fix}(\widetilde{\rho})$. If it has non-zero finite energy, then we call it a \emph{finite energy half-sphere} (associated to $\widetilde{u}$). If $\Gamma = \{ + \infty \}$, then it is called a \emph{finite energy half-plane}.

\subsubsection{Asymptotic behavior}
  Choose a holomorphic chart $\psi_z \colon (V, 0) \to (U_z, z)$, where $U_z$ is a small neighborhood of $z \in \Gamma$. If $z \in \Gamma^+$ or $z \in \Gamma^-$, then we set   $\widetilde{u}(s,t) = \widetilde{u} \circ \psi_z (e^{- 2 \pi (s+it) })$ for $(s,t) \in [R, +\infty) \times S^1$ or $\widetilde{u}(s,t) = \widetilde{u} \circ \psi_z (e^{  2 \pi (s+it) })$ for $(s,t) \in (-\infty, -R]  \times S^1$, respectively,    for some $R>0$ large enough.    The following theorem is due to Hofer \cite[Theorem 31]{Hofer93} and Hofer, Wysocki, and Zehnder \cite[Theorem 1.3]{HWZI}.

\begin{theorem}\label{thm:asymphfwz} Assume that $\widetilde{u} =(a,u) \colon S^2 \setminus \Gamma \to \R \times M$ be a finite energy sphere. We fix $z \in \Gamma^{\pm}$ and consider a holomorphic chart $\psi_z$ near $z $ as above. Then for every sequence $s_n \to \pm \infty$, there exists a subsequence, still denoted by $s_n$, and a periodic   orbit $P=(x,T)$ such that $u(s_n, t)$ converges in $C^{\infty}(S^1, M)$ to $x_T(t) =x( Tt)$ as $n \to +\infty$.   Moreover, if $P$ is non-degenerate, then we have 
\[
\lim_{s \to \pm \infty} u(s, t) = x_T(  t) \quad \text{ in }\; C^{\infty}(S^1, M).
\]
\end{theorem}
 
 \begin{remark} Let $\widetilde{u}=(u,a)$ be a finite energy plane. Abbreviate by $\Omega = \Omega(\widetilde{u})$ the $\omega$-limit set of $\widetilde{u}$ consisting of periodic   orbits $P=(x,T)$ for which there exists  $s_n \to +\infty$ such that  $ u(s_n, t) $ converges in   $C^{\infty}(S^1, M)$ to $x_T( t)$.  It is non-empty, compact, and connected, see \cite[Lemma 13.3.1]{FvK18book}.  The last statement of the preceding theorem tells us that if an asymptotic limit $P$ of $\widetilde{u}$ is non-degenerate, then   $\Omega$ consists of a single element $P$, up to reparametrization. In \cite{Sie17multiple}  Siefring provided explicit examples of finite energy planes with the image of the $\omega$-limit set being diffeomorphic to the two-torus.  
 
 \end{remark}

Let $P=(x,T)$ be a non-degenerate periodic orbit. Denote by $\Gamma^{1,2}(x_T^* \xi)$ and   $\Gamma^{0,2}(x_T^* \xi)$ the Hilbert spaces of $W^{1,2}$-sections and $L^2$-sections  of $x_T^*\xi$, respectively.  Given $J \in \mathcal{J}$, the $L^2$-inner product here is defined as
\[
\left< \zeta_1, \zeta_2 \right> = \int_0^1 d\lambda(x_T(t)) ( \zeta_1(t), J(x_T(t)) \zeta_2(t))dt,
\]
where $\zeta_1, \zeta_2$ are sections of $x_T^*\xi$. For  any symmetric connection $\nabla$ on $TM$, we define the    operator $A_P \colon \Gamma^{1,2}(x_T^* \xi) \to \Gamma^{0,2}(x_T^* \xi)$  as $A_P( \zeta):=  J(- \nabla_t \zeta + T \nabla_{\zeta} R). $  This operator does not depend on the choice of $\nabla$.  Any unitary trivialization $\mathfrak{T} \colon x_T^* \xi \to S^1 \times \C$ provides us with Hilbert space isomorphisms $   \Gamma^{1,2}(x_T^* \xi ) \cong W^{1,2}(S^1, \C)$ and $\Gamma^{0,2}(x_T^* \xi ) \cong L^{ 2}(S^1, \C)$. Under these isomorphisms the operator $A_P$ becomes the self-adjoint operator $A_S$ defined in Subsection \ref{sec:ind1}, where $S=S_P^{\mathfrak{T}}$ is a smooth loop of symmetric matrices associated to $P$ and $\mathfrak{T}$.

In \cite{HWZI}, Hofer, Wysocki, and Zehnder proved the following asymptotic behavior of a non-degenerate finite energy sphere.

\begin{theorem}\label{thm:asymptoticbah} Let $\widetilde{u} = (a,u) \colon S^2 \setminus \Gamma \to \R \times M$ be a   finite energy sphere. Suppose that an asymptotic limit $P_z=(x_z,T_z)$ at $z \in \Gamma$ is non-degenerate.  Let $\psi_z \colon (V, 0) \to (U_z, z)$ be a holomorphic chart as above.  If  $z$ is a positive  puncture $($or a negative puncture$)$,   then there exists  a negative $($or positive$)$ eigenvalue $\eta_z$ of the operator $A_{P_z}$ defined as above and an eigenfunction $\zeta_z$ of $A_{P_z}$ to the eigenvalue $\eta_z$ such that 
\[
\widetilde{u}(s,t) = (T_zs, \exp_{x_{T_z}(t)} [ e^{ \eta_z s}( \zeta_z(t) + \kappa_z(s,t))]),  \quad  s \gg 1 \;\; \text{$($or $s \ll -1)$},
\]
where $\kappa_z$ decays exponentially with all derivatives to $0$ uniformly in $t$ as $s \to + \infty$ $($or as $s \to - \infty$$)$. Here $\exp$ is the exponential map of the restriction of the metric $d(e^r \lambda)( \cdot, \widetilde{J} \cdot)$ to $M \cong \{ 0 \} \times M \subset \R \times M$.   
\end{theorem}
\noindent
The eigenvalue $\eta_z$  and the eigenfunction $\zeta_z$ of $A_{P_z}$ in the above statement are called the \emph{asymptotic eigenvalue} and the \emph{asymptotic eigenfunction} of $\widetilde{u}$ at $z$, respectively.

Assume now that $\widetilde{u} =(a,u)$ is an invariant finite energy sphere. For $ z \in \Gamma_{\mathrm{int}}$, we have the same assertions as above. If $z \in \Gamma_{\p}$, then we   require a holomorphic chart   $\psi_z \colon (V, 0) \to (U_z, z)$ to further satisfy $I \circ \psi_z = \psi_z \circ I$. In this case, $\widetilde{u}$ tends asymptotically to a symmetric periodic orbit $P=(x,T)$ at $ z $. Assume that $P$ is non-degenerate and consider the associated non-degenerate half-chord $C=(c, T/2)$.  Set $c_T(t):=c(Tt)=x(Tt)|_{t \in [0,1/2]}$ and denote by $ \Gamma^{1,2}_{\rho} (c_T^* \xi)$ the set of $W^{1,2}$-sections $\zeta$ of $c_T^* \xi$ with boundary condition $\zeta(t) \in \mathrm{Fix}( D\rho(c_T(t))|_{\xi_{c_T(t)}})$ for $ t=0,1/2$. The Hilbert space $\Gamma^{0,2}(c_T^*\xi)$ is defined as before.    For $J \in \mathcal{J}_{\rho}$ and  any $\rho$-invariant symmetric connection $\nabla$ on $TM$,  the    operator $A_C \colon \Gamma^{1,2}_{\rho} (c_T^* \xi) \to \Gamma^{0,2}(c_T^* \xi)$ is defined by the same formula as $A_P$. As before, via a symmetric unitary trivialization  this operator is conjugated to the operator $A_D^{\mathfrak{T}}$ defined in Subsection \ref{sec:ind2}.    We have the following symmetric counterpart of the preceding two theorems.

\begin{theorem}[{\cite[Corollary 4.2]{FK16}}] \label{thm:asymptoticbahinv} 
 Assume that $\widetilde{u} =(a,u) \colon S^2 \setminus \Gamma \to \R \times M$ is an invariant  finite energy sphere with all asymptotic limits being non-degenerate.  For $z \in \Gamma_{\mathrm{int}}$, the same assertions as in Theorems  \ref{thm:asymphfwz} and  \ref{thm:asymptoticbah} hold.     Let $Q_{z}=(y_{z}, T_{z})$  be the associated asymptotic limit and   $\eta_{z}$ be the asymptotic eigenvalue of $A_{Q_{z}}$.  Then the periodic orbit  $\overline{Q_{z}}=( \rho \circ y_{z} , T_{z})$  is the asymptotic limit of $\widetilde{u}$ at $\overline{z} \in \Gamma_{\mathrm{int}}$ and the associated asymptotic eigenvalue $\eta_{\overline{z}}$ of $A_{\overline{Q_{z}}}$ coincides with $\eta_{z}$.

 Suppose that $ z \in \Gamma_{\p}$ is a positive $($or negative$)$ puncture and $P_z=(x_z,T_z)$ be the associated asymptotic limit.  Let $C_z =(c_z,T_z/2)$ be the half-chord of $P_z$.    Then there exists  a negative $($or positive$)$ eigenvalue $\eta_z^{\rho}$ of the operator $A_{C_z}$   and an eigenfunction $\zeta_z^{\rho}$ of $A_{C_z}$ to the eigenvalue $\eta_z^{\rho} $ such that 
\[
\widetilde{u}_I(s,t) = (T_zs, \exp_{c_{T_z}(t)} [ e^{ \eta_z^{\rho} s}( \zeta_z^{\rho} (t) + \kappa _z ^{\rho}(s,t))]) , \quad  s \gg 1 \;\; \text{$($or $s \ll -1)$},
\]
where $\kappa_z^{\rho}$ decays exponentially with all derivatives to $0$ uniformly in $t$ as $s \to + \infty$ $($or $s \to - \infty)$.  Here $\exp$ is the exponential map of the restriction of the $\widetilde{\rho}$-invariant metric $d(e^r \lambda)( \cdot, \widetilde{J} \cdot)$ to $M \cong \{ 0 \} \times M \subset \R \times M$. 
\end{theorem}

 \subsubsection{Algebraic invariants}
 
 Let  $\widetilde{u} = (a,u)$ be a finite energy sphere such that every asymptotic limit  is non-degenerate and $\pi \circ du \not\equiv0$. Since it is $\widetilde{J}$-holomorphic, the map $z \mapsto \pi \circ du(z)$ is a section of the complex line  bundle $\mathrm{Hom}_{\C} (T( S^2 \setminus \Gamma ), u^* \xi) \to S^2 \setminus \Gamma$. Due to Carleman's similarity principle     $\pi \circ du$ has finitely many zeros, and  they are in particular isolated. Moreover,   as a section of the bundle $u^* \xi \to S^2 \setminus \Gamma$, the zeros have a positive degree, see \cite[Proposition 4.1]{HWZII}. Let $\mathrm{wind}_{\pi}(\widetilde{u})$ denote the sum of the degrees of the zeros of the section $\pi \circ du$.  
\begin{lemma} [{\cite[Corollary 4.2]{HWZII}}] It holds that  $\mathrm{wind}_{\pi}(\widetilde{u}) \geq 0$, and $\mathrm{wind}_{\pi}(\widetilde{u}) =0$ if and only if $\pi \circ du$ is non-vanishing, i.e.\  $u$ is an immersion transverse to   $R$. 
\end{lemma}

Choose $ z \in \Gamma$ with the asymptotic limit $P_z=(x_z,T_z)$. Let $\eta_z$ be  the asymptotic eigenvalue  of  $\widetilde{u}$ at $z$ given in Theorem \ref{thm:asymptoticbah}. Since the operator $A_{P_z}$ is conjugated to the self-adjoint operator $A_{S_{P_z}^{\mathfrak{T}_z}}$, the asymptotic eigenvalue $\eta_z$ is also an eigenvalue of $A_{S_{P_z} ^{\mathfrak{T}_z}}$ for any unitary trivialization $\mathfrak{T}_z$ of $x_{T_z}^* \xi$.  However, the winding number of $\eta_z$, defined in Subsection \ref{sec:ind1}, depends on the choice of the trivialization. Indeed, if $\zeta_z$ is the asymptotic eigenfunction, then $\Phi_{\mathfrak{T}_z} \zeta_z$ is an eigenvector of $A_{S_{P_z}^{\mathfrak{T}_z}}$ to the eigenvalue $\eta_z$, where  $ \Phi_{\mathfrak{T}_z}  \colon \Gamma^{1,2}(x_{T_z}^* \xi ) \to  W^{1,2}(S^1, \C)$ is the Hilbert space isomorphism defined via $\mathfrak{T}_z$.  For two different trivializations $\mathfrak{T}_1, \mathfrak{T}_2$, the winding numbers of the eigenvectors $\Phi_{\mathfrak{T}_1} \zeta_z,  \Phi_{\mathfrak{T}_2} \zeta_z$ are not necessarily equal.

We now define the asymptotic winding number of $\widetilde{u} $ at $z$  with respect to a unitary trivialization $\mathfrak{T}_z$ of $x_{T_z}^* \xi$  as
\[
\wind_{\infty}^{\mathfrak{T}_z}(\widetilde{u}, z) := w(\eta_z, S_{P_z}^{\mathfrak{T}_z}).
\]
The \emph{asymptotic winding number of $\widetilde{u}$} is then defined as
\[
\mathrm{wind}_{\infty} (\widetilde{u}  ) := \sum_{z \in \Gamma^+} \mathrm{wind}_{\infty}^{\mathfrak{T} }(\widetilde{u}, z)- \sum_{z \in \Gamma^-} \mathrm{wind}_{\infty} ^{\mathfrak{T} }(\widetilde{u}, z).
\]
Each term in the right-hand side is computed using a unitary trivialization $\mathfrak{T}$ of $u^* \xi$, and  as the notation indicates, the left-hand side is independent of the choice of $\mathfrak{T}$.   The following statement reveals a relation  between the above defined two winding numbers of $\widetilde{u}$.

\begin{proposition}[{\cite[Proposition 5.6]{HWZII}}] \label{prophwzii5.6} Let $\widetilde{u} \colon S^2 \setminus \Gamma \to \R \times M$ be a non-degenerate finite energy sphere. Then we have
\[
\wind_{\infty}(\widetilde{u}) = \wind_{\pi} (\widetilde{u}) + 2 - \#\Gamma.
\]
In particular, if $\widetilde{u}$ is a finite energy plane, then $\wind_{\infty}(\widetilde{u})  \geq 1$, and $\wind_{\infty}(\widetilde{u}) =1$ if and only if $\pi \circ du$ is non-vanishing, i.e.\  $u$ is an immersion transverse to  $R$. 
\end{proposition}

Likewise, associated to the non-degenerate asymptotic limit $P_z $ of $z \in \Gamma$ with respect to a unitary trivialization $\mathfrak{T}_z$ of $x_{T_z}^* \xi$ is the Conley-Zehnder index
\[
\muCZ^{\mathfrak{T}_z}(P_z)  = 2\alpha(S_P^{\mathfrak{T}_z}) + p(S_P^{\mathfrak{T}_z}),
\]
where it is defined as in Subsection \ref{sec:ind1}. The \emph{index of a non-degenerate finite energy sphere $\widetilde{u}$} is then defined to be
\[
\mu(\widetilde{u}) := \sum_{z \in \Gamma^+} \muCZ^{\mathfrak{T}}(P_z)- \sum_{z \in \Gamma^-} \muCZ^{\mathfrak{T}}(P_z)
\]
Again, each index   in the right-hand side is computed using a unitary trivialization $\mathfrak{T}$ of $u^* \xi$, and    this definition does not depend  on the choice of the trivialization, see \cite[Proposition 5.5]{HWZII}.

\begin{theorem}[{\cite[Theorem 5.8]{HWZII}}] \label{thm:indexhzwinequva}
If $\widetilde{u}=(a,u)$ is a non-degenerate finite energy sphere with $\pi \circ du \not\equiv 0$, then 
\[
\mu(\widetilde{u}) \geq 2 \wind_{\pi}(\widetilde{u}) + 4 - 2 \# \Gamma^{\mathrm{even}} - \# \Gamma^{\mathrm{odd}},
\]
where $\Gamma^{\mathrm{even}}$ and $\Gamma^{\mathrm{odd}}$ are the sets of punctures whose asymptotic limit has even and odd Conley-Zehnder index, respectively. 
\end{theorem}

If $\widetilde{u}$ is a non-degenerate finite energy plane, then both its asymptotic winding number and  index are independent of the choice of the trivializations. By definition of $\alpha = \alpha(S)$ and Proposition \ref{prophwzii5.6} we have

\begin{theorem}[{\cite[Theorem 1.2]{HWZII}}] \label{HWZIIthm} If $\widetilde{u}$ be a non-degenerate finite energy plane,  then 
\[
\muCZ(\widetilde{u}) := \mu(\widetilde{u}) \geq 2.
\]
\end{theorem}

The following definition is due to Hryniewicz \cite{Hry12fast}.
\begin{definition} 
A non-degenerate finite energy plane $\widetilde{u} =(a,u) \colon \C \to \R \times M$ is said to be \emph{fast} if and only if   $u$ is an immersion transverse to the Reeb vector field, or equivalently, $\wind_{\infty}(\widetilde{u}) =1$.
\end{definition}

\noindent
From the definitions  we obtain the following inequality for a non-degenerate finite energy plane $\widetilde{u}$:
\[
1 \leq \wind_{\infty}(\widetilde{u}) \leq \left\lfloor \frac{\muCZ(\widetilde{u})}{2} \right\rfloor
\]
from which the following assertion is obtained.

\begin{lemma}\label{lem:muCZfast}
If a non-degenerate finite energy plane $\widetilde{u} =(a,u) \colon \C \to \R \times M$ has  $\muCZ(\widetilde{u}) \in \{2,3\}$, then it is fast.
\end{lemma}

We now consider a  finite energy half-sphere $\widetilde{u}_I$ associated to a non-degenerate invariant finite energy sphere $\widetilde{u}$. The $\pi$-winding number $\wind_{\pi}(\widetilde{u}_I)$ is defined to be
\[
\wind_{\pi}(\widetilde{u}_I):= \frac{1}{2} \wind_{\pi}(\widetilde{u}).
\]
For every  $ z \in \Gamma_{\mathrm{int}} \cap H$,  the asymptotic winding number and the Conley-Zehnder index with respect to a given trivialization are defined as above. Assume that $ z \in \Gamma_{\p}$. Let $P_z=(x_z,T_z)$ be the asymptotic limit of $\widetilde{u}$ at $z$ which is a symmetric periodic orbit,  and $C_z=(c_z,T_z/2)$ be its half-chord which is the asymptotic limit of $\widetilde{u}_I$ at $z$.  In view of Theorem \ref{thm:asymptoticbahinv}, we find the asymptotic eigenvalue $\eta_z^{\rho}$ and the asymptotic eigenfunction $\zeta_z^{\rho}$ of $\widetilde{u}_I$ at $z$. Let $\mathfrak{T}_z$ be a symmetric unitary trivialization   of $x_{T_z}^*\xi$.  As before, we define the asymptotic relative winding number of $\widetilde{u}_I$ at $ z \in \Gamma_{\p}$ with respect to $\mathfrak{T}_z$ as
\[
\wind_{\infty}^{\mathfrak{T}_z}(\widetilde{u}_I, z) := w(\eta_z, D_{C_z}^{\mathfrak{T}_z}),
\]
where the latter is the relative winding number of the eigenvalue   $\eta_z$ of the operator $D_{C_z}^{\mathfrak{T}_z}:= S_{P_z}^{\mathfrak{T}_z} |_{[0,1/2]}$ defined as in Subsection \ref{sec:ind2}. Likewise, the Robbin-Salamon index of $P_z$ with respect to $\mathfrak{T}_z$ is defined as
\[
\muRS^{\mathfrak{T}_z}(P_z) = 2\alpha_I( D_{C_z}^{\mathfrak{T}_z}) + \frac{1}{2},
\]
where $\alpha_I$ is defined as in Subsection \ref{sec:ind2}. Given any symmetric unitary trivialization $\mathfrak{T}$ of $u^*\xi$,  the \emph{relative asymptotic winding number of $\widetilde{u}_I$} is  defined as
\[
\mathrm{wind}_{\infty} (\widetilde{u}_I  ) := \sum_{z \in \Gamma_I^+} \mathrm{wind}_{\infty}^{\mathfrak{T} }(\widetilde{u}_I, z)- \sum_{z \in \Gamma_I^-} \mathrm{wind}_{\infty} ^{\mathfrak{T} }(\widetilde{u}_I, z),
\]
     and the \emph{index of a non-degenerate finite energy half-sphere $\widetilde{u}_I$} is   defined to be
\[
\mu (\widetilde{u}_I) := \sum_{z \in \Gamma^+_{\mathrm{int}} \cap H} \muCZ^{\mathfrak{T} }(P_z)  +  \sum_{z \in \Gamma^+_{\p}  } \muRS^{\mathfrak{T} }(P_z)   - \sum_{z \in \Gamma^-_{\mathrm{int}} \cap H} \muCZ^{\mathfrak{T} }(P_z)  -  \sum_{z \in \Gamma^-_{\p}  } \muRS^{\mathfrak{T} }(P_z)   .
\]
These are independent of   the choice of the trivialization.

\begin{lemma}[{\cite[Proposition 4.6 and Corollary 4.7]{FK16}}]\label{cor:lemaKF16} If   $\widetilde{u}=(a,u)$ is a non-degenerate invariant finite energy sphere with $\pi \circ du \not\equiv 0$, then 
\[
\mathrm{wind}_{\infty}( \widetilde{u})  = 2 \mathrm{wind}_{\infty}( \widetilde{u}_I ) .
\]
It follows that 
\[
\mathrm{wind}_{\infty}( \widetilde{u}_I ) = \mathrm{wind}_{\pi} ( \widetilde{u}_I) + 1 - \frac{ \# \Gamma_{\p}}{2} - \# ( \Gamma_{\mathrm{int}} \cap H).
\]
In particular, if $\widetilde{u}$ is an invariant finite energy plane, then $\wind_{\infty}(\widetilde{u}_I)  \geq 1/2$, and $\wind_{\infty}(\widetilde{u}_I) =1/2$ if and only if $\pi \circ du$ is non-vanishing.  
\end{lemma}

\begin{proposition}[{\cite[Proposition 4.9]{FK16}}] \label{prop:FK16index} If $\widetilde{u}=(a,u)$ is a non-degenerate invariant finite energy sphere with $\pi \circ du \not\equiv 0$, then 
\[
\mu(\widetilde{u}_I)  \geq 2\wind_{\pi}(\widetilde{u}_I) + 2 - \frac{1}{2} \# \Gamma_{\p} -  \# \Gamma_{\mathrm{int}}^{\mathrm{even}}   - \#( \Gamma_{\mathrm{int}}^{\mathrm{odd}} \cap H),
\]
where $\Gamma_{\mathrm{int}}^{\mathrm{even}}$ and $\Gamma_{\mathrm{int}}^{\mathrm{odd}}$ are the sets of  interior punctures whose asymptotic limit has even and odd Conley-Zehnder index, respectively.

\end{proposition}

If $\widetilde{u}$ is a non-degenerate invariant  finite energy plane, then both the relative  asymptotic winding number and   index of $\widetilde{u}_I$ are well-defined, and we have

\begin{theorem}[{\cite[Proposition 4.9]{FK16}}] \label{FK16thminv} If $\widetilde{u}$ is a non-degenerate invariant finite energy plane, then
\[
\muRS(\widetilde{u} ) :=\mu(\widetilde{u}_I) \geq \frac{3}{2}.
\]
\end{theorem}

\noindent
By definition    the following corollary is immediate.

\begin{corollary}\label{cor:muRSfast}
Assume that $\widetilde{u} =(a,u) \colon \C \to \R \times M$ is a non-degenerate invariant finite energy plane such that  $\muRS(\widetilde{u} ) = 3/2  $. Then it is fast.
\end{corollary}

 \section{Fredholm theory}\label{sec:Fredth}

Let $(M, \lambda, \rho)$ be a closed real   contact  three-manifold.
Given two $\rho$-invariant positive smooth functions $f,g$ on $M$ with $f<g$, we choose a smooth function $h \colon \R \times M \to (0, \infty)$ satisfying
\begin{align}
\nonumber & f(x) \leq h(a, x ) \leq g(x ) \quad \forall (a,x ) \in \R \times M, \\
\nonumber & h(a,  x ) = \begin{cases} f(x) & \text{ if } a \leq -2, \\ g(x) & \text{ if } a \geq 2, \end{cases}\\
\nonumber& \frac{\p h}{\p a}  \geq 0 \;\; \text{ on}\;\; \R \times M,\\
 \label{eq:positive} & \frac{\p h}{\p a} \geq \sigma >0 \;\; \text{on }\; [-1, 1] \times M \;\; \text{for some constant } \sigma >0,\\
\nonumber & h(a, \rho(x)) = h(a, x ) \quad \forall (a, x ) \in \R \times M.
\end{align}
 Condition \eqref{eq:positive} tells us that the two-form $\Omega= d(h \lambda)$ is a symplectic form on $[-1 ,  1] \times M$. 
For every $a \in \R$, we define a  $\rho$-anti-invariant contact form  $\lambda_a := h(a , \cdot) \lambda$  on $M$ defining the contact structure $ \ker \lambda =\xi   = \ker \lambda_a $.
The associated Reeb vector field is denoted by $R_a$.

 Let $ a\mapsto J_a$ be a smooth family of $d\lambda_a$-compatible and $\rho |_{\xi}$-anti-invariant almost complex structures on $\xi$ satisfying that
\[
J_a= \begin{cases} J_g & \text{ if } a \geq 2 , \\ J_f & \text{ if } a \leq -2, \end{cases}
\]
for some fixed $J_g \in \mathcal{J}_{\rho}(\xi, d(g\lambda) |_{\xi})$ and $J_f  \in \mathcal{J}_{\rho}(\xi, d(f\lambda) |_{\xi})$.
This gives rise to the $\widetilde{\rho}$-anti-invariant almost complex structure $\overline{J}$ on $\R \times M$ defined by
\[
\overline{J}(a,x)(h, k) = ( -\lambda_a(x)(k), J_a(x) \pi_a (k) + hR_a(x)),
\]
where $\widetilde{\rho} = \mathrm{Id}_{\R} \times \rho$  and $\pi_a \colon TM \to \xi$ is the projection along $R_a$.
We consider a  $\widetilde{\rho}$-anti-invariant almost complex structure $\hat{J}$ on $\R \times M$ 
which coincides with $\overline{J}$ on $(\R \setminus [- 1,  1] ) \times M$ and which is $\Omega$-compatible on $[- 1,  1] \times M$.
 The space of such almost complex structures $\hat{J}$ on $\R \times M$,  denoted by $\mathcal{J}_{\widetilde{\rho}}(g\lambda , J_g, f\lambda , J_f )$, is  non-empty and  contractible in the $C^{\infty}$-topology.

 Let  $\Gamma \subset S^2$ be a finite set of points in $S^2$ such that $I(\Gamma) = \Gamma$. Recall that $I(z) = \overline{z}$, $z \in S^2$.  Let $j$ be an $I$-anti-invariant complex structure on $S^2 \setminus \Gamma$.  We shall study smooth maps $\widetilde{u} \colon S^2 \setminus \Gamma \to \R \times M$ satisfying
\begin{equation}\label{eq:generalfiepdddd}
\hat{J}(\widetilde{u}) \circ d\widetilde{u} =  d\widetilde{u} \circ j  \quad \text{ and } \quad \widetilde{\rho} \circ \widetilde{u} \circ I = \widetilde{u} \;\; \text{ on }\;  S^2 \setminus \Gamma
\end{equation}
with finite energy $E(\widetilde{u}) < \infty$ defined as follows. Abbreviate by $\Xi$ the set of non-decreasing smooth functions $\varphi \colon \R \to [0,1]$ such that $\varphi \equiv \frac{1}{2}$ on $[- 1, 1]$. For $\varphi \in \Xi$, we introduce a one-form $\lambda_{\varphi}$ on $\R \times M$ defined as
\[
\lambda_{\varphi}(a, x) (h, k) = \varphi(a) \lambda_a(x) k.
\]
The energy $E(\widetilde{u})$ is then defined as 
\[
E(\widetilde{u}):= \sup_{\varphi \in \Xi} \int_{S^2 \setminus \Gamma} \widetilde{u}^* d\lambda_{\varphi}.
\]
 Such maps are called \emph{generalized invariant finite energy spheres}. If $\Gamma = \{ + \infty \}$, then they are referred to as  \emph{generalized invariant finite energy planes}.

\subsection{The implicit function theorem}

  For every $\hat{J}$ not necessarily anti-invariant, Hofer, Wysocki, and Zehnder proved the implicit function theorem for embedded generalized finite energy planes, see \cite[Theorem 1.6]{HWZIII}. An important feature is that  no genericity assumption on almost complex structures is assumed, due to automatic transversality.

\begin{theorem} \label{thm:iFTfheotem} Fix  $\hat{J} \in    \mathcal{J}_{\widetilde{\rho}}(g\lambda , J_g, f\lambda , J_f )$ and let $\widetilde{u} =(a,u) \colon \C \to \R \times M$ be an embedded  generalized invariant finite energy  $\hat{J}$-holomorphic plane with a non-degenerate and simply covered asymptotic limit $P=(x,T)$. Assume that    $\muCZ( \widetilde{u})=3 $.     Then for all $\epsilon>0$ small enough, there exists an embedded generalized   invariant  finite energy   $\hat{J}$-holomorphic plane  $\widetilde{w} = (c,w)$ asymptotic to $P$  such that 
\[
\inf \{ c(z) \colon z \in \C \} = \inf\{a(z) \colon z \in \C \}- \epsilon   .
\]
\end{theorem}
\begin{proof} An application of the implicit function theorem due to Hofer, Wysocki, and Zehnder provides us with a smooth embedding $\widetilde{\Phi}=(A, \Phi) \colon \mathbb{B} \times \C \to \R \times M$ such that $ \mathbb{B} \subset \C$ is a small open ball centered at $0$, and for each $b \in  \mathbb{B}$, the map $\widetilde{\Phi}(b, \cdot)$ is an embedded generalized finite energy $\hat{J}$-holomorphic plane asymptotic to $P$. We in particular have $\widetilde{\Phi}(0, \cdot) = \widetilde{u}$.   This embedding foliates an open neighborhood of $\widetilde{u}(\C)$ in $\R \times M$ and is unique up to   reparametrization. We introduce the  embedding $\widetilde{\Phi}_{\rho} \colon  \mathbb{B} \times \C \to \R \times M$ defined by 
\[
\widetilde{\Phi}_{\rho} (b, z) := \widetilde{\rho} \circ \widetilde{\Phi} (b,  I(z)), \;\; (b,z) \in  \mathbb{B} \times \C .
\]
 Because $\hat{J}$ is $\widetilde{\rho}$-anti-invariant, for each $b \in  \mathbb{B}$, the map $\widetilde{\Phi}_{\rho}(b, \cdot)$ is also an embedded generalized finite energy $\hat{J}$-holomorphic plane asymptotic to $P$.  Moreover, since $\widetilde{u}$ is invariant, we have $\widetilde{\Phi}_{\rho} (0, \cdot) = \widetilde{u} =\widetilde{\Phi}  (0, \cdot)$.

Choose $ x_0 \in \im(u) \cap \mathrm{Fix}(\rho)$ and $z_0 \in \C$ such that $ u(z_0)=x_0$. Set $a_0:= a(z_0)$. Since $\widetilde{u}$ is invariant,   $ a(\overline{z_0}) = a_0$. For $\epsilon  >0$   small enough, there exists $b _0 = b_0(z_0, \epsilon) \in \mathbb{B}$ such that $\widetilde{\Phi}(b_0, \C)$ contains the point $(a_0-\epsilon, x_0) \in \R \times M$ because $\widetilde{\Phi}( \mathbb{B}, \C)$ is an open neighborhood of $\widetilde{u}(\C)$ in $\R \times M$.  In particular, $a_0 - \epsilon \in A(b_0, \C)$ and $x_0 \in \Phi(b_0, \C)$.  We observe that $(a_0 - \epsilon, x_0 ) \in \widetilde{\Phi}_{\rho}(b_0, \C)$. Indeed, it follows from that  $\widetilde{\Phi}_{\rho}(b_0, \C) = (A(b_0, \C), \rho \circ \Phi(b_0, \C))$ and   $x_0 \in \mathrm{Fix}(\rho)$. Consequently, $\widetilde{\Phi} (b_0, \C) \cap\widetilde{\Phi}_{\rho}(b_0, \C) \neq \emptyset$. The mutual disjointness now tells us that  $\widetilde{\Phi} (b_0, \C)=\widetilde{\Phi}_{\rho}(b_0, \C)$. By reparametrizing if necessary, we conclude that $\widetilde{\Phi}(b_0, \cdot)=:\widetilde{w} $ is an embedded generalized invariant finite energy $\hat{J}$-holomorphic plane asymptotic to $P$.    Recall from above that $\widetilde{u}(z_0) = (a_0, x_0)$ and $\widetilde{w}(z_1)= (a_0- \epsilon, x_0)$ for some $ z_1 \in \C$. For $\epsilon >0$ small enough we have $\epsilon_* \widetilde{w}(\C)  \in \widetilde{\Phi}( \mathbb{B}, \C)$, where $\epsilon_* \widetilde{w} := (c+\epsilon, w)$. It follows again from the mutual disjointness that $\epsilon_* \widetilde{w}(\C) = \widetilde{u}(\C)$ and hence 
\[
\inf \{ c(z) \colon z \in \C \} = \inf\{a(z) \colon z \in \C \} - \epsilon < \inf \{ a(z) \colon z \in \C \}.
\]
 Since $\epsilon>0$ was arbitrary, this finishes the proof of the theorem. 
\end{proof}

\subsection{Transversality} 

 Let $P_1^+, \ldots, P_k^+$  and  $Q_1^+, \ldots, Q_{\ell}^+$ be symmetric and  non-symmetric periodic orbits of $ R_{g\lambda }$, respectively. Similarly, let    $P_1^-, \ldots, P_m^-$ and   $Q_1^-, \ldots, Q_n^-$ be symmetric and   non-symmetric periodic orbits of $ R_{f\lambda }$, respectively.  Assume that   they are all  non-degenerate.

\begin{theorem}\label{thm:genfinite} There exists a dense subset $\mathcal{J}_{\mathrm{reg}} \subset  \mathcal{J}_{\widetilde{\rho}}(g\lambda , J_g, f\lambda , J_f )$ such that the following assertion holds. Pick  $\hat{J} \in \mathcal{J}_{\mathrm{reg}}$ and let $\widetilde{u}_I$ be the half-sphere of  a somewhere injective immersed generalized invariant finite energy $\hat{J}$-holomorphic sphere $\widetilde{u}=(a,u) \colon  S^2 \setminus \Gamma  \to   \R \times M $. Assume that $\widetilde{u}_I$   has   $k+\ell$ positive punctures asymptotic to the given periodic orbits of $R_{g\lambda }$ and   $m+n$ negative punctures asymptotic to the given periodic orbits  $R_{f\lambda }$, then its Fredholm index, defined by
\begin{align*}
\mathrm{Fred}(\widetilde{u}_I):= &   \sum_{j=1}^k \muRS(P^+_j) + \sum_{i=1}^{\ell} \muCZ(Q^+_i)-  \sum_{j=1}^m \muRS(P^-_j) - \sum_{i=1}^{n} \muCZ(Q^-_i) \\ &+ \frac{k+m}{2} +\ell+n-1 ,
\end{align*}
is non-negative.
The indices of the asymptotic orbits are computed using a symmetric unitary trivialization of $u^*\xi$. 
\end{theorem}
\noindent
The proof is based on standard arguments given in \cite{Bour02, Dra04, HWZIII}. See    \cite[Theorem 3.2]{Zhou19} for a complete proof.  We comment that, as Schwarz pointed out, one is not able to compute the Fredholm index via spectral flow. Instead, it can be  computed using the linear gluing operation and the Riemann-Roch theorem, see \cite{Sch95}.

\section{Proof of Theorems \ref{ThmA1},   \ref{ThmB1}, and \ref{Thmnew}}\label{se:profthemBF}

This section is devoted to prove the theorems on necessary and sufficient conditions for a symmetric periodic orbit to bound an invariant disk-like global surface of section.

We begin with considering almost complex structures.  Let $(M, \omega)$ be a symplectic manifold equipped with an anti-symplectic involution $\rho$ and a $p$-periodic symplectomorphism $\sigma$ satisfying condition \eqref{eq:action}. Recall that an $\omega$-compatible almost complex structure $J \in \mathcal{J}(M, \omega)$  is said to be $\rho$-anti-invariant if $D\rho^{-1} \circ J \circ D \rho = -J$. The space of such $J$'s is denoted by $\mathcal{J}_{\rho}(M, \omega)$. It is  non-empty and contractible in the $C^{\infty}$-topology, see    \cite[Proposition 2.4.2]{Evers}. Actually, it is a convex set.

An  $\omega$-compatible almost complex structure $J \in \mathcal{J}(M, \omega)$ is called \emph{$(\rho, \sigma) $-anti-invariant} if 
\begin{equation}\label{eq:jconju}
D\rho_j ^{-1}\circ J \circ D\rho_j = - J, \quad \forall j,
\end{equation}
where $\rho_j := \sigma^j \circ \rho$. 
Note that $J \in \mathcal{J}_{\rho}(M, \omega)$ is $(\rho, \sigma)$-anti-invariant if and only if it is  $\sigma$-invariant, meaning that $D\sigma ^{-1}\circ J \circ D \sigma = J$.  The space of $(\rho,\sigma)$-anti-invariant and  $\omega$-compatible almost complex structures is denoted by $\mathcal{J}_{\rho,\sigma} (M, \omega)$.

\begin{lemma} \label{spaceofacs} The space   $\mathcal{J}_{\rho, \sigma}(M, \omega)$ is     non-empty and contractible in the $C^{\infty} $-topology.
\end{lemma}
\begin{proof} Pick any Riemannian metric $g_0$ on $M$ and define
\[
g:= \frac{\sum_{j=0}^{p-1} (\sigma^j)^*( g_0 + \rho^* g_0) }{2p}.
\]
In view of \eqref{eq:action} this Riemannian metric is invariant under both $\rho$ and $\sigma$.     Consider the smooth section $A$ of the automorphism bundle $\mathrm{Aut}(TM) \to M$ defined by $\omega_x(X, Y) = g_x(A_xX, Y)$ for all $x \in M$ and for all $X,Y \in T_x M$. In view  of the invariances of $g$, the section $A$ satisfies that $D\rho ^{-1} \circ A  \circ D \rho  = -A$ and $D\sigma^{-1} \circ A \circ D\sigma= A$.  For each $ x\in M$, the automorphism $-A^2_x$ is symmetric and positive-definite with respect to $g$. Denoting by $Q  $ the associated smooth map defined as $Q_x := \sqrt{-A_x^2}$ for all $ x \in M$, this implies that $J:= A Q^{-1}$  is $\omega$-compatible almost complex structure on $M$. We claim that  $J \in \mathcal{J}_{\rho,  \sigma}(M, \omega). $   Indeed,  $D\rho$ and $D\sigma $  preserve  the eigenspaces of $-A^2$, and hence   $Q$ commutes with them.   This proves the claim and  shows that the space $\mathcal{J}_{\rho,  \sigma}(M, \omega)$  is non-empty.   By a standard argument, it is convex. For more details, see for example  \cite[Proposition 1.3.10]{Geiges08book}. This finishes the proof of the lemma.
 \end{proof}

The following statement will be crucial  in the proofs of the theorems.

\begin{theorem}\label{thm:automatrix} \label{thm:cobordism}
Let $P= (x,T)$ be a  simply covered periodic orbit on a non-degenerate tight three-sphere $(S^3, \lambda)$.  Assume that   $\muCZ(P) \geq 3$ and it is linked to every periodic orbit of $\muCZ=2$. 
 Then the following assertion holds: Assume that there exists  a $d\lambda|_{\xi}$-compatible almost complex structure $J$ such that there exists an embedded fast finite energy $\widetilde{J}$-holomorphic plane asymptotic to $P$. Then for every $d\lambda|_{\xi}$-compatible almost complex structure $J',$   there   exists an embedded fast finite energy $\widetilde{J'}$-holomorphic plane asymptotic to $P$.
\end{theorem}
\begin{proof} 
We fix any $d\lambda|_{\xi}$-compatible almost complex structure $J$, and let $\widetilde{u} = (a,u)$ be an embedded fast finite energy $\widetilde{J}$-holomorphic plane asymptotic to $P$ with $\muCZ(P) \geq 3$. Automatic transversality for embedded fast finite energy planes, see   \cite[Section 15.5]{FvK18book} and  \cite{Hry12fast}, tells us that the moduli space $\mathcal{M}_{\mathrm{fast}}(P, J)$ of (unparametrized) fast finite energy $\widetilde{J}$-holomorphic planes asymptotic to $P$   is a smooth manifold   of dimension two. Moreover, every element is embedded. Recall that the SFT-like almost complex structure $\widetilde{J}$ is $\R$-invariant from which we obtain the free $\R$-action on $\mathcal{M}_{\mathrm{fast}}(P, J)$ given by
\[
r_* [\widetilde{u}] = r_*[ (a,u)] = [(a+r, u)], \quad r \in \R, \;\; [ \widetilde{u}] \in \mathcal{M}_{\mathrm{fast}}(P,J).
\]
It is not hard to see that the quotient $\mathcal{M}_{\mathrm{fast}}(P,J)/\R$ is also a smooth manifold. We claim that it is compact.  Choose any sequence $\{\widetilde{u}_k \}$ such that $[\widetilde{u}_k] \in \mathcal{M}_{\mathrm{fast}}(P,J)$ for all $k \in \N$. The SFT-compactness theorem \cite{SFT03, HWZ95I, HWZcorr} tells us that up to   subsequence and   reparametrization, the sequence converges in $C^{\infty}_{\mathrm{loc}}$-topology to an embedded finite energy spheres $\widetilde{u} =(a,u)\colon \C \setminus \Gamma\to \R \times S^3$ such that it tends asymptotically to $P$ at $+\infty$,    $\Gamma$ consists of negative punctures, and it is not a trivial orbit cylinder over $P$.   Assume by contradiction that $\Gamma \neq \emptyset$. 
 Since $\muCZ(P) \geq 3$,  an argument using Fredholm theory and bubbling off analysis shows that there exists a periodic orbit which is not linked to $P$ and has the Conley-Zehnder index equal to 2, see  \cite[Section 4]{HS11}.  Because of the hypothesis, we conclude that $\Gamma = \emptyset$, and hence $\widetilde{u}$ is an embedded finite energy plane.  It is not hard to see that it is indeed fast. This proves the claim.

The remaining proof is exactly as in  \cite[Theorem 17.1.3]{FvK18book}. We include it here for readers' convenience. 
As the space of $d\lambda|_{\xi}$-compatible almost complex structures is non-empty and contractible, we can choose a smooth path $\{J_t\}_{t \in [0,1]}$ of $d\lambda|_{\xi}$-compatible almost complex structures such that $J_0 = J$ and $J_1 = J'$. Consider the moduli space
\[
\mathcal{N}= \{ (t, [\widetilde{u}]) \in [0,1] \times(  \mathcal{M}_{\mathrm{fast}}(P, J_t) / \R) \}.
\]
The preceding discussion tells us that    it  is a two-dimensional compact manifold with boundary 
\[
\p\mathcal{N}=  \left( \{ 0 \} \times  \mathcal{M}_{\mathrm{fast}}(P, J ) / \R \right)  \cup  \left( \{ 1 \} \times  \mathcal{M}_{\mathrm{fast}}(P, J ') / \R \right). 
\]
As the map $\mathcal{N} \to [0,1]$, $(t, [\widetilde{u}])\mapsto t$ is a proper submersion, the Ehresmann fibration theorem shows that it is a locally trivial fibration. In particular, we obtain the diffeomorphism
\[
 \mathcal{M}_{\mathrm{fast}}(P, J ) / \R \cong  \mathcal{M}_{\mathrm{fast}}(P, J') / \R .
 \]
 By the hypothesis, the left-hand side is non-empty. This completes the proof of the theorem. 
\end{proof}

We are now in position to prove the theorems.

\medskip

\noindent
\emph{Proof of Theorem \ref{ThmA1}}. 
 Suppose first that a simply covered symmetric periodic orbit $P$ bounds an invariant disk-like global surface of section. Then in view of \cite[Section 6]{HS11} we see that $\muCZ(P) \geq3$ and $sl(P)=-1$. It is obvious that $P$ is unknotted. Moreover, by definition of a global surface of section, it is linked to every periodic orbit different from $P$.

We now prove sufficiency. We follow the argument due to Frauenfelder and van Koert  in  \cite[Corollary 17.1.6]{FvK18book} closely.    Suppose that an unknotted simply covered symmetric periodic orbit $P$ has $sl(P)=-1$ and $\muCZ(P) \geq 3$. The bubbling off analysis due to Hryniewicz and Salom\~ao \cite{HS11} shows that if $P$ satisfies  the  hypothesis  in the statement, then there exists a $d\lambda|_{\xi}$-compatible almost complex structure $J'$ for which we have an embedded fast finite energy $\widetilde{J'}$-holomorphic plane asymptotic to $P$, which is not necessarily invariant. Choose any $J \in \mathcal{J}_{\rho}$. By Theorem \ref{thm:cobordism}, we find that $\mathcal{M}_{\mathrm{fast}}(P, J )/\R $ is  a non-empty compact one-dimensional manifold. 
 A standard argument shows that  
 \[
  \mathcal{M}_{\mathrm{fast}}(P, J )/\R \cong S^1
  \]
   and 
 \[
S^3 \setminus x(\R) = \bigcup_{[ \widetilde{u}]=[(a,u)] \in   \mathcal{M}_{\mathrm{fast}}(P, J )/\R} u(\C) 
\]
is a $\lambda$-adapted open book decomposition with disk-like pages.    See for example \cite[Theorem 17.1.1]{FvK18book}.

Abbreviate by $\widehat{\mathcal{M}}_{\mathrm{fast}}(P, J)$ the moduli space of parametrized  fast finite energy planes asymptotic to $P$ such that $\widehat{\mathcal{M}}_{\mathrm{fast}}(P, J) \to    {\mathcal{M}}_{\mathrm{fast}}(P, J)$, $\widetilde{u} \mapsto [ \widetilde{u}]$. Recall that every element  in $\widehat{\mathcal{M}}_{\mathrm{fast}}(P, J)$ is embedded. As $J$ is $\rho$-anti-invariant, the involution $\rho$ induces the involution $\rho_* \colon  \widehat{\mathcal{M}}_{\mathrm{fast}}(P, J) \to \widehat{\mathcal{M}}_{\mathrm{fast}}(P, J)$  given by
\[
\rho_* (\widetilde{u}):= \widetilde{\rho} \circ \widetilde{u} \circ I, \quad \widetilde{u} \in  \widehat{\mathcal{M}}_{\mathrm{fast}}(P, J) . 
\]
Recall that $\mathrm{Fix}(\rho) \cong S^1$. We denote by $\mathcal{L}_1$ and $\mathcal{L}_2$ the connected components of  $ \mathrm{Fix}(\rho) \setminus x(\R)$ which are diffeomorphic to an open interval. Choose $q \in \mathcal{L}_1$.  The preceding argument shows that there exists  an embedded fast finite energy $\widetilde{J}$-holomorphic plane $\widetilde{u}=(a,u)$ asymptotic to $P$ such that $q \in u(\C)$. Due to the construction of the involution $\rho_*$, we have $\rho_* (\widetilde{u}) = (a \circ I , \rho \circ u \circ I)$ is also an embedded fast finite energy $\widetilde{J}$-holomorphic plane   asymptotic to $P$ such that $q \in \rho ( u  (\C))$. We in particular find that 
\[
u(\C) \cap \rho \left(  u(\C) \right) \neq \emptyset 
\]
from which we conclude that $u(\C) = \rho( u(\C) )$, i.e.\ the closed embedded disk $\overline{u(\C)}$ is an invariant disk-like global surface of section. Since the restriction of $\rho$ to $\overline{u(\C)}$ is orientation-reserving,  {K}er\'{e}kj\'{a}rt\'{o}'s theorem     tells us that $ \mathcal{L}_1 \subset \mathrm{Fix}(\rho|_{\overline{u(\C)}})$. In the same way, we find an embedded fast finite energy $\widetilde{J}$-holomorphic plane  $\widetilde{v}=(b,v)$ asymptotic to $P$ such that $\overline{v(\C)}$ is an invariant disk-like global surface of section and $\mathcal{L}_2 \subset \mathrm{Fix}(\rho|_{\overline{v(\C)}})$.  If $\lambda$ is dynamically convex, a limiting argument as in \cite{HWZ98convex, Hry14system} allows us to drop the non-degeneracy assumption. This finishes the proof of the theorem.  \hfill $\square$

\bigskip

\noindent
\emph{Proof of Theorem \ref{ThmB1}.}   Fix  $J \in \mathcal{J}_{\rho,  \sigma}(\xi, d\lambda|_{\xi})$. In view of Lemma  \ref{spaceofacs} and the discussion preceding it, by arguing as in the above proof, we  find  the involutions    $  (\rho_j) _* \colon  \widehat{\mathcal{M}}_{\mathrm{fast}}(P, J) \to \widehat{\mathcal{M}}_{\mathrm{fast}}(P, J)$ for $j=0, \ldots, p-1$. Choose $q \in \mathrm{Fix}(\rho) \cap \mathrm{Fix}(\sigma)$ and   find an embedded fast finite energy $\widetilde{J}$-holomorphic plane $\widetilde{u}=(a,u)$ asymptotic to $P$ such that $ q\in u(\C)$.   As above, this implies that 
\[
   u(\C) = \rho  ( u(\C)) = \sigma( u(\C)).
 \]
 In other words, the disk $u(\C)$   a  $(\rho, \sigma)$-invariant disk-like global surface  of section.

 We now argue as in \cite[Remark 6.6]{FK16}. Recall that the first return map $\psi$ associated to $\mathfrak{D}:= u(\C)$ is defined by  $\psi(x) = \phi_R^{\tau(x)}(x)$, where $\tau(x)$ is the first return time of $ x \in \mathring{\mathfrak{D}}$. Define a diffeomorphism $\Phi \colon \R / \Z \times \mathfrak{D} \to S^3 \setminus P  $ as
\[
\Phi(t , x):= \phi_R^{ t \tau(x) } (x) ,
\]
 such that for each $t \in \R / \Z$, the closure $\overline{\Phi(t, \mathfrak{D})}$   is a disk-like global surface of section spanned by $P$. Since $\rho _j  \circ \phi_R^t  \circ \rho_j  = \phi_R^{-t}$ for all $j$, we find that 
 \[
 \tau_x = \tau_{\rho_j  \circ \phi_R^{\tau(x)}(x)} \quad \forall j.
 \] 
   This open book decomposition is then $(\rho,\sigma)$-symmetric in the sense that
\[
\Phi (1-t, \mathfrak{D}) = \rho _j\circ \Phi(t, \mathfrak{D}), \quad \forall t \in \R /\Z , \; \forall j=0,1,\ldots, p-1.
\]
We  find that $P$ bounds two $(\rho, \sigma)$-invariant global surfaces of section $\mathfrak{D} = \Phi(0, \mathfrak{D})$  and $\mathfrak{D}':=\Phi( 1/2, \mathfrak{D})$.
Each of them contains the only one  common fixed point of $\rho$ and $\sigma$, i.e.\   $\# ( \mathrm{Fix}(\rho ) \cap \mathrm{Fix}(\sigma) ) =   2$. This completes the proof of the theorem.   \hfill $\square$

\bigskip

\noindent
\emph{Proof of Theorem \ref{Thmnew}.} Let $P=(x,T)$ be a simply covered symmetric periodic orbit on $(L(p,q), \overline{\lambda}, \overline{\rho})$, where $\overline{\lambda}$ is dynamically convex.

If   $P$ bounds an invariant $p$-rational disk-like global surface of section, by definition  it is $p$-unknotted. An application of \cite[Lemma 3.10]{HLS15} shows $sl(P)=-1/p$.

 Assume that $P$ is $p$-unknotted and has $sl(P)=-1/p$. Choose $J \in \mathcal{J}_{\overline{\rho}}( \overline{\xi}, d \overline{\lambda}|_{\overline{\xi}})$. Hryniewicz and Salom\~ao show in \cite[Section 4]{HS16elliptic} that $\mathcal{M}_{\mathrm{fast}}(P, J) /\R \cong S^1$, and $L(p,q)$ admits a $\overline{\lambda}$-adapted $p$-rational open book decomposition with disk-like pages whose   binding is  $P$. We fix $q \in \mathrm{Fix}(\overline{\rho}) \setminus x(\R)$ and argue as before to obtain that there exists an invariant $p$-rational disk-like global surface of section containing $q$. This finishes the proof of the theorem. \hfill $\square$

\section{Proof of Theorems  \ref{ThmA2} and \ref{ThmB2}}\label{sec:thmA}


We give a quick  review  of    the example provided by Hofer, Wysocki, and Zehnder in \cite[Lemma 1.6]{HWZ95I} and \cite[Section 4]{HWZ98convex}.

Consider the Hamiltonian $H \colon \C^2 \to \R$, defined by
\[
H(z_1, z_2) := \frac{ \lvert z_1 \rvert^2}{r_1^2} +  \frac{ \lvert z_2  \rvert^2}{r_2^2}, 
\]
where $0<r_1<r_2$ and  $r_2^2 / r_1^2 \notin \mathbb{Q}$. It is invariant under  the family of the  exact anti-symplectic involutions 
\[
\rho_{\theta} (z_1, z_2) =  ( e^{ i \theta_1}\overline{z _1}, e^{i \theta_2}\overline{z_2} ),   \quad  \theta=(\theta_1, \theta_2) \in \R^2.
\]
  The (regular) level set  $E:= H^{-1}(1)$ is referred to as an \emph{irrational ellipsoid}. The restriction  of each $\rho _\theta$ to $E$ is an anti-contact involution.  By abuse of notation, we denote the restrictions again by the same letters.  For sake of convenience,  in what follows we concentrate only on $\rho_\theta$ with $\theta_1 = \theta_2=0$ and write $\rho =\rho_\theta$. An analogue assertion holds also for the other involutions. The fixed point set of $\rho$ is the ellipse
\[
\{ (x_1, 0,  x_2, 0 ) \in E \mid \frac{x_1^2}{r_1^2} + \frac{x_2^2}{r_2^2} =1 \} .  
\]
 We define the Hamiltonian vector field $X_H$   as $\iota_{X_H}\omega_0 = - dH$, where $\omega_0 = dx_1 \wedge dy_1 + dx_2 \wedge dy_2$ is the standard symplectic form. The Liouville vector field
\[
Y= \frac{1}{2} \left( x_1 \p_{x_1} + y_1 \p_{y_1} + x_2 \p_{x_2} + y_2 \p_{y_2} \right)
\]
is transverse to $E$ so that the Liouville one-form $\lambda_0 = \iota_{Y} \omega_0  $ restricts to  the  contact form $\alpha$ on $E$.    It is straightforward to check that $R  = X_H$, where $R= R_{\alpha}$ is the Reeb vector field associated to $\alpha$, and that $\rho  ^* R = -R$.   
The Reeb flow is given by
\[
\phi_R^t(z_1, z_2) = \left( e^{ (2 / r_1^2)it}z_1 , e^{ (2 / r_2^2)it}z_2 \right).
\]
The condition $r_1^2 / r_2^2 \notin \mathbb{Q}$ tells us that there exist exactly two simply covered periodic orbits, up to reparametrization,
\[
\gamma_1(t) =( r_1e^{ (2 / r_1^2)it}, 0) \quad \text{ and } \quad \gamma_2(t) =(0,  r_2 e^{ (2 / r_2^2)it}) 
\]
having the minimal periods $T_1 = \pi r_1^2$ and $T_2 = \pi r_2^2$, respectively.  
 As usual, we abbreviate $P_j = (\gamma_j, T_j), j=1,2$.   For each $ j=1,2$ the half-chord of $P_j$ with respect to $\rho$ is written as   $C_j   = (c_j, T_j/2)$, where $ c_j(t) := \gamma_j(t) |_{[0,T_j/2]}$.  These orbits have the following properties:
\begin{enumerate}
\item    $P_1$ and $P_2$ are   symmetric and unknotted and have $sl(P_1)= sl(P_2)=-1$. 
\item   $P_1$ and $P_2$ and all their iterates are non-degenerate and elliptic.
\item   $C_1$ and $  C_2$ and all their iterates are non-degenerate.
\item $\muCZ(P_1) = 2\muCZ(C_1 )=3 $ and $\muCZ(P_2) = 2\muRS(C_2)= 3+ 2\left \lfloor r_2^2 / r_1^2 \right \rfloor \geq 5.$
\end{enumerate}
  
  Set $S^3 = \{ (z_1, z_2) \in \C^2 \mid \lvert z_1 \rvert^2  + \lvert z_2 \rvert^2=1 \}$. 
 Define the diffeomorphism $\Psi \colon S^3 \to E$ as $z \mapsto \sqrt{f_E(z)}z$, where  $f_E \colon S^3 \to \R$ is given by $f_E(z_1, z_2):= 1/H (z_1, z_2) $.  It is straightforward to check that $\Psi^* ( \lambda_0 |_E) = f_E \lambda_0 |_{S^3} =:\lambda_E |_{S^3}$ and $\Psi^* \rho = \rho$. Therefore,  the Hamiltonian vector field $X_H$ on $E$ is equivalent to the Reeb vector field $R_E$ of the contact form $\lambda_E |_{S^3}$ on $S^3$.  In particular, the $\lambda_E$-dynamics on $S^3$ satisfies all the properties of the $X_H$-dynamics on $E$.  By abuse of notation we denote the corresponding periodic orbits and chords on $S^3$ again by $P_1, P_2, C_1$, and $C_2$. Note that $\im (\gamma_1 )= S^1 \times \{ 0 \}$ and  $ \im (\gamma_2) = \{ 0 \} \times S^1$.

By  the proof of Theorem  \ref{ThmA1} we obtain the following.
\begin{proposition}\label{prop:ellipexit} 
For every $J \in \mathcal{J}_{\rho }$, there exist two embedded invariant  fast  finite energy $\widetilde{J}$-holomorphic planes asymptotic to $P_1$ whose  projections to $S^3$ are invariant disk-like global surfaces of section.  An analogous  statement holds for $P_2$. 
\end{proposition}

Recall that  for coprime integers $p \geq q \geq 1$, $(S^3, \lambda_E)$ admits  the   $p$-periodic strict contactomorphism
\[
g_{p,q}(z_1, z_2) = (e^{ 2 \pi i /p}z_1, e^{ 2 \pi i q/p}z_2)
\]
satisfying that $g_{p,q} \circ \rho \circ g_{p,q} = \rho$. It generates a free $\Z_p$-action on $S^3$ whose associated orbit space is the lens space $L(p,q)$. The contact form $\lambda_E$ and the anti-contact involution $\rho$ descend  to a contact form $\overline{\lambda_E}$  and an anti-contact involution $\overline{\rho}$ on $L(p,q)$, respectively.  Let $\pi_{p,q} \colon S^3 \to L(p,q)$ be the quotient projection.  The Reeb vector field $\overline{R_E}$ of $\overline{\lambda_E}$ has exactly two simply covered periodic orbits $ \overline{P_1}= (\pi _{p,q}\circ \gamma_1, T_1/p)$ and $\overline{P_2} = (\pi _{p,q}\circ \gamma_2, T_2 /p)$.  They are $\overline{\rho}$-symmetric and $p$-unknotted and has $sl( \overline{P_1}) = sl(\overline{P_2}) = -1/p$.

Fix $y \in \mathrm{Fix}(\overline{\rho}) \setminus \pi_{p,q}(S^1 \times \{ 0 \} )$ and $J \in \mathcal{J}_{\overline{\rho}}$. In view of \cite[Proposition 4.17]{HS16elliptic} we find an embedded fast finite energy $\widetilde{J}$-holomorphic plane $\widetilde{u}=(a,u)$ asymptotic to the $p$-th iterate  $\overline{P _1}^p  $ of $\overline{P_1}$ such that $y \in u(\C)$. Since $J$ is anti-invariant, the pseudoholomorphic map $\widetilde{\overline{\rho}} \circ \widetilde{u} \circ I$ has the same properties. By the choice of $y$, we in particular have $y \in    \overline{\rho}  \circ u(\C)$. Then \cite[Theorem 4.18]{HS16elliptic} tells us that $\widetilde{u}(\C) = \widetilde{\overline{\rho}}\circ \widetilde{u} \circ I(\C)$.  We have proven the following.

\begin{proposition}\label{prop:rp3ellipsoid} 
For each $J \in \mathcal{J}_{\overline{\rho}}$, there exists an  embedded invariant  finite energy $\widetilde{J}$-holomorphic plane  in $L(p,q)$ asymptotic to $\overline{P_1}^p $. 
\end{proposition}

 \subsection{The existence of a symmetric periodic orbit}
 
 This subsection is devoted to prove Theorem \ref{ThmA2}.

 Assume that    a non-degenerate real tight three-sphere $(S^3, \lambda, \rho )$ satisfies   $\mathscr{P}_* = \emptyset$. Let $f \colon S^3 \to (0, \infty)$ be   a smooth $\rho$-invariant function such that $\lambda |_{S^3}= f \lambda_0 |_{S^3}$, where $\rho$ is of the form \eqref{antis} with $\theta_1 = \theta_2 =0$.  We denote by $R= R_{\lambda}$ the Reeb vector field of $\lambda$. Let $(S^3, \lambda_E, \rho)$ be as above, where  $r_1$ and $r_2$ are chosen to be   $f_E > f.$   
  interpolate between the $\lambda_E$-dynamics and the $\lambda$-dynamics as in  Section \ref{sec:Fredth} with the data $(M, g, f) = (S^3, f_E, f)$,   $J_E \in \mathcal{J}_{\rho}(\xi, d\lambda_E|_{\xi})$, and $J \in \mathcal{J}_{\rho}(\xi, d\lambda|_{\xi})$ and study generalized  invariant  finite energy planes.  

 Due to Hofer, Wysocki, and Zehnder  \cite[Theorem 4.6]{HWZ98convex} every generalized  invariant  finite energy plane tends asymptotically to a periodic orbit of $R_E$. We fix $\hat{J} \in \mathcal{J}_{\widetilde{\rho}}(\lambda_E, J_E, \lambda, J)$ and denote by $\Theta $ the set of   generalized  invariant  finite energy $\hat{J}$-holomorphic planes, modulo  reparametrizations,  asymptotic to $P_1 = (\gamma_1, T_1)$ which has the smallest period among all periodic orbits of $R_E$.  We observe that this set is non-empty.
 Indeed,  in view of Proposition \ref{prop:ellipexit}   there exists an embedded invariant finite energy  $\widetilde{J}_E$-holomorphic plane $\widetilde{v} = (b,v)$ asymptotic to $P_1$.
 Since $+\infty$ is its unique puncture, the smooth function $b \colon \C \to \R$ is bounded from below, and hence we find a constant $C>0$ such that $b(z) + C>2$ for all $ z \in \C$.
 Define $\widetilde{u}_E = (a_E, u_E) := (b + C , v)$.
 As $a_E(z) >2$ on $\C$, we find that $\hat{J} = \widetilde{J}_E$ and hence $\widetilde{u}_E \in \Theta$. Using  \cite[Proposition 4.13]{HWZ98convex}, we find that every element in $\Theta$ is an embedding. Moreover,  if   $\widetilde{u}, \widetilde{v} \in \Theta $, then they satisfy either $\im (\widetilde{u}) = \im (\widetilde{v})$  or  $\im (\widetilde{u}) \cap \im (\widetilde{v}) = \emptyset.$  
  
  Arguing as in \cite[Propositions 4.17 and 4.18]{HWZ98convex}, we find using   Theorem \ref{thm:iFTfheotem}     a sequence $\widetilde{u}_k =(a_k, u_k) \in \Theta $ such that 
\begin{equation}\label{eq:bubblineoff}
\inf_k (\min_{z \in \C}a_k (z) ) = -\infty \quad \text{ and } \quad \sup_k ( \max_{z \in \C}\lvert \nabla \widetilde{u}_k \rvert )= \infty.
\end{equation}
Let $\Gamma = \{ z_1, \ldots, z_k, \zeta_1, \ldots, \zeta_{\ell}, \overline{\zeta_1}, \ldots, \overline{\zeta_{\ell}} \} $ be the (non-empty) set of bubbling off points of the sequence $\{ \widetilde{u}_k \}$ such that $\mathrm{Im}(z_j) =0$ and $\mathrm{Im}(\zeta_j) \neq 0$ for all $j$. Using a standard argument we find a sequence $\{ \widetilde{u}_n\}  = \{ (a_n, u_n ) \} \in \Theta$ which, up to   subsequence, converges in $C_{\mathrm{loc}}^{\infty}(\C \setminus \Gamma, \R \times S^3)$ to an embedded  generalized invariant  finite energy $\hat{J}$-holomorphic sphere $\widetilde{u}=(a,u) \colon \C \setminus \Gamma \to \R \times S^3$ such that $\Gamma$ consists only of negative punctures. In view of \eqref{eq:bubblineoff} we find symmetric periodic orbits $P_1, \ldots, P_k$ and non-symmetric periodic orbits $Q_1, \ldots, Q_{\ell}$ of the Reeb vector field $R$ such that they are the negative asymptotic limits of $\widetilde{u}$ at $z_1, \ldots, z_k$ and $\zeta_1, \ldots, \zeta_{\ell}$, respectively. Note that $\overline{Q_i} $ is the asymptotic limit at $\overline{\zeta_i}$, $i=1,\ldots, \ell$.

We carry out bubbling off analysis following \cite{HWZ95I, HWZcorr, HWZ03foliation}. 
Pick any $z_{j_0} \in \Gamma $. For each $\epsilon>0$ small enough, the limit
\[
m_{\epsilon}( z_{j_0}) = \lim_{n \to \infty} \int_{B_{\epsilon}(z_{j_0})} u_n ^* d \lambda = \int_{\p B_{\epsilon}(z_{j_0})} u ^*   \lambda 
\]
exists since $\widetilde{u}_n \to \widetilde{u}$ in $C^{\infty}_{\mathrm{loc}}( \C \setminus \Gamma , \R \times S^3  )$.
The quantity $m_{\epsilon}( z_{j_0})$ is decreasing as $\epsilon \to 0^+$, and hence we have
\[
m(z_{j_0})  := \lim_{\epsilon \to 0^+} m_{\epsilon}( z_{j_0}) \in \R,
\]
which is referred to as the \emph{mass of $\widetilde{u}$ at the puncture $z_{j_0}$}.
Let $\sigma>0$ be a positive real number less than $\min \{ \sigma_1, \sigma_2 \}$, where $\sigma_1$ is the minimum of the set of periods of all periodic   orbits of $R$ and 
\[
\sigma_2 =   \min \{ \lvert T_1- T_2 \rvert \mid T_1 \neq T_2  \leq \pi r_1^2 \text{ are periods of periodic   orbits of $R$}\}.
\]
Note that $\sigma_1, \sigma_2$ are positive since the contact form $\lambda$ is assumed to be non-degenerate. 

Choose $\epsilon >0$ such that 
\begin{equation*}\label{eq:sdf33}
 m_{\epsilon}(z_{j_0}) - m( z_{j_0}) \leq \frac{\sigma}{2}
 \end{equation*}
  and   a sequence $z_n \in \D$ defined by $a_n(z_n) = \inf_{B_{\epsilon}(  z_{j_0}  )} a_n$.  Since $z_{j_0} \in \Gamma$, we have $z_n \to z_{j_0}$ as $n \to \infty$. We write $z_n = (x_n, y_n)$, and by abuse of notation we identify $x_n$ with $(x_n, 0)$. 
Note that $x_n \to z_{j_0}$ as $n \to \infty$. We claim that    a sequence $\{ \delta_n \}$ defined by 
\begin{equation}\label{eq:choiceofzn}
\int_{B_{\delta_n}(x_n)} u_n ^* d\lambda =   m(z_{j_0})- \sigma>0
\end{equation}
satisfies that $\delta_n \to 0$ as $n \to \infty$.
Indeed, for $\epsilon>0$ sufficiently small, we find that
\[
\lim_{n \to \infty} \int_{B_{\epsilon}( x_n)} u_n ^* d\lambda = \lim_{n \to \infty} \int_{ \p B_{\epsilon}( x_n)} u_n ^* \lambda = \int_{ \p B_{\epsilon}( z_{j_0} )} u  ^* \lambda =  m_{\epsilon}(z_{j_0}).
\]
Thus, for $n$ large enough we obtain that
\[
    \int_{B_{\epsilon}(x_n)} u_n^* d \lambda \geq  m_{\epsilon}(z_{j_0}) -  {\sigma} \geq  m (z_{j_0}) - {\sigma}.
\]
  Since $x_n \to z_{j_0}$, this proves the claim. 
 We then find a sequence $R_n \to \infty$ such that 
 \begin{equation} \label{eq:choiceorRN}
  B_{R_n \delta_n}(x_n) \subset B_{\epsilon}(z_{j_0}).
  \end{equation}
Define the rescaled $\widetilde{J}$-holomorphic maps $\widetilde{v}_n \colon B_{ R_n } (0) \to \R \times S^3$ as
\begin{equation}\label{eq:limitsequence}
\widetilde{v}_n(z) = (b_n(z), v_n(z)) :=  \left( a_n (x_n + \delta_n z) - a_n(x_n + 2\delta_n), u_n (x_n + \delta_n z)\right) 
\end{equation}
which are invariant.  In view of  \eqref{eq:choiceorRN} we find that
\[
\int_{ B_{R_n}(0) \setminus \D} v_n^* d \lambda= \int_{ B_{R_n \delta_n}(x_n) \setminus B_{\delta_n}(x_n) } u_n^* d \lambda \leq \int_{ B_{\epsilon }(z_{j_0} ) \setminus B_{\delta_n}(x_n) } u_n^* d \lambda 
\]
from which together with  \eqref{eq:choiceofzn} we obtain that
\[
\limsup_{n \to \infty} \int_{ B_{R_n}(0) \setminus \D} v_n^* d \lambda  \leq m_{\epsilon}(z_{j_0}) - m( z_{j_0}) + \sigma
\]
for   $\epsilon>0$ small enough. Since $\epsilon $ was arbitrary, we have 
\[
\limsup_{n \to \infty} \int_{ B_{R_n}(0) \setminus \D} v_n^* d \lambda  \leq   \sigma,
\]
and hence    the sequence $\widetilde{v}_n$ has uniform gradient bound on $\C \setminus \D$. 
A standard argument shows that  $\widetilde{v}_n$ converges in $C^{\infty}_{\mathrm{loc}}( \C \setminus \Gamma' , \R \times S^3 )$ to an   invariant  finite energy $\widetilde{J}$-holomorphic sphere $\widetilde{v}=(b,v) \colon \C \setminus \Gamma' \to \R \times S^3  $, where $\Gamma' \subset \D$ is a finite set consisting of negative punctures such that $I(\Gamma') = \Gamma'$. Moreover, $\widetilde{v}$ tends asymptotically to $P_{j_0}  $ at the unique positive puncture $+\infty.$

If $\zeta_{i_0} \in \Gamma, $ then arguing as above we find a  finite energy $\widetilde{J}$-holomorphic sphere $\widetilde{w} =(c ,w ) \colon \C \setminus \Gamma'' \to \R \times S^3$, where $\Gamma'' \subset \D$ is a finite set consisting of negative punctures and  $\widetilde{w} $ tends asymptotically to $Q_{i_0} $ at a unique positive puncture $+\infty.$ The presence of the $\widetilde{\rho}$-symmetry  provides us  with  a finite energy $\widetilde{J}$-holomorphic sphere $\widetilde{\rho} \circ \widetilde{w}  \circ I=(c \circ I ,\rho \circ w \circ I) \colon \C \setminus I(\Gamma'') \to \R \times S^3$ which  tends asymptotically to $\overline{Q_{i_0}} $ at the unique positive puncture $+\infty.$

As in \cite[Proposition 8.1]{HWZcorr}, one can show that if   $\Gamma '\neq \emptyset$, then $0 \in \Gamma'$. Moreover, if $\Gamma'= \{ 0 \}$, then it then takes at least the amount $\sigma$ of $d \lambda$-energy away.  If $\#\Gamma' \geq2 $, then by means of \cite[Theorem 6.11]{HWZII} we find that every negative limit of $\widetilde{v}$ has smaller action than $P_{j_0}$. An analogous statement holds for $\widetilde{w}$. Thus, if either $\Gamma'$ or $\Gamma''$ is non-empty, then we  proceed in a similar way as above and obtain the   following  symmetric bubbling off tree:

\begin{enumerate}
\item The top is a single  embedded generalized    invariant finite energy sphere $\widetilde{u} \colon \C \setminus \Gamma \to \R \times S^3$. 
  It tends asymptotically to the symmetric periodic orbit $P_1 = (x_1, T_1)$ of $R_E$ at its unique positive puncture $+\infty$.
  All other punctures in $\Gamma$ are negative.
At  each $z_j \in \Gamma$ and   each $\zeta_i\in \Gamma$, the map $\widetilde{u}$  tends asymptotically to a symmetric periodic orbit $P_j$  and    a non-symmetric periodic orbit $Q_i$ of  the Reeb vector field $R$, respectively. 
The non-symmetric  periodic orbit $\overline{Q _i}  $ is the asymptotic limit of $\widetilde{u}$ at   $\overline{\zeta_i} \in \Gamma$.

\item The curves in the  middle  are    finite energy $\widetilde{J}$-holomorphic spheres $\widetilde{v}_i \colon \C \setminus \Gamma_i \to \R \times S^3$ for the $\lambda$-equation with one positive end and finitely many negative ends.

\item The bottom consists of  finite energy $\widetilde{J}$-holomorphic planes   for the $\lambda$-equation. 

\item One of the negative  asymptotic orbits of a finite energy sphere, not a plane,  coincides with the positive asymptotic orbit of some finite energy sphere in the next generation.

\item A finite energy $\widetilde{J}$-holomorphic sphere  is invariant if and only if  the positive asymptotic orbit is symmetric. 
\end{enumerate}
\medskip
\noindent
See   Figure \ref{fig:building} for an illustration of the above described symmetric bubbling off tree.
 
  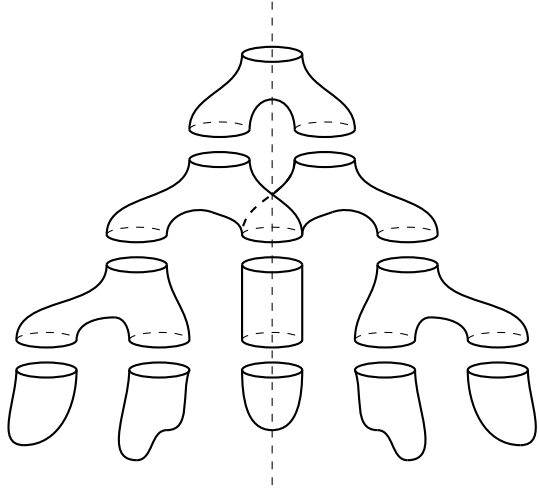
\begin{figure}[h]
\begin{center}
\begin{tikzpicture}

\draw[dashed] (0.4,3.7) to (0.4,-2.8);

 \draw[thick]   (0,3) arc (180:360:0.4cm and 0.1cm);
 \draw[thick]  (0,3) arc (180:0:0.4cm and 0.1cm);

 \draw[thick] (-0.7,2) arc (180:360:0.4cm and 0.1cm);
 \draw[dashed]  (-0.7,2) arc (180:0:0.4cm and 0.1cm);

 \draw[thick] ( 0.7,2) arc (180:360:0.4cm and 0.1cm);
 \draw[dashed]  ( 0.7,2) arc (180:0:0.4cm and 0.1cm);

 \draw [thick]  (0.7,1.6) arc (180:360:0.4cm and 0.1cm);
 \draw[thick]  (0.7,1.6) arc (180:0:0.4cm and 0.1cm);

 \draw[thick]   (-0.7,1.6) arc (180:360:0.4cm and 0.1cm);
 \draw[thick]  (-0.7,1.6) arc (180:0:0.4cm and 0.1cm);

 \draw[thick] (0,0.6) arc (180:360:0.4cm and 0.1cm);
 \draw[dashed]  ( 0, 0.6) arc (180:0:0.4cm and 0.1cm);

 \draw[thick] ( -1.8,0.6) arc (180:360:0.4cm and 0.1cm);
 \draw[dashed]  (-1.8,0.6) arc (180:0:0.4cm and 0.1cm);

 \draw[thick] (  1.8,0.6) arc (180:360:0.4cm and 0.1cm);
 \draw[dashed]  ( 1.8,0.6) arc (180:0:0.4cm and 0.1cm);

 \draw[thick] (0,0.2) arc (180:360:0.4cm and 0.1cm);
 \draw [thick] ( 0, 0.2) arc (180:0:0.4cm and 0.1cm);

 \draw[thick] ( -1.8,0.2) arc (180:360:0.4cm and 0.1cm);
 \draw  [thick] (-1.8,0.2) arc (180:0:0.4cm and 0.1cm);

 \draw[thick] (  1.8,0.2) arc (180:360:0.4cm and 0.1cm);
 \draw[thick] ( 1.8,0.2) arc (180:0:0.4cm and 0.1cm);

 \draw[thick] (0, -0.8) arc (180:360:0.4cm and 0.1cm);
 \draw[dashed]  ( 0, -0.8) arc (180:0:0.4cm and 0.1cm);

 \draw[thick] (  1.5, -0.8) arc (180:360:0.4cm and 0.1cm);
 \draw[dashed]  ( 1 .5, -0.8) arc (180:0:0.4cm and 0.1cm);

 \draw[thick] ( 3,-0.8) arc (180:360:0.4cm and 0.1cm);
 \draw[dashed,  ]  ( 3, -0.8) arc (180:0:0.4cm and 0.1cm);

 \draw[thick] (  -1.5,-0.8) arc (180:360:0.4cm and 0.1cm);
 \draw[dashed]  ( -1 .5, -0.8) arc (180:0:0.4cm and 0.1cm);

 \draw[thick] ( -3, -0.8) arc (180:360:0.4cm and 0.1cm);
 \draw[dashed]  ( -3, -0.8) arc (180:0:0.4cm and 0.1cm);

 \draw[thick] (0, -1.2) arc (180:360:0.4cm and 0.1cm);
 \draw  [thick] ( 0, -1.2) arc (180:0:0.4cm and 0.1cm);

 \draw[thick] (  1.5, -1.2) arc (180:360:0.4cm and 0.1cm);
 \draw[thick]   ( 1 .5, -1.2) arc (180:0:0.4cm and 0.1cm);

 \draw[thick] ( 3, -1.2) arc (180:360:0.4cm and 0.1cm);
 \draw[thick]  ( 3, -1.2) arc (180:0:0.4cm and 0.1cm);

 \draw[thick] (  -1.5, -1.2) arc (180:360:0.4cm and 0.1cm);
 \draw[thick]  ( -1 .5, -1.2) arc (180:0:0.4cm and 0.1cm);

 \draw[thick] ( -3, -1.2) arc (180:360:0.4cm and 0.1cm);
 \draw[thick]    ( -3, -1.2) arc (180:0:0.4cm and 0.1cm);

  \draw[thick] (0, 0.2) to  (0, -0.8);
  \draw[thick] (0.8, 0.2) to  (0.8, -0.8);
   
   \draw[thick]  (0, -1.2) to [out=270, in=180] (0.4, -2);
      \draw[thick]  (0.8, -1.2) to [out=270, in= 0] (0.4, -2);

      \draw[thick]  (1.5, -1.2) to [out=290, in=180] (1.8, -2);
            \draw[thick]  (1.8, -2) to [out= 0, in=180] (2.2, -2.4);

      \draw[thick]  (2.3, -1.2) to [out=270, in=0] (2.2, -2.4);

         \draw[thick]  (3, -1.2) to [out=270, in=180] ( 3.7 , -2.2);
      \draw[thick]  (3.8, -1.2) to [out=270, in=0] (   3.7, -2.2);

      \draw[thick]  (0, 3) to [out=270, in=90] (-0.7,  2);
      \draw[thick]  (0.8, 3) to [out=270, in=90] (1.5,  2);
         \draw[thick]  (0.1, 2) to [out=90, in=180] ( 0.4 ,  2.4);
      \draw[thick]  (0.4,  2.4) to [out=0, in=90] (   0.7, 2);

      \draw[thick]   (1.8, 0.2) to [out=270, in=90] (1.5, -0.8);
      \draw[thick]   (2.6, 0.2) to [out=270, in=90] (3.8, -0.8);
      \draw[thick]   (2.3, -0.8) to [out=90, in=180] (2.5 , -0.5 );
      \draw[thick]   (2. 5, -0.5 ) to [out= 0, in=90] (3, -0.8);

      \draw[thick, dashed ]   (0.7, 1.6) to [out=270, in=90] (0, 0.6);
       \draw[thick]   (0.7, 1.6) to [out=270, in=40] (0.4, 1.13);
      \draw[thick]   (1.5, 1.6) to [out=270, in=90] (2.6, 0.6);
      \draw[thick]   (0.8,  0.6) to [out=90, in= 200] (1.2 , 0.9 );
      \draw[thick]   (1.2, 0.9 ) to [out= 20, in=90] (1.8,  0.6);

\begin{scope}[xscale=-1, xshift=-0.8cm]

      \draw[thick]  (1.5, -1.2) to [out=290, in=180] (1.8, -2);
            \draw[thick]  (1.8, -2) to [out= 0, in=180] (2.2, -2.4);

      \draw[thick]  (2.3, -1.2) to [out=270, in=0] (2.2, -2.4);         \draw[thick]  (3, -1.2) to [out=270, in=180] ( 3.7 , -2.2);
      \draw[thick]  (3.8, -1.2) to [out=270, in=0] (   3.7, -2.2);

      \draw[thick]   (1.8, 0.2) to [out=270, in=90] (1.5, -0.8);
      \draw[thick]   (2.6, 0.2) to [out=270, in=90] (3.8, -0.8);
      \draw[thick]   (2.3, -0.8) to [out=90, in=180] (2.5 , -0.5 );
      \draw[thick]   (2. 5, -0.5 ) to [out= 0, in=90] (3, -0.8);

      \draw[thick]   (0.7, 1.6) to [out=270, in=90] (0, 0.6);
      \draw[thick]   (1.5, 1.6) to [out=270, in=90] (2.6, 0.6);
      \draw[thick]   (0.8,  0.6) to [out=90, in= 200] (1.2 , 0.9 );
      \draw[thick]   (1.2, 0.9 ) to [out= 20, in=90] (1.8,  0.6);

\end{scope}

 \end{tikzpicture}
\end{center}
\caption{An illustration of a symmetric bubbling off tree. The dashed vertical line indicates the symmetry}
\label{fig:building}
\end{figure}

\begin{proposition}\label{lem:indexbetweeb} Let $P=(x,T)$ and $Q=(y,T')$ be symmetric and non-symmetric periodic orbits appearing as asymptotic limits in the  symmetric bubbling off tree above.  Assume  that $\hat{J}$ is generic in the sense of Theorem \ref{thm:genfinite}. Then  $\mu_{\mathrm{RS}}(P )=  3/2$  and $\mu_{\mathrm{CZ}}(Q )= 2$. 
\end{proposition}
\begin{proof} We follow \cite[Lemma 5.14]{HS11} closely.   

Let $\widetilde{w} =(c,w)\colon \C \setminus \Gamma^0 \to \R \times S^3$ be the invariant finite energy  sphere in the symmetric bubbling off tree having $P$ as its positive asymptotic limit.  Set $\Gamma^0 = \{ z^0_1, \ldots, z_m^0, \zeta_1^0, \ldots, \zeta^0_n,  \overline{\zeta_1^0}, \ldots, \overline{\zeta^0_n}  \}$ with $\mathrm{Im}(z^0_j) =0$ and $\mathrm{Im}(\zeta^0_i)\neq 0$ for all $i,j$. Abbreviate by $P^0_j$ and $Q^0_i$ the associated negative limits which are symmetric and non-symmetric, respectively.   We claim that if  $\muRS(P ) \leq 1/2$, then there exists either $j \in \{ 1, \ldots, m\}$ such that $\muRS(P^0_j) \leq 1/2$ or $i \in \{ 1, \ldots, n\}$ such that $\muCZ(Q^0_i) \leq 1 $.   Arguing indirectly we assume that $\muRS(P^0_j) \geq 3/2$ and $\muCZ(Q^0_i) \geq 2$ for all $i, j$. The asymptotic behavior of the finite energy half-sphere $\widetilde{w}_I$ associated to $\widetilde{w}$ described in Theorem \ref{thm:asymptoticbahinv} implies that     $ \wind_{\infty}(\widetilde{w}_I,+ \infty) \leq 0$,  $\wind_{\infty}(\widetilde{w}_I, z^0_j) \geq  1  $, and $\wind_{\infty}(\widetilde{w}_I, \zeta^0_i) \geq 1$. Here, the asymptotic winding number at each puncture is computed via a global trivialization of the trivial bundle $\xi \to S^3$. 
 If    $\pi \circ dw$ does not vanish identically, then in view of Lemma \ref{cor:lemaKF16},  we find that 
\begin{align*}
\frac{1}{2} -\frac{1}{2} m - n  &\leq \mathrm{wind}_{\infty}( \widetilde{w}_I,+ \infty)  - \sum_{j=1}^m   \mathrm{wind}_{\infty}( \widetilde{w}_I, z^0_j ) - \sum_{ i=1}^n  \mathrm{wind}_{\infty}( \widetilde{w}_I, \zeta^0_i)    \\
& \leq - m - n
\end{align*}
  which is absurd. Therefore, $\pi \circ dw \equiv 0$. Denote by $\tau $ the minimal period of $P  $ so that $T =  \kappa  \tau $, where $\kappa \geq 1$ is the covering number. In view of  \cite[Theorem 6.11]{HWZII}   we find that   $n=0$ and $  P_j ^0 = (x , \kappa_j \tau )$ for all $j$   with $\kappa_j \in \N$ and  $\kappa   = \sum_{j=1}^m \kappa_j $.  If $\muRS(x , \tau/2) \geq 3/2$, in view of Proposition \ref{lem:indeiterrs} we find that $\muRS(P   ) \geq 3/2$,  which contradicts the hypothesis.    Therefore, $\muRS(x , \tau/2) \leq1/2$.
Again by   Proposition \ref{lem:indeiterrs}    $\muRS(P^0_j) \leq 1/2$ for all $j$. This contradiction proves the claim.

Let    $\widetilde{w}'=(c',w') \colon \C \setminus \Gamma_1 \to \R \times S^3$ be the finite energy sphere in the symmetric bubbling off tree having $Q$ as its positive asymptotic limit.  Assume that $\muCZ(Q ) \leq 1$. Using Propositions   \ref{lem:indeitercz} and  \ref{prophwzii5.6},  one can  argue as above and find that  among the negative asymptotic limits  there exists either a symmetric periodic orbit  of   $\muRS \leq 1/2$ or  a non-symmetric periodic orbit  $\muCZ  \leq 1 $.  

We repeat the above process and obtain a finite energy $\widetilde{J}$-holomorphic plane in the bottom of the bubbling off tree. If it is   invariant, then it is   asymptotic to a symmetric periodic orbit $\overline{P}$ such that $\muRS(\overline{P}) \leq 1/2 $. If it is not invariant, then it is asymptotic to a non-symmetric periodic orbit $\overline{Q}$ such that   $\muCZ(\overline{Q}) \leq 1$.  
These contradict   Theorems  \ref{HWZIIthm} and \ref{FK16thminv}.  Therefore, we have $\muRS(P) \geq 3/2$ and $\muCZ(Q) \geq 2$.

In view of $\muRS(P_1) = 3/2$, an application of   Theorem \ref{thm:genfinite} provides the inequality 
\[
1 \geq \sum_{j=1}^m \left( \muRS(P_j) - \frac{1}{2} \right) + \sum_{i=1}^n \left( \muCZ(Q_i) - 1\right).
\]
The preceding discussion   implies that  $1 \geq m+n$ and hence either  $m=0, \;n=1$ or  $m=1,\; n=0$.   In the first  case $\widetilde{u}$ has exactly two   negative asymptotic limits $Q^0,\; \overline{Q^0}$ which are non-symmetric  and have $\muCZ(Q^0) = \muCZ(\overline{Q^0}) = 2$. In the latter case the negative asymptotic limit $P^0$   is symmetric and has $\muRS (P^0) = 3/2$ so that in view of Proposition \ref{lem:indeiterrs} it is simply covered.

Assume the non-symmetric case. Let $\widetilde{v}^0=(b^0,v^0) \colon \C \setminus \Gamma^0  \to \R \times S^3$ be the finite energy $\widetilde{J}$-holomorphic sphere in the second generation of the symmetric bubbling off tree having $Q^0$ as its unique positive asymptotic limit. Denote by $Q^{0}_1, \ldots, Q^0_k$ its negative asymptotic limits. An application of Theorem \ref{thm:indexhzwinequva}  tells us that
\begin{align*}
\muCZ(Q^0) - \sum_{i=1}^k \muCZ(Q^0_i) &\geq 2 \mathrm{wind}_{\pi}(\widetilde{v}) + 4 - 2 (\# \Gamma^{\mathrm{even}}+1)- \# \Gamma^{\mathrm{odd}}\\
&\geq 2 - 2\# \Gamma .
\end{align*}
Since $\muCZ(Q^0) = 2$, we have $0 \geq  \sum_{i=1}^k\left(  \muCZ(Q^0_i) -2\right)$. Again by the preceding discussion we   have  $\muCZ(Q^0_i) = 2$ for all $i$.  If $Q^0_{i_0}$ is a symmetric periodic orbit for some $i_0 \in \{ 1,\ldots, k\}$, then together with Lemma \ref{inexd"rel} the fact that $\muRS(Q^0_{i_0}) \geq 3/2$ tells us that $\muRS(Q^0_{i_0})=3/2$.

 We now assume the symmetric case, i.e. $(m,n)=(1,0)$.  Let $\widetilde{v}^0=(b^0,v^0) \colon \C \setminus \Gamma^0  \to \R \times S^3$  be the invariant finite energy $\widetilde{J}$-holomorphic sphere   having $P^0$ as its   positive asymptotic limit. Denote by $P_1^0, \ldots, P_{\ell}^0$ and by $Q_1^0, \ldots, Q_k^0$ its  symmetric negative asymptotic limits and non-symmetric negative asymptotic limits, respectively. Recall that $\muRS(P_j^0) \geq 3/2$ and $\muCZ(Q_i^0) \geq 2$ for each $i,j$.  In view of Proposition \ref{prop:FK16index}, arguing as above, we find that
\[
\sum_{j=1}^{\ell} \left( \muRS(P_j^0) - \frac{1}{2} \right) + \sum_{i=1}^{k} \left( \muCZ(Q_i^0) - 2 \right) \leq 0
\]
from which we conclude that $\ell=0$ and $\muCZ(Q_i^0)=2$ for all $i$.

We have shown that all positive asymptotic limits in the second and the third generations of the bubbling off tree have $\muRS=3/2$ for symmetric asymptotic orbits and $\muCZ=2$ for non-symmetric asymptotic orbits. By repeating the process, we obtain the desired result. This finishes   the proof of the proposition.
\end{proof}

Assume that $\hat{J}$ is generic in the sense of Theorem \ref{thm:genfinite} so that the assertion of the preceding proposition holds.   Consider the finite energy planes in the bottom of the bubbling off tree. If there exists a non-invariant plane, i.e.\ a finite energy plane $\widetilde{u}=(a,u)$ whose asymptotic limit $Q$ is non-symmetric, then $\muCZ(Q)=2$. By Theorem \ref{HWZIIthm}, $\widetilde{u}$ is fast. By definition, $u$ is an immersion transverse to $R$, and hence $\widetilde{u}$ is an immersion. We consider the set 
\[
D=\{ (z_1, z_2) \in \C \times \C \mid z_1 \neq z_2, \; \widetilde{u}(z_1) =\widetilde{u}(z_2) \}.
\]
If $D$ has a limit point in $( \C \times \C) \setminus \Delta$, where $\Delta$ is the diagonal, then Carleman's similarity principle tells us that there exists a polynomial $p \colon \C \to \C$ such that $\mathrm{deg}(p) \geq2$ and a $\widetilde{J}$-holomorphic plane $\widetilde{v} \colon \C \to \R \times S^3$ such that $\widetilde{u} = \widetilde{v} \circ p$. This implies that $d \widetilde{u}$ has a zero, which is impossible because $\widetilde{u}$ is an immersion. Therefore, $D$ is a discrete set. Recall that $\widetilde{u}$ is the limit of a sequence $\widetilde{u}_n$ of some embedded $\widetilde{J}$-holomorphic maps. If $D \neq \emptyset$, then positivity and stability of self-intersections of immersed pseudoholomorphic curves tell us that $\widetilde{u}_n$ has a self-intersection for $n$ sufficiently large. This contradiction shows that $D =\emptyset$. In view of the asymptotic behavior of a non-degenerate finite energy plane we conclude that $\widetilde{u}$ is embedded. By   \cite[Theorem 1.7]{HWZ96unknotted} we then have $sl(Q)=-1$. In other words,
\[
Q \in \mathscr{P}_* = \emptyset,
\]
which is absurd. Therefore, every bottom finite energy plane is invariant, and its asymptotic limit $P$ is an unknotted simply covered symmetric periodic orbit of $\muRS (P) = 3/2$. In view of Corollary \ref{cor:muRSfast} it is fast. Arguing as above, we find that every bottom finite energy plane is embedded. Again by   \cite[Theorem 1.7]{HWZ96unknotted}, the asymptotic limit $P$ has  $sl(P)=-1$.   Lemma \ref{inexd"rel}  shows that  $\muCZ(P) \in \{2,3,4 \}$.  This proves the assertion of the theorem in the non-degenerate case.

\begin{remark}\label{rmk:infcat} We have assumed that $\lambda$ is non-degenerate. However, the proof only makes use of   periodic orbits of $R$ of periods less than $T_1 = \pi r_1^2$. Therefore, the assertion also holds if  we assume that all periodic orbits of periods less than $T_1$ are non-degenerate. 
\end{remark}

 In order to complete the proof, assume that $\lambda$ is dynamically convex, not necessarily non-degenerate.

 \begin{lemma}\label{lemmarealrbinson} Let $(S^3, \lambda, \rho)$ be a    real tight three-sphere. Assume that $\lambda$ is dynamically convex and fix $T>0$. Then there exists a sequence of smooth $\rho$-invariant functions $f_k \colon S^3 \to (0, \infty)$ converging to the constant function $1$ having the properties that all periodic orbits of $f_k \lambda$ of periods less than $T$ are non-degenerate and have $\muCZ \geq 3$ for every $k$.    In particular, every symmetric periodic orbit of $f_k \lambda$ of period less than $T$ has $\muRS \geq 3/2$ for all $k$. 
 \end{lemma}
 \begin{proof} 
 We follow \cite[Corollary 6.2]{HWZ98convex} closely.
  Denote by $R= R_\lambda$ the Reeb vector field of $\lambda$.  
Recall that the  symplectization $( \R \times S^3, d(e^r \lambda))$,  where $r $ denotes the coordinate on $\R$, is equipped with the exact anti-symplectic involution  $\widetilde{\rho}:= \mathrm{Id}_{\R} \times \rho$.     
We identify $S^3$ with $\{ 0 \} \times S^3$.

 Denote by $C^{\infty}_c( \R \times S^3 ; \widetilde{\rho})$ the space of $\widetilde{\rho}$-invariant smooth Hamiltonians on $\R \times S^3$ such that each  element $H$ equals   a non-zero constant outside a compact subset of $\R \times S^3$. 
 A slight modification of the proofs of the first three steps in the proof of  \cite[Theorem B.1]{CF09} shows that there is a subset $\mathcal{U} \subset C^{\infty}_c( \R \times S^3 ; \widetilde{\rho})$ of the second category which consists of  $H \in C^{\infty}_c( \R \times S^3 ; \widetilde{\rho})$  such that $0$ is a regular value of $H$ and every periodic   orbit of $X_H$ on $H^{-1}(0)$ is non-degenerate. 
 
Pick  $H \in  C^{\infty}_c( \R \times S^3 ; \widetilde{\rho})$ such that $0$ is a regular value of $H$, $H^{-1}(0)=S^3$ and  $X_H = R$ along $S^3$. Find a sequence $H_k \in  \mathcal{U}$ converging to $H$ as $k \to +\infty$. 
There exists $k_0 \in \N$ such that for all $k \geq k_0$,  the regular level  $S_k:=H_k^{-1}(0) \subset \R \times S^3$ is $C^\infty$-close to $S^3$ so that the Liouville vector field $\p_r$ is transverse to $S_k$. Note that  $X_{H_k} = R_k$ along $S_k$, where $R_k$ is the Reeb vector field of the contact form $e^r \lambda |_{S_k}$ on $S_k$.  In particular, every periodic    orbit of  $R_k$  is non-degenerate for all $k \geq k_0$.

  Let $k \geq k_0$. Find a smooth function $\psi_k \colon S^3 \to \R$ such that 
\[
S_k = \{ (\psi_k(x), x) \in \R \times S^3 \mid x \in S^3 \} 
\]
so that we have a diffeomorphism
\[
\Phi_k \colon S^3 \to S_k, \quad x \mapsto (\psi_k(x), x)
\]
satisfying
\[\rho = \Phi_k ^{-1} \circ \widetilde{\rho}|_{M_k} \circ \Phi_k .\]
Abbreviate $\lambda_k :=   \Phi_k ^*( e^r \lambda |_{S_k}) $ that is a $\rho$-anti-invariant contact form on $S^3$,  $C^\infty$-close to $\lambda$. 
It follows that  $(S^3, \lambda_k, \rho)$ is a real contact manifold. The real Gray stability theorem  \cite[Theorem 2.1.6]{Evers} now states that there is a diffeomorphism $\Psi_k \colon S^3 \to S^3$ such that it   satisfies $\Psi_k \circ \rho = \rho \circ \Psi_k$ and there is a    
sequence of $\rho$-invariant smooth functions $f_k \colon S^3 \to \R$ such that $\Psi_k ^* \lambda_k = f_k \lambda$ and $f_k$ converges to the constant function $1$ as $k \to +\infty$.  Note that   the contact form  $f_k \lambda$ is non-degenerate for all $k \geq k_0$.

Assume by contradiction that we find a periodic orbit $P_k = (x_k, T_k)$ of $f_k \lambda $ having  period  $T_k \leq T$ and $\mu_{\mathrm{CZ}}(P_k)  \leq 2$ for all $k \geq k_0$. Since the periods are uniformly bounded, 
using Arzel\`a-Ascoli theorem  
we find   a periodic orbit $P_0 = (x_0, T_0)$ of $\lambda$ such that, up to subsequence, $T_k \to T_0$ and $x_k (T_k \cdot) \to x_0 ( T_0 \cdot)$ in $C^{\infty}(S^1)$ as $k \to +\infty$. The lower semi-continuity of   the Conley-Zehnder index  implies   $\mu_{\mathrm{CZ}}(P_0) \leq 2$, contradicting the dynamical convexity of $\lambda$. The statement on the Robbin-Salamon index follows from Lemma \ref{inexd"rel}.
This finishes the proof of the lemma. 
 \end{proof}

 \begin{remark}\label{rmk:lensspaceok} 
 The only reason we work on $S^3$ in the above lemma is that  the Conley-Zehnder and Robbin-Salamon indices are well-defined on $S^3$.  An analogous statement holds for lens spaces if we replace periodic orbits by contractible periodic orbits. 
  \end{remark}

   Recall that $\lambda = f\lambda_0$ for some $f \colon S^3 \to (0, \infty)$  such that $f<f_E$.  Choose a sequence $f_k$ of positive smooth functions on $S^3$ such that $f_k$ is $\rho$-invariant and converges to the constant function $1$, and all periodic orbits of $f_k \lambda$ of periods less than $T_1$ are non-degenerate and have $\muCZ \geq 3$. In view of  Remark \ref{rmk:infcat},  the proof in the non-degenerate case tells us that there exists  a sequence $P_k = (x_k, T_k)$ such that $P_k$ is an unknotted simply covered symmetric periodic orbit of $f_k \lambda$ such that $sl(P_k)=-1$ and $\muCZ(P_k) \in \{3,4\}$ for all $k$. Moreover, we have the uniform bound $T_k \leq \pi r_1^2$. Therefore, by Arzel\`a-Ascoli theorem, up to a subsequence, $P_k$ converges in $C^{\infty}(S^1)$ to a symmetric periodic orbit $P=(x,T)$ of $\lambda$ as $k\to \infty$. Since the Conley-Zehnder index is lower semi-continuous,  we have $ \muCZ(P) \in \{3,4\}$.  By \cite[Theorem 3.4]{HWZ98convex}, $P$ is simply covered. It follows that $P$ is transversally isotopic to  all $P_k$ for $k$ sufficiently large, and hence it is unknotted and has self-linking number equal to $-1$. 
  
   We have proved the theorem for $\rho$ with $\theta_1 = \theta_2 =0$. 
  An analogous  argument also holds   with arbitrary $\theta_1, \theta_2$.    This finishes the proof of Theorem \ref{ThmA2}.

 \subsection{The existence of a $(\rho, \sigma)$-symmetric periodic orbit} We prove Theorem \ref{ThmB2} by following the argument in \cite{Sch19} closely.  Fix $p \geq 1$ and assume that  $\lambda$ is dynamically convex. We only consider the case that $\theta_2=0$, namely, we have       $\rho(z_1, z_2) = (e^{i \theta} \overline{z_1}, \overline{z_2})$ for some $\theta \in \R$. The other case can be proved in a similar manner by interchanging the roles of $r_1$ and $r_2$ in the $\lambda_E$-dynamics.

Assume first that $\overline{\lambda}$ is non-degenerate.    For sake of convenience, we may assume that $\theta_1=0$.   We also assume that $\hat{J}$ is generic in the sense of Theorem \ref{thm:genfinite}. We carry out bubbling off analysis as in the preceding subsection   and obtain a symmetric bubbling off tree. The analysis we need is done in \cite[Section 3]{HS16elliptic} and \cite[Section 4.1]{Sch19}. The top of the symmetric bubbling off tree is a single generalized invariant finite energy sphere $\widetilde{u} \colon \C \setminus \Gamma \to \R \times L(p,1)$ which is asymptotic to $\overline{P_1} ^p $ at $+\infty$.  All punctures in $\Gamma$ are negative. Moreover, every  asymptotic limit in the bubbling off tree is contractible.  In particular, symmetric asymptotic orbits are symmetrically contractible, meaning that if $x \colon \R \to L(p,1)$ is a symmetric asymptotic limit, then there exists a spanning disk $v_x \colon \D \to L(p,1)$ such that $\overline{\rho} \circ v_x = v_x \circ I$. Therefore, every negative  asymptotic limit has well-defined Conley-Zehnder index or Robbin-Salamon index.   In what follows, $\muCZ(P)$ or $\muRS(P)$ denotes the Conley-Zehnder index or the Robbin-Salamon index of $P$ computed using a  trivialization that extends to a (symmetric) spanning disk, respectively.

  \begin{lemma} \label{lemindxc} We have   $\Gamma = \{ z_* \}$ with $\mathrm{Im}(z_*) =0$. The associated asymptotic limit   $P $   has $\muRS(P)=3/2$.

  \end{lemma}
  \begin{proof}  
  
      Different from the previous subsection,   $\widetilde{u}$ is not necessarily embedded because the positive asymptotic limit $\overline{P_1} ^p $ is not simply covered.  
  We first assume that $\widetilde{u}$ is not somewhere injective. Then we find a somewhere injective generalized invariant finite energy sphere $\widetilde{v} \colon \C \setminus \Gamma' \to \R \times L(p,1)$ and a polynomial $q \colon \C \to \C$ with $\mathrm{deg}(q) \geq2$ such that $\widetilde{u} = \widetilde{v} \circ q$,   $q^{-1}( \Gamma')=\Gamma$, and $I \circ q = q \circ I$, implying that  $q$ has real coefficients.
    Moreover, $\widetilde{v}$ is asymptotic to $\overline{P_1}^k $ at $+\infty$, where $k=p/\mathrm{deg}(q)$, and the punctures in $\Gamma' \neq \emptyset$  are negative. Note that $\overline{P_1}^k $ is non-contractible. Set $\Gamma'= \{ z_1' , \ldots, z_m', \zeta_1', \ldots, \zeta_{\ell}', \overline{\zeta_1'}, \ldots, \overline{\zeta_{\ell}'}\}$ with $\mathrm{Im}(z_j')=0$ and $\mathrm{Im}(\zeta_j')\neq0$ for all $j$. Abbreviate by $P_{z_j '}$ and $Q_{\zeta_i '}$ the associated asymptotic limits of $\widetilde{v}$.  We observe that at least one negative asymptotic limit is non-contractible. Indeed, if every negative asymptotic limit is contractible, then this forces $\overline{P_1}^k $ to be contractible.

 We claim that  the non-symmetric asymptotic limits   are contractible. 
 To this end,  let $w \in \Gamma'$ be a negative puncture of $\widetilde{v}$ whose corresponding asymptotic limit, denoted by $P_w$, is non-contractible. Write $q^{-1}(w) = \{ z_1, \ldots, z_r\} \subset \Gamma$. 
We then find $d \neq 0$ and $n_1, \ldots, n_r \in \N$ such that
 \[
 q( z) = w + d ( z-z_1)^{n_1} \cdots (z-z_r)^{n_r}.
 \]
If $q'(z_j) \neq 0$ for some $j$, then    $q$ is a local bi-holomorphism near $z_j$, implying that $P_w$ is contractible, a contradiction. Therefore, $q'(z_j)=0$, and hence $n_j \geq 2$ for all $j$. 
Since all coefficients of $q$ are real, this implies that $w$ is real, proving the claim.

It was shown in  \cite[Lemma 4.7]{Sch19}  that if there exists more than one non-contractible negative asymptotic limit, then the polynomial $q$ has at least ${\mathrm{deg}}(q)$ critical points. This contradiction shows that there is a unique non-contractible negative asymptotic limit that is symmetric.  Without loss of generality, we may assume that $P_{z_1'}  = (x',T')$ is the non-contractible negative asymptotic limit. Its $\mathrm{deg}(q)$-fold cover $(P_{z_1'} )^{\mathrm{deg}(q)}$  is (symmetrically) contractible.

    Recall that $P_1=(x,T)$ is   a non-degenerate  symmetric periodic orbit on $(S^3, \lambda_E, \rho)$ such that $x(t)=(e^{2 \pi i t/T} ,0)$ and $T=\pi r_1^2$. Under the projection $\pi_{p,1} \colon S^3 \to L(p,1)$, we obtain a simply covered $\overline{\rho}$-symmetric periodic orbit  $\overline{P_1 }=(x, \overline{T})$ on $L(p,1)$, where $\overline{T}=T/p$.   
 Consider the $\rho$-symmetric unitary trivialization
  \[
 \widetilde{\Psi} \colon S^1 \times \C \to x_T^*\xi, \quad ( e^{2\pi  it}, \alpha) \mapsto  ( x_T(t), e^{ -2\pi  i(p-1)t } \alpha).
  \]
 Denoting by $\widetilde{\Psi}(t, \cdot):=\widetilde{\Psi}(e^{  2 \pi i t}, \cdot)$,   $t \in \R/\Z$, we find that 
 \[
 \widetilde{\Psi}\left( t + \frac{1}{p} \right) = e^{ 2 \pi i  /p} \widetilde{\Psi}(t), \quad \forall t \in \R/\Z.
 \]
 Therefore, it descends to a $\overline{\rho}$-symmetric unitary trivialization $\overline{\Psi}$ of $x_{\overline{T}}^* \overline{\xi}$.       It was shown in \cite[Lemma 4.8]{Sch19} that $\muCZ^{ \overline{\Psi}^k}( \overline{P_1}^k      )= 2k-1$.  By Lemma \ref{inexd"rel}, we have 
 
 \noindent
  \begin{align*}
 &\lvert \muCZ^{ \overline{\Psi}^k}(  \overline{P_1}^k   )   -2 \muRS^{ \overline{\Psi}^k}(      \overline{P_1}^k   )  \rvert \leq 1 \\
 \Rightarrow \quad  &   \muCZ^{ \overline{\Psi}^k}(  \overline{P_1}^k    )  -1 \leq 2 \muRS^{ \overline{\Psi}^k}(    \overline{P_1}^k   )  \leq  \muCZ^{ \overline{\Psi}^k}(  \overline{P_1}^k    )  +1 \\
  \Rightarrow \quad  &  k-1 \leq   \muRS^{ \overline{\Psi}^k}(  \overline{P_1}^k     )  \leq  k  .
  \end{align*}
  Since $\muRS^{ \overline{\Psi}^k}(   \overline{P_1}^k   )  \in \Z+ \frac{1}{2}$,  we conclude that
 \[
 \muRS^{ \overline{\Psi}^k}( \overline{P_1}^k  ) = k-\frac{1}{2}.  
 \]
  Let $v_j'$ be a symmetric spanning disk of $P_{z_j '}$ for $j=2,\ldots, m$ and   $w_i'$  be a  capping disk  of  $Q_{\zeta_i '}$ for    $i=1, \ldots, \ell$. Choose a symmetric unitary trivialization $\widehat{\Psi}$ over a cylinder $\widetilde{v} \# v_2' \# \cdots \#v_m' \# w_1'\# \cdots \# w_{\ell}' \# \overline{w_1'} \# \cdots \# \overline{w_{\ell}'}$ such that $\widehat{\Psi} = \overline{\Psi}^k$ along $\overline{P_1}^k $. Denote by $\Psi'$ the induced symmetric unitary trivialization   of $(x'_{T'})^* \overline{\xi}$.  
  The second half of the proof  of   \cite[Lemma 4.9]{Sch19} carries over to our case, and hence we obtain  $ \muCZ^{\Psi'} (  {P_{z_1'}}   ) \geq 2k-1.$ 
  Arguing as before, we find that 
 \[
 \muRS^{\Psi'}(   {P_{z_1'}}   ) \geq k-\frac{1}{2}.  
 \]
 Since $\widetilde{u}$ is the limit of a sequence of embedded pseudoholomorphic curves, by positivity and stability of self-intersections of immersed pseudoholomorphic curves, $\widetilde{v}$ is immersed. An application of Theorem \ref{thm:genfinite} provides us with the following inequality
  \[
 \muRS^{\overline{\Psi}^k} (   \overline{P_1}^k  ) -   \muRS^{\Psi'}(   {P_{z_1'}} ) -  \sum_{j=2}^m \left( \muRS( P_j'  ) - \frac{1}{2} \right) -  \sum_{i=1}^{\ell} \left( \muCZ( Q_i'  ) -1 \right) \geq 0.
  \]
Together with the dynamical convexity of $\lambda$  this implies that   $(m,\ell)=(1,0)$, i.e.\  $\Gamma'=\{z_1'\}$,  and $\muRS^{\Psi'}(   {P_{z_1'}}   ) =k-\frac{1}{2}$.

As has been shown by Schneider in \cite[Section 4]{Sch19}, we have $q(z) = z_1'+a(z-z_*)^{\mathrm{deg}(q)}$ for some non-zero $a \in \C$ and hence $\Gamma=\{z_*\}$. This implies that $\widetilde{u}$ has a unique negative asymptotic limit $P:= (P_{z_1'} )^{\mathrm{deg}(q)}$ at $z_*$ which is   contractible.

     There exists an embedded symmetric spanning disk $u_0 \colon \D \to S^3$ of $x_T$ such that $u_0^* \xi$ has a $\rho$-symmetric unitary trivialization $\widehat{\Phi}$ that induces a $\rho$-symmetric unitary trivialization $\Phi \colon S^1 \times \C \to x_T^* \xi$ given by
   \[
   \Phi( e^{2\pi it}, \alpha) = ( x_T(t), e^{ -2\pi  it } \alpha), \quad e^{2\pi   it} \in S^1,\; \alpha \in \C.
   \]
    For a proof, see \cite[Lemma 1.6]{HWZ95I}. 
Using the symmetric unitary trivialization $\widehat{\Phi} \# \widehat{\Psi}$   over the symmetric spanning disk $u_0 \# \widetilde{v} \# v_2' \# \cdots \# v_m' \# w_1'\# \cdots \# w_{\ell}' \# \overline{w_1'} \# \cdots \# \overline{w_{\ell}'}$ of $P$, Schneider showed  that $\muCZ(P)=3$. Arguing as above we finally have
\[
\muRS(P)=\frac{3}{2}.
\]

We now assume that $\widetilde{u}$ is somewhere injective. Since it is immersed, we are able to apply Theorem \ref{thm:genfinite}  and obtain that
\[
\muRS( \overline{P_1}^p) - \frac{1}{2} \geq \sum_{z \in  \Gamma_{\p}} \left( \muRS(P_z) - \frac{1}{2} \right)+ \sum_{\zeta \in  \Gamma_{\mathrm{int}}\cap H } \left( \muCZ(Q_z) - 1\right).
\]
Since $\muRS( \overline{P_1}^p) =3/2$ and $\lambda$ is dynamically convex, we conclude that  $\Gamma= \{z\}$ and $\muRS(P_z)=3/2$. This finishes the proof of the lemma.  \end{proof}

  Let $\widetilde{v}  =(b, v) \colon \C \setminus \Gamma^0 \to \R \times L(p,1)$ be the invariant finite energy $\widetilde{J}$-holomorphic sphere in the second generation of the symmetric bubbling off tree. The preceding lemma tells us that the positive asymptotic limit $P$  has   $\muRS(P)=3/2$.    Assume that $\Gamma^0 \neq \emptyset$. We first suppose that $\pi \circ dv \not\equiv0$. Let $\mathfrak{T}$ be a symmetric unitary trivialization of $\widetilde{v}^*\overline{\xi}$ that extends to (symmetric) spanning disks of the negative asymptotic limits. Since $\muRS(P) =3/2$, $\muRS(P_z)\geq3/2$ for all $ z \in \Gamma^0_{\p}$, and $\muCZ(Q_{\zeta}) \geq3$ for all $ \zeta \in \Gamma^0_{\mathrm{int}}\cap H$, we have $\wind_{\infty}^{\mathfrak{T}}(\widetilde{v}_I, +\infty) \leq 1/2$, $\wind_{\infty}^{\mathfrak{T}}(\widetilde{v}_I, z) \geq1$ for all $ z \in \Gamma^0_{\p}$, and $\wind_{\infty}^{\mathfrak{T}}(\widetilde{v}_I, \zeta) \geq 2$ for all $ \zeta \in \Gamma^0_{\mathrm{int}}\cap H$.  An application of Lemma \ref{cor:lemaKF16} tells us that 
  \begin{align*}
  0& \leq \wind_{\pi}(\widetilde{v}) \\
  &= \wind_{\infty}^{\mathfrak{T}}(\widetilde{v}_I, +\infty) - \sum_{ z \in \Gamma^0_{\p}} \wind_{\infty}^{\mathfrak{T}}(\widetilde{v}_I, z) - \sum_{ \zeta \in \Gamma^0_{\mathrm{int}}\cap H} \wind_{\infty}^{\mathfrak{T}}(\widetilde{v}_I, \zeta) \\
  & \quad -1 + \frac{\# \Gamma^0_{\p}}{2} + \# ( \Gamma^0_{\mathrm{int}}\cap H) \\
  &\leq -\frac{1}{2} -  \frac{\# \Gamma^0_{\p}}{2} - \# ( \Gamma^0_{\mathrm{int}}\cap H) 
  \end{align*} 
  which is absurd. Therefore, $\pi \circ dv  \equiv0$. Recall that   $\# \Gamma^0 \geq 2$.  By  \cite[Theorem 6.11]{HWZII} we have $\mathrm{im} \widetilde{v} = \R \times x(\R)$, where $P=(x,T)$, and $\Gamma^0 = \Gamma^0_{\p}$, i.e.\ every negative asymptotic limit is symmetric.  Denote by $P_{\min}=(x, T_{\min})$ the underlying simply covered symmetric periodic orbit of $P$ with $T=m_{\infty} T_{\min}$ for some $m_{\infty} \in \N$. For each $z \in \Gamma^0$, the associated asymptotic limit $P_z =(x, T_z)$ satisfies that $T_z = m_z T_{\min}$, where $m_z \in \N$ and $\sum_{z \in \Gamma^0} m_z = m_{\infty}$. Since $\# \Gamma^0 \geq 2$, we have $m_{\infty} \geq 2$.

  \begin{lemma} \label{lem:covnoncontd} The symmetric periodic orbit $P_{\min}$  is non-contractible, and the covering number $m_{\infty}$ is minimal in the sense that if $P_{\min}^k$ is contractible, then $k \geq m_{\infty}$.  Moreover, $m_{\infty}$ divides $p$. 
  \end{lemma}
  \begin{proof}    Let $Q=(y,T)$ be a periodic orbit on $(M, \lambda)$ and choose a unitary trivialization $\mathfrak{T}$ of $y^*\xi$. Associated to $Q$ and $\mathfrak{T}$ is  the \emph{transverse rotation number}   $\rho^{\mathfrak{T}}(Q)$ with respect to $\mathfrak{T}$. This depends only on the homotopy class of $\mathfrak{T}$. We do not give the definition, but recollect the following important properties. 
   \begin{itemize}
   \item Given $n \in N$, we write $Q^n=(y, nT)$. Then 
   \begin{equation}\label{eq:rotnume}
   \rho^{   \mathfrak{T}^n}(Q^n)=n    \rho^{\mathfrak{T}}(Q).
   \end{equation}
   \item    If $Q^n$ is non-degenerate, then
   \begin{equation}\label{eq:rotnumind}
   \muCZ^{\mathfrak{T}^n} (Q^n)  =\begin{cases} 2\rho^{  \mathfrak{T}^n}(Q^n)   & \text{ if $Q^n$ is hyperbolic,}  \\ 2 \left \lfloor \rho^{ \mathfrak{T}^n}(Q^n)  \right \rfloor  +1 & \text{ if $Q^n$ is elliptic.}  \end{cases}
   \end{equation}
   In particular, if $Q^n$ is elliptic, $\rho^{\mathfrak{T}^n}(Q^n)$ is not an integer.
   \end{itemize}
  \noindent 
  If $Q$ is contractible, we simply write as $\rho(Q)$, where the transverse rotation number is computed using any unitary trivialization over any spanning disk.

  Assume by contradiction that $P_{\min}$ is contractible. Since $\muRS(P)=3/2$, we have $\muCZ(P) \in \{3 , 4 \}$.  If $\muCZ(P)=3$, then by \eqref{eq:rotnumind} we have $1 <   \rho(P) <2$ and hence 
  \[
  \frac{1}{m_{\infty}} < \rho(P_{\min}) < \frac{2}{m_{\infty}} \leq 1
  \]
  by \eqref{eq:rotnume}. If   $\muCZ(P)=4$, then   $\rho(P)=2$ and hence $\rho(P_{\min})=2/m_{\infty} \leq 1$.  In any case we have $\rho(P_{\min}) \leq 1$.  On the other hand, since $P_{\min}$ is contractible, the dynamical convexity of the contact form implies  $\muRS(P_{\min})\geq 3/2$ so that $1 <   \rho(P_{\min}) $. This  contradiction shows that $P_{\min}$ is non-contractible.
  
  The remaining assertions are proved in \cite[Lemma 4.10]{Sch19}. This completes the proof of the lemma.   
  \end{proof}

  \noindent
  Since $P_z =(x,  m_z T_{\min})$ is contractible, and $ 1 \leq m_z \leq m_{\infty}$, by the preceding lemma we have $m_z = m_{\infty}$ for all $ z \in \Gamma^0$ from which we obtain that  $\# \Gamma^0=1$. This contradicts to the fact that $\# \Gamma^0 \geq2$.

  We have shown that the symmetric bubbling off tree consists of the two pseudoholomorphic curves: The top is a   generalized invariant finite energy sphere $\widetilde{u} \colon \C \setminus \{z_* \} \to \R \times L(p,1)$ such that $\overline{P_1}^p$ is the asymptotic limit at $+\infty$, and the negative asymptotic limit is a symmetric periodic orbit $P$ of $\muRS(P)=3/2$. The bottom is an invariant finite energy plane $\widetilde{v} =(b,v) \colon \C \to \R \times L(p,1)$ asymptotic to $P$. Since $\muRS(P)=3/2$, by Corollary \ref{cor:muRSfast} we find that $\widetilde{v}$ is fast, and hence   $v$ is an immersion transverse to the Reeb vector field. Moreover, it is somewhere injective. Otherwise, there exists a somewhere injective invariant finite energy plane $\widetilde{w}$ and a polynomial $q \colon \C \to \C$ with $\mathrm{deg}(q) \geq 2$ such that $\widetilde{v}= \widetilde{w}\circ q$. This forces $d \widetilde{v}$ to have a zero, which is impossible since $\widetilde{v}$ is an immersion. As in the previous subsection, using positivity and stability of self-intersections of immersed pseudoholomorphic curves, we find that $\widetilde{v}$ is embedded.

  We claim that there exists $R\gg1$ such that $v$ is an embedding on $\C \setminus B_R(0)$. Lemma \ref{lem:covnoncontd} implies that either $P$ is simply covered or there exists a simply covered non-contractible  symmetric periodic orbit $P_{\min} = (x, T_{\min})$ such that $P=(P_{\min})^m$ for some $m \geq2$ that divides $p$. In the first case, the claim is straightforward. We now assume the last case. We shall show that $+\infty$ is a relative prime puncture, meaning that given a symmetric unitary trivialization $\Phi$ of $x_{T_{\min}}^*\overline{\xi}$, the winding number $\wind_{\infty}^{\Phi^m}(\widetilde{v})$ and $m$ are relative prime integers. This is independent of the choice of a trivialization. Once this is proved, the claim follows from \cite[Lemma 2.11]{HS16elliptic}. It remains to compute $\wind_{\infty}^{\Phi^m}(\widetilde{v})$ with respect to a suitable trivialization $\Phi$.

    Let $\Psi$ be   a symmetric unitary trivialization of $\overline{\xi}$ along periodic orbit $Q$ such that $[Q]=1 \in \Z_p \cong \pi_1( L(p,1))$,  constructed in the proof of Lemma \ref{lemindxc}. Since $P=(P_{\min})^m$ is contractible, it follows that $\Psi^p$ defines a symmetric unitary trivialization of $x_T^* \overline{\xi}$. Recall that by $\muRS(P) $ we mean  the Robbin-Salamon index of $P$ computed using a symmetric unitary trivialization $\mathfrak{T}$  that extends over a symmetric spanning disk. In view of Lemma \ref{lem:difftrirel}  we find that
    \[
    \muRS^{\Psi^p}(P) = \muRS(P)+ \wind (\Psi^p, \mathfrak{T}). 
    \]
     It was shown  in \cite[Section 4]{Sch19} that $\wind (\Psi^p, \mathfrak{T})=p-2$. Since $\muRS(P)=3/2$, it follows that
     \[
        \muRS^{\Psi^p}(P) = p - \frac{1}{2}.
        \]
  Using the analytic definition of the Robbin-Salamon index     and the fastness of $\widetilde{v}$ we obtain that
  \[
  \wind_{\infty}^{\Psi^p}(\widetilde{v}) = 2  \wind_{\infty}^{\Psi^p}(\widetilde{v}_I) = p-1.
  \]
            Since $P_{\min}$ is non-contractible, we have $[P_{\min}]:=r \neq 0  \in \Z_p \cong   \pi_1 ( L(p,1)) $ from which we find that $\Psi^r$ provides a symmetric unitary trivialization of $\overline{\xi}$ along $P_{\min}$. We set $\Phi=\Psi^r$ and compute $\wind_{\infty}^{\Phi^m}(\widetilde{v})$. We find that 
\[
\wind_{\infty}^{\Phi^m}(\widetilde{v}) = \wind_{\infty}^{\Psi^{rm}}(\widetilde{v}) = \wind_{\infty}^{\Psi^{np}}(\widetilde{v}) = n \wind_{\infty}^{\Psi^{p}}(\widetilde{v}) =n(p-1),
\]    
    where the fact that $rm=np$ for some $n \in \N$ in the second identity follows from the contractibility of $P=(P_{\min})^m$. Since $m$ divides $p$, the two integers $m$ and $p-1$ are relative prime. It is easy to see that $n$ and $m$ are also relatively prime. Therefore, $\wind_{\infty}^{\Phi^m}(\widetilde{v})$ and $m$ are relatively prime from which the claim is proved.

     In view of the previous claim, the argument given in \cite[Section 3.1.8]{HS16elliptic} carries over to our case, and hence we have $v(\C) \cap x(\R) = \emptyset$. Since $\widetilde{v}$ is embedded, we are able to apply \cite[Theorem 2.3]{HWZII} to obtain that $v \colon \C \to L(p,1) \setminus x(\R)$ is injective. Therefore, we are able to compute the rational self-linking number of $P$ with respect to $v$, and by \cite[Lemma 3.10]{HLS15} we have
     \[
     sl(P) = - \frac{1}{m},
     \]
     where $m$ is the covering number of $P$ that divides $p$. An application of \cite[Lemma 7.3]{HLS15} tells us that $P$ binds a $\overline{\lambda}$-adapted $m$-rational open book decomposition with disk-like pages. In view of \cite[Theorem 1.3]{HLS15} we find that $L (p,1)$ is diffeomorphic to $L(m,k)$ for some $k$. The classification of lens spaces implies that $m=p$.  We have found that there exists a $p$-unknotted $\overline{\rho}$-symmetric periodic orbit  of $sl=-1/p$ whose $p$-th iterate  has $\muRS =3/2$, provided that $\overline{\lambda}$ is non-degenerate.

    As in the previous subsection, an argument using Remark \ref{rmk:infcat}, Lemma \ref{lemmarealrbinson} and Remark \ref{rmk:lensspaceok} shows that there exists a sequence of non-degenerate contact forms $\overline{\lambda}_n$ on $L(p,1)$ converging to $\lambda$ such that $\overline{\rho}^*\overline{\lambda}_n  = -\overline{\lambda}_n $ and  for all $n$ large enough, $\overline{\lambda}_n$ admits a $p$-unknotted $\overline{\rho}$-symmetric periodic orbit $P_n$ of $\muRS(P_n^p) = 3/2$ and $sl(P_n)=-1/p$.  Since their periods are uniformly bounded by $\pi r_1 ^2/p$, by applying the Arzel\`a-Ascoli theorem we find that up to  subsequence, $P_n$ converges in $C^{\infty}$ to a $\overline{\rho}$-symmetric periodic orbit $P$ of $\lambda$. Since $P_n^p$ is contractible for every $n$, the $p$-th iterate $P^n$ is also contractible. The lower semi-continuity of the Conley-Zehnder index implies that $\muCZ(P^p) \in \{3,4\}$. It is not hard to see that $P$ is simply covered so that it is transversely isotopic to $P_n$ for $n$ sufficiently large. Therefore, $P$ is $p$-unknotted and has $sl(P)=-1/p$.      
   This finishes the proof of Theorem \ref{ThmB2}.

  \bibliographystyle{abbrv}
\bibliography{mybibfile}

\end{document}